\documentclass[a4paper,11pt]{amsart}
\usepackage{amssymb}
\usepackage{amsmath}
\usepackage{amsthm}
\usepackage{verbatim}
\usepackage{url}
\usepackage[all]{xy}

\theoremstyle{plain}

\newtheorem{theorem}{Theorem}[section]

\newtheorem{proposition}[theorem]{Proposition}
\newtheorem{lemma}[theorem]{Lemma}
\newtheorem{corollary}[theorem]{Corollary}
\newtheorem{question}[theorem]{Question}

\theoremstyle{definition}

\newtheorem{remark}[theorem]{Remark}

\newtheorem{example}[theorem]{Example}
\newtheorem{construction}[theorem]{Construction}
\newtheorem{notation}[theorem]{Notation}

\title[Real Galois cohomology]
{Real homogenous spaces, Galois cohomology,\\ and Reeder puzzles}

\author{Mikhail Borovoi and Zachi Evenor}

\address{Borovoi: Raymond and Beverly Sackler School of Mathematical Sciences,
Tel Aviv University, 6997801 Tel Aviv, Israel}
\email{borovoi@post.tau.ac.il}

\address{Evenor: Raymond and Beverly Sackler School of Mathematical Sciences,
Tel Aviv University, 6997801 Tel Aviv, Israel}
\email{zachi.evenor@gmail.com}

\thanks{Borovoi was partially supported by the Hermann Minkowski Center for Geometry. 
Both Borovoi and Evenor were partially supported by the Israel Science Foundation (grant No. 870/16)}



\begin{document}


\newcommand{\R}{{\mathbb{R}}}
\newcommand{\Z}{{\mathbb{Z}}}
\newcommand{\C}{{\mathbb{C}}}
\newcommand{\Q}{{\mathbb{Q}}}
\newcommand{\F}{{\mathbb{F}}}
\renewcommand{\H}{{\mathbb{H}}}

\renewcommand{\AA}{{\mathbf{A}}}
\newcommand{\BB}{{\mathbf{B}}}
\newcommand{\CC}{{\mathbf{C}}}
\newcommand{\DD}{{\mathbf{D}}}
\newcommand{\EE}{{\mathbf{E}}}
\newcommand{\FF}{{\mathbf{F}}}
\newcommand{\GG}{{\mathbf{G}}}

\newcommand{\T}{{\mathcal{M}}}
\newcommand{\M}{{\mathcal{M}}}

\newcommand{\X}{{{\sf X}}}

\newcommand{\G}{{\mathbb{G}}}

\newcommand{\ii}{{\mathrm{i}}}

\renewcommand{\aa}{{\boldsymbol{a}}}
\newcommand{\bb}{{\boldsymbol{b}}}
\newcommand{\cc}{{\boldsymbol{c}}}
\newcommand{\dd}{{\boldsymbol{d}}}
\newcommand{\ttt}{{\boldsymbol{t}}}
\newcommand{\qq}{{\boldsymbol{q}}}
\newcommand{\pp}{{\boldsymbol{p}}}

\newcommand{\zbar}{{\bar{z}}}
\newcommand{\gbar}{{\bar{g}}}
\newcommand{\nbar}{{\bar{n}}}
\newcommand{\xbar}{{\bar{x}}}

\newcommand{\Htil}{{\widetilde{H}}}
\newcommand{\Ttil}{{\widetilde{T}}}
\newcommand{\ttil}{{\tilde t}}

\newcommand{\vk}{{\varkappa}}
\newcommand{\ve}{\varepsilon}
\newcommand{\vev}{\varepsilon^\vee}

\renewcommand{\gg}{{\mathfrak{g}}}
\newcommand{\tfrak}{{\mathfrak{t}}}

\newcommand{\into}{\hookrightarrow}
\newcommand{\isoto}{\overset{\sim}{\to}}
\newcommand{\onto}{\twoheadrightarrow}
\newcommand{\labelto}[1]{\xrightarrow{\makebox[1.5em]{\scriptsize ${#1}$}}}

\newcommand{\GL}{{\bf{GL}}}
\newcommand{\SL}{{\bf{SL}}}
\newcommand{\Sp}{{\bf{Sp}}}
\newcommand{\PSp}{{\bf{PSp}}}
\newcommand{\SO}{{{\bf SO}}}
\newcommand{\PSO}{{\bf{PSO}}}
\newcommand{\Spin}{{{\bf Spin}}}
\newcommand{\HSpin}{{\bf{HSpin}}}
\newcommand{\PGL}{{\bf{PGL}}}
\newcommand{\SU}{{\bf SU}}
\newcommand{\PSU}{{\bf PSU}}

\newcommand{\Hom}{{\rm Hom}}
\newcommand{\Inn}{{\rm Inn}}
\newcommand{\Aut}{{\rm Aut}}
\newcommand{\Lie}{{\rm Lie\,}}
\newcommand{\Gal}{{\rm Gal}}
\newcommand{\coker}{{\rm coker\,}}
\newcommand{\tors}{{\rm tors}}
\newcommand{\Ext}{{\rm Ext}}
\newcommand{\Stab}{{\rm Stab}}
\newcommand{\res}{{\rm res}}
\newcommand{\Ad}{{\rm Ad}}
\newcommand{\Cl}{{\rm Cl}}
\newcommand{\ad}{{\rm ad}}
\newcommand{\im}{{\rm im\,}}
\newcommand{\id}{{{\rm id}}}
\newcommand{\diag}{{\rm diag}}

\newcommand{\Orbset}{\operatorname{Orb}}         
\newcommand{\Or}{{\rm Orb}}                                
\newcommand{\Orbs}[1]{ \# \mathrm{Orb}( {#1} ) } 

\newcommand{\boxone}{ *+[F]{1} }
\newcommand{\ccc}{{  \lower0.20ex\hbox{{\text{\Large$\circ$}}}}}
\newcommand{\Lbul}{{\lower0.20ex\hbox{\text{\Large$\bullet$}}}}
\newcommand{\bc}[1]{{\overset{#1}{\ccc}}}
\newcommand{\bcu}[1]{{\underset{#1}{\ccc}}}
\newcommand{\bcb}[1]{{\overset{#1}{\Lbul}}}
\newcommand{\bcbu}[1]{{\underset{#1}{\Lbul}}}
\newcommand{\sxymatrix}[1]{ \xymatrix@1@R=5pt@C=9pt{#1} }
\newcommand{\mxymatrix}[1]{ \xymatrix@1@R=0pt@C=9pt{#1} }
\newcommand{\rline}{ \ar@{-}[r] }
\newcommand{\lline}{ \ar@{-}[l] }
\newcommand{\dline}{ \ar@{-}[d] }
\newcommand{\upline}{ \ar@{-}[u] }
\newcommand{\arr}{\ar@{=>}[r]}

\newcommand{\RRR}{{\Rightarrow}}
\newcommand{\gr}{\! > \!}
\newcommand{\less}{ \!\! < \!\! }
\newcommand{\LLL}{\!\Leftarrow\!}
\newcommand{\boe}{{\sxymatrix{\boxone}}}
\renewcommand{\ll}{\!-\!}

\newcommand{\half}{{\tfrac{1}{2}}}

\newcommand{\Ss}{\sideset{}{'}\sum_{k\succeq i}}
\newcommand{\Ssd}{\sideset{}{'}\sum_{k\succeq i,\,k\in D^\tau}}

\newcommand{\DeltaK}{{K}}

\newcommand{\nl}{0}

\newcommand{\alphabar}{\alpha'}
\newcommand{\Dtil}{{\widetilde{D}}}
\newcommand{\Dtilbar}{{\widetilde{D'}}}
\newcommand{\mbar}{m'}
\newcommand{\ibar}{{i'}}
\newcommand{\jbar}{{j'}}
\newcommand{\Pibar}{{\Pi'}}
\newcommand{\nprime}{{n'}}

 \newcommand{\SAut}{\mathrm{SAut}}

\newcommand{\hs}{\kern 0.8pt}
\newcommand{\hh}{\kern 1.0pt}
\newcommand{\kk}{\kern 2.0pt}

\newcommand{\emm}{\bfseries}



\begin{abstract}
Let $G$ be a simply connected absolutely simple algebraic group defined over the field of real numbers $\R$.
Let $H$ be a simply connected semisimple $\R$-subgroup of $G$.
We consider the homogeneous space $X=G/H$.
We ask: how many connected components has $X(\R)$?

We give a method of answering this question.
Our method is based on our solutions of generalized Reeder puzzles.
\end{abstract}

\maketitle



\setcounter{section}{-1}

\section{Introduction}
In this paper by a semisimple or reductive group
we always mean a {\em connected} semisimple or reductive group, respectively.
Let $G$ be a  {\em simply connected} absolutely simple algebraic group over the field of real numbers $\R$.
Let $H\subset G$ be a  {\em simply connected} semisimple $\R$-subgroup.
We consider the homogeneous space $X=G/H$, which is an algebraic variety over $\R$.
The topological space $X(\R)$ of  $\R$-points of $X$ need not be connected.
We ask
\begin{question}\label{q:1}
How many connected components has $X(\R)$?
\end{question}

The group of $\R$-points $G(\R)$ acts on the left on $X(\R)$, and we consider the orbits of this action.
By Lemma \ref{lem:orbits-components} below,  the set of connected components of $X(\R)$
is the set of orbits $G(\R)\backslash X(\R)$ of $G(\R)$ in $X(\R)$.
On the other hand, there is a canonical bijection
\begin{equation}\label{e:Serre}
 G(\R)\backslash X(\R)\isoto \ker\left[H^1(\R,H)\to H^1(\R,G)\right],
 \end{equation}
see Serre \cite[Section I.5.4, Corollary 1 of Proposition 36]{Serre},
where $H^1(\R,G)$ denotes the first (nonabelian) Galois cohomology of $G$.
We see that Question \ref{q:1} is equivalent to the following question:
\begin{question}\label{q:2}
What is the cardinality  of the finite set $\ker\left[H^1(\R,H)\to H^1(\R,G)\right]$?
\end{question}
In this paper we give a method of answering Question \ref{q:2} and hence, Question \ref{q:1}.
Namely, we give an explicit description of the Galois cohomology sets
$H^1(\R,G)$ for all simply connected $\R$-groups $G$,
permitting one to compute the kernel in Question \ref{q:2}.
We describe  $H^1(\R,G)$ using our  solutions of generalized Reeder puzzles.

Let $G$ be a {\em simply connected,} absolutely simple, simply-laced, compact $\R$-group.
Let $T\subset G$ be a maximal torus.
Let $\Pi$ be a basis of the root system $R(G_\C, T_\C)$.
Let $D=D(G_\C,T_\C,\Pi)$ be the Dynkin diagram of $G$ with the set of vertices numbered by $1,2,\dots,n$,
then $D$ is simply-laced, i.e. it has no multiple edges.
By a {\em labeling of} $D$ we mean a family $\aa=(a_i)_{i=1,\dots,n}$, where $a_i\in\Z/2\Z$.
In other words, at any vertex $i$ we write a label $a_i=0,1$.
We consider the set $L(D)$ of the labelings of $D$, it is an $n$-dimensional vector space over the field $\Z/2\Z$.

For any vertex $i$ we define the {\em move} $\M_i$ applied to a labeling $\aa$:
if the vertex $i$ has an {\em odd} number of neighbors with 1,
$\M_i$ {\em changes} $a_i$ (from 0 to 1 or from 1 to 0), otherwise it does nothing.
Clearly $\M_i(\M_i(\aa))=\aa$.
We say that two labelings $\aa,\aa'$ are {\em equivalent} if we can pass from $\aa$  to $\aa'$ by a finite sequence of moves.
This is indeed an equivalence relation on $L(D)$.
We denote the corresponding set of equivalence classes by $\Cl(D)$.
It is the set of orbits of the Weyl group $W$ acting on $L(D)$,
and we  denote it also by $\Or(D)$ in Sections \ref{sec:An}--\ref{sec:G2} below.
The set $\Cl(D)$ has a neutral element $[\nl]$, the class of the zero labeling $\nl$.
To solve the puzzle means to describe the set of equivalence classes $\Cl(D)$ and to describe each equivalence class.

This is original Reeder's puzzle \cite{Reeder}, except that Reeder formulated his puzzle for any simply-laced graph,
not necessarily a simply-laced Dynkin diagram.
For  a compact, simply connected, simply-laced group $G$ with Dynkin diagram $D$,
the pointed set $\Cl(D)$ is in a bijection with $H^1(\R,G)$  .
In order to deal with non-simply-laced and noncompact groups, we generalize the puzzle.

We permit non-simply-laced Dynkin diagrams.
Then, when counting the number of neighbors with 1 of a given vertex $i$,
we do not count the {\em shorter} neighbors of $i$ connected with $i$ by a {\em double} edge.
In other words, ``the long roots don't see the short roots''.

 We consider also colored Dynkin diagrams, which correspond to  non-compact inner forms of compact groups.
A {\em coloring} of a Dynkin diagram $D$ is a family
\[\ttt=(t_i)_{i=1,\dots,n},\quad t_i\in\Z/2\Z.\]
If $t_i=1$, we color vertex $i$ in black, otherwise we leave it white.
When vertex $i$ is white, the move $\M_i$ acts as above.
When $i$ is black, the move $\M_i$ changes $a_i$ if $i$ has  an {\em even} number of neighbors with 1,
and does nothing otherwise.
We write sometimes $L(D,\ttt)$ for the set $L(D)$ with this Reeder puzzle.
We denote the corresponding set of equivalence classes by $\Cl(D,\ttt)$.
If $\ttt=\boldsymbol{0}=(0,\dots, 0)$, we have $\Cl(D,\boldsymbol{0})=\Cl(D)$.
Note that if $D$ has a {\em black} vertex $i$, then the move $\M_i$
takes the zero labeling $\nl$ to a nonzero labeling
and hence does not respect the group structure in $L(D)$.

We recall the definition of $H^1(\R,G)$, cf. \cite[Section III.4.5]{Serre}.
Let $G$ be a linear algebraic group over $\R$.
We denote by $G(\C)$ the set of $\C$-points of $G$.
The first Galois cohomology set $H^1(\R,G)$ is, by definition,
$Z^1(\R,G)/\sim$, where the set of $1$-cocycles $Z^1(\R,G)$ is defined by $Z^1(\R,G)=\{z\in G(\C)\ |\ z\zbar=1\}$,
and two $1$-cocycles $z,z'\in Z^1(\R,G)$ are cohomologous  (we write $z\sim z'$) if
$z'=gz\bar g^{-1}$ for some $g\in G(\C)$.
Here the bar denotes the complex conjugation in $G(\C)$; note that $G(\R)=\{g\in G(\C)\ |\ \bar g=g\}$.
By definition, the neutral element $[1]\in H^1(\R,G)$ is the class of the neutral cocycle $1\in Z^1(\R,G)\subset G(\C)$.

We write $G(\R)_2$ for the set of elements $g\in G(\R)$ such that $g^2=1$.
Then $g\bar{g}=g^2=1$, hence $G(\R)_2\subset Z^1(\R,G)$, and so we obtain a canonical map $G(\R)_2\to H^1(\R,G)$.

Let  $G$ be a simply connected absolutely simple $\R$-group.
For simplicity we assume in the Introduction
that $G$ is an {\em inner form of a compact group.}
Then $G$ has a compact maximal torus $T$, see Subsection \ref{subsec:t}.
Choose a basis $\Pi$ of the root system $R=R(G_\C,T_\C)$.
We obtain an isomorphism
\[ \gamma\colon L(D)\isoto T(\R)_2\subset Z^1(\R, G),\]
see formula \eqref{e:gamma} in Section \ref{sec:compact}.
This isomorphism induces a map $L(D)\to H^1(\R,G)$, which is surjective
by a result of Kottwitz \cite[Lemma 10.2]{Kottwitz}.
By Theorem \ref{thm:inner}  the fibers of this map are equivalence classes
of the  Reeder puzzle for $(D,\ttt)$ for a certain coloring $\ttt$ of $D$.
In other words, we obtain a bijection
\[ \Cl(D,\ttt)\isoto H^1(\R,G). \]
Moreover, for a suitable basis $\Pi$ the coloring $\ttt$
can be obtained from a Kac diagram by removing vertex 0, see Section \ref{ss:Kac}.
In Sections \ref{sec:An}--\ref{sec:G2} we solve  case by case the generalized  Reeder puzzles for all such pairs $(D,\ttt)$.
Namely, in each case we give a set $\Xi$ of representatives for all  equivalence classes in $\Cl(D,\ttt)$
and describe explicitly the equivalence class $[\nl]\subset L(D,\ttt)$ of the zero labeling $\nl$.

Now let $H$ be a simply connected semisimple $\R$-subgroup of a simply connected  absolutely simple $\R$-group $G$.
For simplicity, we assume in the Introduction that $H$ is absolutely simple
and that $G$ and $H$ are inner forms of compact groups.
Then $G$ contains a compact maximal torus $T_G$ and $H$ contains a compact maximal torus $T_H$.
We may and shall assume that $T_H\subset T_G$.
We denote by $(D_H,\ttt_H)$ and $(D_G,\ttt_G)$ the corresponding Reeder puzzles.
For a good choice of bases $\Pi_H$ and $\Pi_G$ we obtain colorings $\ttt_H$ and $\ttt_G$ coming from Kac diagrams
(in particular, not more than one vertex of each of $D_H$ and $D_G$ is black).

We describe our method of answering Question \ref{q:2}.
The embedding $T_H\into T_G$ induces an embedding $T_H(\R)_2\into T_G(\R)_2\,$.
Thus we obtain an injective homomorphism
\[\iota\colon L(D_H)\to L(D_G),\]
which can be computed explicitly.
Using results of Sections \ref{sec:An}--\ref{sec:G2} for the group $H$,
we construct a finite subset $\Xi\subset L(D_H,\ttt_H)$ containing exactly one representative of each equivalence class
for the corresponding Reeder puzzle.
For any $\xi\in\Xi\subset L(D_H,\ttt_H)$ we compute $\iota(\xi)\in L(D_G,\ttt_G)$.
Using  results of Sections \ref{sec:An}--\ref{sec:G2} for the group $G$,
namely, the description of the equivalence class $[\nl]$ of $\nl$ in $L(G,\ttt_G)$\,,
we can check whether $\iota(\xi)\in L(D_G,\ttt_G)$ lies in  $[\nl]$ or not.
We obtain a subset $\Xi_0$ of $\Xi$ consisting of all $\xi\in\Xi$ such that $\iota(\xi)\in[\nl]$.
One can show (see Section \ref{sec:examples}) that $\Xi_0$
is in a bijection with $\ker\left[ H^1(\R,H)\to H^1(\R, G)\right]$ and therefore,
the cardinality of $\Xi_0$ answers Questions \ref{q:2} and  \ref{q:1}.

Note that in order to answer Question \ref{q:2},
we compute in Sections \ref{sec:An}--\ref{sec:G2} the sets  $\Cl(D,\ttt)$
for inner forms of a compact group and certain sets $\Cl(D,\tau,\ttt)$ for outer forms.
Since each of these sets is in a bijection with the corresponding Galois cohomology set,
we in particular compute the cardinalities of the Galois cohomology sets $H^1(\R,H)$
for all absolutely simple simply connected $\R$-groups $H$.
These cardinalities have been known.
The Galois cohomology of classical groups and adjoint groups is well known.
S.~Garibaldi and N.~Semenov \cite[Example 5.1]{GS} computed $H^1(\R,H)$ for a certain nonsplit
simply connected group $H$ of type $\EE_7$.
B.~Conrad \cite[Proof of Lemma 4.9]{Conrad} computed  $H^1(\R,H)$
for the split simply connected groups $H$ of types of $\EE_6$ and $\EE_7$.
The cardinalities  of the Galois cohomology sets for ``most'' of  simple $\R$-groups, in particular,
for all absolutely simple simply connected $\R$-groups,
were recently computed  by  J.~Adams \cite{A} by a method different from ours.
Our results agree with the previous results, in particular with the tables of Adams \cite{A}.
Later, after the first version of the present paper appeared in arXiv,
Borovoi and Timashev  \cite{BoT} proposed a combinatorial method based on the notion of a Kac diagram,
permitting  one to compute easily  the cardinality of $H^1(\R,H)$
when $H$ is an inner form of any compact semisimple $\R$-group, not necessarily simply connected.
However, it seems that neither of these alternative approaches
permits one to answer Question \ref{q:1} about $(G/H)(\R)$,
except for the case when $H^1(\R,G)=1$ (which happens only when $G=\SL(n)$ or $G=\Sp(2n,\R)$).

The rest of the paper is structured as follows.
In Section \ref{sec:1} we recall results of \cite{Bo}.
In Sections \ref{sec:compact} and \ref{sec:inner}  we compute the moves $\T_i$
in the case when $G$ is compact and when it is a noncompact inner form of a compact group, respectively.
In particular, in Section  \ref{sec:inner}  we prove Theorem \ref{thm:inner} describing the pointed set $H^1(\R,G)$
for an {\em inner} form $G$ of a compact simply connected simple  group in terms of the corresponding generalized Reeder puzzle.
In Section \ref{sec:outer} we prove Theorem \ref{cor:Theorem-3-Bo},
which reduces computing the Galois cohomology of an {\em outer} form of a compact, simply connected, simple  $\R$-group to
computing Galois cohomology of an {\em inner} form of another compact group.
In Section \ref{sec:Kac} we describe the generalized Reeder puzzle for $G$
in terms of the Kac diagram of $G$ from \cite[Table 7]{OV}.
In Sections \ref{sec:An}--\ref{sec:G2}
we solve the generalized Reeder puzzles
for all isomorphism classes of simply connected absolutely simple $\R$-groups $G$.
We state the assertions necessary for our calculations, but omit straightforward  proofs for brevity.
In the last Section \ref{sec:examples} we describe our method of answering Questions \ref{q:2} and \ref{q:1}
for all simply connected  $H$ (not necessarily simple),
and we give examples of calculations using  results of Sections \ref{sec:An}--\ref{sec:G2}.


\section{Galois cohomology of reductive real groups}
\label{sec:1}

In this section we state briefly the necessary results of \cite{Bo}.
For details see \cite{Bo} or \cite{Borovoi-arXiv}.

Let $G$ be a  reductive group over $\R$.
Let $T$ be a {\em fundamental torus} of $G$, i.e.,
a maximal torus of $G$ (defined over $\R$) containing a maximal compact torus $T_0$ of $G$.
Then $T$ is the centralizer of $T_0$ in $G$; see \cite[Section 7]{Borovoi-arXiv}.
Let $T_1$ be the largest {\em split} subtorus of $T$.
We write $T(\R)_2$ for the group of elements of $T(\R)$ of order dividing 2.
\begin{lemma}[{\cite[Lemma 1.1]{Bo}, see also \cite[Lemma 3(a)]{Borovoi-arXiv}}]
\label{lem:Bo88}
The map $T(\R)_2\to H^1(\R,T)$ induces a canonical isomorphism $T(\R)_2/T_1(\R)_2\isoto H^1(\R,T)$.
\end{lemma}

Set $N_0=\mathcal{N}_G(T_0)$, $W_0=N_0/T$.
We have $W_0(\C)=W_0(\R)$; see \cite[Section 7]{Borovoi-arXiv}.
We define a left action of the group $W_0(\R)$ on the set $H^1(\R,T)$.
Let $w\in W_0(\R)$ be represented by $n\in N_0(\C)$ and let $\xi\in H^1(\R,T)$, $\xi=[z]$,
where $z\in Z^1(\R,T)$ is a cocycle and $[z]$ denotes the cohomology class of  $z$.
We set
\begin{equation}\label{eq:Bo-action}
w* \xi:= [nz\nbar^{-1}]=[nzn^{-1}\cdot n\nbar^{-1}],
\end{equation}
where the bar denotes the complex conjugation in $G(\C)$.
This is a well-defined action; see \cite[Construction 8]{Borovoi-arXiv}.
(Note that in general the action $*$ does not respect the group structure on $H^1(\R,T)$.\,)
It is easy to see that the images of $\xi$ and $w*\xi$ in $H^1(\R,G)$ coincide.
Therefore, we obtain a canonical map
$$W_0(\R)\backslash H^1(\R,T)\to H^1(\R,G).$$

\begin{proposition}[{\cite[Theorem 1]{Bo}, see also \cite[Theorem 9]{Borovoi-arXiv}}]
\label{prop:Bo88}
The map
$$W_0(\R)\backslash H^1(\R,T)\to H^1(\R,G)$$
induced by the map $H^1(\R,T)\to H^1(\R,G)$ is a bijection.
\end{proposition}


\section{Weyl action for compact groups}
\label{sec:compact}
{\em We change our notation.}
In Sections \ref{sec:compact} -- \ref{sec:Kac},
$G$ is a {\em  simply connected, simple, {\bfseries compact}} (i.e., anisotropic) linear algebraic group over $\R$.

Let $T$ be a maximal torus of $G$. Let $X^*=\X^*(T_\C):=\Hom(T_\C, \G_{m,\C})$ denote the character group of $T_\C$,
where $\G_{m,\C}$ is the multiplicative group over $\C$.
Let $R=R(G_\C,T_\C)\subset X^*$ denote the root system of $G_\C$ with respect to $T_\C$,
then we have a root decomposition
\[ \Lie G_\C=\Lie T_\C\oplus\bigoplus_{\beta\in R}\gg_\beta\,.\]
Let
$\Pi\subset R$ be a basis  of $R$ (a system of simple roots).
Note that $\Pi$ does not have to be a basis of $X^*$.
Write $\Pi=\{\alpha_1,\dots,\alpha_n\}$, then a simple root $\alpha_i$ is a homomorphism $\alpha_i\colon T_\C\to \G_{m,\C}$.
Let $R_+\subset R$ denote the set of positive roots with respect to the basis $\Pi$,
and let $B\subset G_\C$ denote the corresponding Borel subgroup of $G_\C$ containing $T_\C$, then
\[ \Lie B=\Lie T_\C\oplus\bigoplus_{\beta\in R_+}\gg_\beta\,.\]
Let $D=D(G_\C,T_\C,\Pi)=D(G_\C,T_\C,B)$ denote the Dynkin diagram of $G_\C$ with respect to $T_\C$ and $\Pi$,
then the set of vertices of $D$ is $\Pi$.
Let  $W=W(G,T)=N/T$ denote the Weyl group, where $N$ is the normalizer of $T$ in $G$.
By abuse of notation we write $W$ also for the group of points $W(\R)=W(\C)$.

Let $X_*=\X_*(T_\C):=\Hom(\G_{m,\C}, T_\C)$ denote the cocharacter group of $T$.
There is a canonical pairing
$$
\langle\ ,\,\rangle\colon X^*\times X_*\to\Z,\quad (\chi,x)\mapsto \langle \chi,x\rangle\in \Z,\quad \chi\in X^*,\ x\in X_*
$$
defined by
$$
\chi\circ x\,=\ ( z\mapsto z^{\langle \chi,x\rangle}\, ) \colon \ \G_{m,\C}\to \G_{m,\C}\,.
$$
We have a canonical basis $\Pi^\vee= \{\alpha_1^\vee,\dots,\alpha_n^\vee\}$ of the dual root system $R^\vee$,
where the simple coroot $\alpha_i^\vee\colon \G_{m,\C}\to T_\C$ is the  coroot corresponding to the simple root $\alpha_i$;
see \cite[Sections 7.4 and 7.5]{Springer}. Note that $\langle \alpha_i , \alpha_i^\vee \rangle = 2$.
Since $G$ is {\em simply connected}, $\Pi^\vee$ is a basis of $X_*$
(this is one of the definitions of a simply connected semisimple algebraic group,
cf.~\cite[Section 2.15]{SpringerAMS}).

\begin{lemma}[well-known]\label{lem:repr}
Let $G$, $T,\ N$, and $W$ be as above (in particular, $G$ is {\em compact}).
Then for any $w\in W(\R)=W(\C)$ there exists a representative $n\in N(\R)$ (and not just in $N(\C)$).
\end{lemma}

\begin{proof}
The group $W$ is generated by the reflections $r_1,\dots,r_n$, hence, it suffices to find such an $n$ for a reflection $w=r_i$.
This reduces to the case where $G=\SU_2$ and $T$ is the diagonal torus, when we can take
\[n=\begin{pmatrix}0 &1\\-1 & 0\end{pmatrix}.\]
\end{proof}

Since $G$ is {\em compact,} by Borel and Serre \cite[Theorem 6.8, Example (a)]{Borel-Serre}, see also Serre \cite[III.4.5, Example (a)]{Serre}
(or by Lemma \ref{lem:Bo88} and Proposition \ref{prop:Bo88} above, where $T$ is compact, $N_0=N$, and $W_0=W$)
we have a bijection $W\backslash T(\R)_2\isoto H^1(\R,G)$.
Here $W$ acts on $T(\R)_2$ in the standard way.
Namely, since $G$ is compact,  by Lemma \ref{lem:repr} we can choose a representative $n$ of $w\in W$ in $N(\R)$,
and for $a\in T(\R)_2$ we set
\[ w*a=na\nbar^{-1}=nan^{-1}.\]
Therefore, we are interested in the standard action of $W$ on $T(\R)_2$.
We identify  $X_*/2X_*$ with $T(\R)_2$ by $x+ 2 X_*\mapsto x(-1)\in T(\R)_2$ for $x\in X_*$.
The canonical $\Z$-basis $\alpha_1^\vee,\dots,\alpha_n^\vee$ of $X_*$  gives a $\Z/2\Z$-basis of $X_*/2X_*$,
which we shall again write as $\alpha_1^\vee,\dots,\alpha_n^\vee$.

By a {\em labeling} of the Dynkin diagram $D$ we mean a vector $\aa=(a_i)_{i=1,\dots,n}$, where $a_i\in \Z/2\Z$, i.e., $a_i=0,1$.
In other words, at each vertex $i$ of $D$ we write a label $a_i\in \Z/2\Z$.
We denote the abelian group of labelings of $D$ by $L(D)$.
We have a canonical isomorphism
\begin{equation}\label{e:gamma}
\gamma\colon L(D)\isoto T(\R)_2\subset Z^1(\R,G),\quad \aa\mapsto a= \prod_{i=1}^n \left(\alpha_i^\vee(-1)\right)^{a_i}.
\end{equation}
By abuse of notation we  denote by $\gamma$ both the isomorphism $\gamma\colon L(D)\isoto T(\R)_2$
and the embedding $\gamma\colon L(D)\isoto T(\R)_2\into Z^1(\R,G)$.
Thus with $\aa\in L(D)$ we associate $a=\gamma(\aa)\in T(\R)_2\subset Z^1(\R,G)$.
We also associate with $\aa$ the element $\sum_k a_k\alpha^\vee_k\in X_*/2X_*$.

We wish to compute the orbits of  $W$ in $T(\R)_2$  with respect to the standard left action.
The Weyl group $W$ is generated by the reflections $r_i=r_{\alpha_i}$.
We define the {\em moves}
$\T_i\colon L(D)\to L(D)$ on the set of labelings $L(D)$ by  $\T_i\hh \aa=\aa'$, where
\begin{equation}\label{eq:r-i-action}
r_i\left(\prod_{j=1}^n \left(\alpha_j^\vee(-1)\right)^{a_j}\right)=
\prod_{j=1}^n\left( \alpha_j^\vee(-1)\right)^{a'_j} \
\text{ i.e., }\
r_i \left(\sum_{j=1}^n a_j\,\alpha_j^\vee\right)=\sum_{j=1}^n a'_j\,\alpha_j^\vee\in X_*/2X_*.
\end{equation}
Note that if $\aa'=\T_i\hh \aa$, then $\aa=\T_i\hh\aa'$, because $r_i^2=1$.
We say that two labelings $\aa,\aa'\in L(D)$ are {\em equivalent}
if we can relate them by a series of moves.
The set of orbits  of  $W$ in $T(\R)_2$ is in a canonical bijection
with the set of equivalence classes of labelings $\aa\in L(D)$
of the Dynkin diagram $D$ of $(G_\C,T_\C,\Pi)$ with respect to the moves.

The following Lemma \ref{prop:non-twisted} says that the moves
defined in this sections are indeed the moves of the Reeder puzzle on $D$.

\begin{lemma}\label{prop:non-twisted}
Let $G$ be a simply connected, simple, {\em compact} $\R$-group of absolute rank $n$,
and $D$ its Dynkin diagram, as above.
Define the moves $\T_i\colon  L(D)\to L(D)$ by \eqref{eq:r-i-action}.
Then we have $a'_j=a_j$ for $j\neq i$, and $a'_i$ is given by
\begin{equation}\label{non-twisted-simply-laced}
 a'_i = a_i + \Ss a_k
\end{equation}
(addition in $\Z/2\Z$), where  $\Ss$ means
that the sum is taken over all the {\em neighbors} $k\neq i$ of $i$
except for the vertices $k$ connected to $i$ by a double edge
such that the root $\alpha_k$ is {\em shorter} than $\alpha_i$.
\end{lemma}

\begin{proof}
A reflection $r_i$ acts on $X_*$ by
\begin{equation}\label{eq:Springer-reflection}
r_i(y)=y-\langle\alpha_i,y\rangle \alpha_i^\vee,
\end{equation}
cf. \cite[Section 7.4.1]{Springer}.
If $y=\sum_k a_k\alpha_k^\vee\in X_*$, then
$$
r_i (y) =y-\sum_k  a_k\langle\alpha_i,\alpha_k^\vee\rangle\alpha_i^\vee,
$$
and the same formula holds if $y=\sum_k a_k\alpha_k^\vee\in X_*/2X_*$.
If we write $r_i (y)=\sum_k a'_k\alpha_k^\vee$, then clearly $a'_j=a_j$ for $j\neq i$, and
\begin{equation}\label{eq:action-sum}
a'_i=a_i+\sum_k (-a_k)\langle \alpha_i,\alpha_k^\vee\rangle,
\end{equation}
so we need only to compute (in $\Z/2\Z$) the sum in \eqref{eq:action-sum}.

We may assume that our root system $R$ is a root system in a Euclidean space $V$.
Then
$$
\langle \alpha_i,\alpha_k^\vee\rangle=\frac{2(\alpha_i,\alpha_k)}{(\alpha_k,\alpha_k)},
$$
where $(\alpha_i,\alpha_k)$ is the scalar product in $V$.
If $k=i$, then $\langle\alpha_i,\alpha_k^\vee\rangle=\langle\alpha_i,\alpha_i^\vee\rangle=2\equiv 0 \pmod{2}$.
If two different vertices $i$ and $k$ are not connected by an edge, then $\langle\alpha_i,\alpha_k^\vee\rangle=0$.
Thus the sum in \eqref{eq:action-sum} is taken over vertices $k$ different from $i$ that are connected to $i$ by an edge.
Now we consider cases.
If vertices $i$ and $k$ are connected by a single edge, then $\langle\alpha_i,\alpha_k^\vee\rangle=-1$
\cite[VI.1.3, possibility (3)\,]{Bourbaki},
hence vertex $k$ gives $a_k$ to the sum in \eqref{eq:action-sum}.
If they are connected by a triple edge, then either $\langle\alpha_i,\alpha_k^\vee\rangle=-1$
or $\langle\alpha_i,\alpha_k^\vee\rangle=-3\equiv -1\pmod{2}$
\cite[VI.1.3, possibility (7)\,]{Bourbaki},
and again vertex $k$ gives $a_k$ to  the sum.
If they are connected by a double edge and the root $\alpha_k$ is  {\em longer} than $\alpha_i$,
then $\langle\alpha_i,\alpha_k^\vee\rangle=-1$ \cite[VI.1.3, possibility (5)\,]{Bourbaki},
and again vertex $k$ gives $a_k$ to  the sum.
However, if the vertices $i$ and $k$  are connected by a double edge
and the root $\alpha_k$ is {\em shorter} than $\alpha_i$\emph{},
then  $\langle\alpha_i,\alpha_k^\vee\rangle=-2\equiv 0\pmod{2}$ \cite[VI.1.3, possibility (5)\,]{Bourbaki},
hence vertex $k$ gives nothing to the sum in \eqref{eq:action-sum}.
We conclude that formula \eqref{eq:action-sum} can be written as \eqref{non-twisted-simply-laced}.
\end{proof}

\begin{corollary}\label{cor:non-twisted}
If $G$ is as in Lemma \ref{prop:non-twisted}, in particular $G$ is compact, then the map \eqref{e:gamma}
induces a bijection $\Cl(D)\isoto H^1(\R,G)$, where the moves $\T_i$ act on $L(D)$ by
formula \eqref{non-twisted-simply-laced}.
\end{corollary}


\section{Weyl action for inner forms}
\label{sec:inner}

In this section $G$, $T$, $R$, $\Pi$, $D$, and $W$ are as in Section \ref{sec:compact},
in particular $G$ is a  simply connected, simple,
{\em compact}  linear algebraic group over $\R$.

\subsection{The $t$-twisted action}
\label{subsec:t}
Write $G^\ad=G/Z_G,\ T^\ad=T/Z_G$, where $Z_G$ denotes the center of $G$.
Then $T^\ad$ is a maximal torus in the adjoint group $G^\ad$.
Consider an inner twisted form (inner twist)  $_z G$ of $G$,
 where  $z\in Z^1(\R,G^\ad)$.
It is well known that $z$ is cohomologous to some  $t\in T^\ad(\R)_2$
(see e.g., \cite[III.4.5, Example (a)]{Serre}). {\em We fix such  an element} $t$.
Then $_z G\simeq \hs_t G$.
We have $_t G(\C)=G(\C)$,
but the complex conjugation in $_t G(\C)$ is given by
$$
g\mapsto {}^*\gbar=\Inn(t)(\gbar) .
$$
This means that if we lift $t\in T^\ad(\R)_2$ to some $\ttil\in T(\C)$,
then the complex  conjugation in $_t G(\C)$ is given by
$$
^*\gbar=\ttil\,\gbar\, \ttil^{-1}.
$$

Since $\ttil\in T(\C)$, we have $_t T=T$, hence $_t T$ is a compact  maximal  torus in $_t G$,
hence it is a fundamental torus of $_t G$.
Thus any inner form of a compact semisimple $\R$-group has a compact maximal torus.
Let $T_0$ of Section \ref{sec:1} be the maximal compact subtorus of $_t T$, then clearly $T_0=\hs_t T=T$.
Let $W_0:=W_0(\hs_t G,\hs_t T)$ be the group $W_0$ of Section \ref{sec:1},
then $W_0=W(G,T)= W$, because $W_0$ was defined in terms of $T_0$.

We consider the $t$-twisted action of $W_0=W$ given by formula \eqref{eq:Bo-action}
on $H^1(\R,\hs_t T)=H^1(\R,T)=T(\R)_2$.
Let $w\in W(\R)=W(\C)$, $w=nT$, where $n\in N(\R)$.
Then
\[\nbar=n,\quad {}^* \nbar=\ttil\nbar\ttil^{-1}=\ttil n\ttil^{-1}. \]
For $a\in T(\R)_2=T(\C)_2$ the $t$-twisted action of $w$ is given by
\begin{equation}\label{eq:ect-0-gen}
w* a= n\, a\, {}^* \nbar^{-1}
=n\, a \, \ttil\, \nbar^{-1} \ttil^{-1}
=n\, a \, \ttil\, n^{-1} \ttil^{-1}
= n a n^{-1}\cdot  n\ttil n^{-1} \ttil ^{-1}.
\end{equation}
In particular, let $r_j\in W(\R)=W(\C)$ be the reflection corresponding to a simple root $\alpha_j$.
Write $r_j=n_j T$ for some $n_j\in N(\R)$.
For $a\in T(\R)_2$ the $t$-twisted action of $r_j$ is given by
\begin{equation}\label{eq:ect-0}
r_j* a= n_j\, a\, {}^* \nbar_j^{-1}
=n_j\, a \, \ttil\, n_j^{-1} \ttil^{-1}= n_j a n_j^{-1}\cdot  n_j\ttil n_j^{-1} \ttil ^{-1}.
\end{equation}
Note that
\begin{equation}\label{eq:twisting}
r_j* a=r_j(a) \cdot n_j\ttil n_j^{-1} \ttil^{-1},
\end{equation}
where $r_j(a)=n_j a n_j^{-1}$.
In particular, we have  $r_j * 1= n_j\ttil n_j^{-1} \ttil^{-1}$, so in general $r_j* 1 \neq 1$ and therefore,
the $t$-twisted action does not preserve the group structure in $T(\R)_2$.

Define
\begin{equation}\label{eq:t-bold}
\ttt=(t_i)\in(\Z/2\Z)^n, \quad\text{where}\quad (-1)^{t_i}=\alpha_i(t).
\end{equation}
We regard $\ttt$ as a {\em coloring} of the diagram $D$.
We color a vertex $i$ in black if $t_i=1$, and leave $i$ uncolored (i.e., white) if $t_i=0$.
Denote by $_\ttt D:=(D,\ttt)$ the Dynkin diagram $D=D(G_\C,T_C,\Pi)$
together with the coloring $\ttt$ .
The notation $_\ttt D$ suggests that we regard $_\ttt D=(D,\ttt) $ as an (inner) twist of $D$ by $\ttt$ .

We compute the moves corresponding to the $t$-twisted action.
For each vertex $i$ of $D$,
we define the move $\T_i$ by
$\T_i\hh \aa=\aa'$, where
\begin{equation*}\label{eq:r-i-action-twisted}
 r_i*\left(\prod_{j=1}^n \left(\alpha_j^\vee(-1)\right)^{a_j}\right)=
\prod_{j=1}^n \left(\alpha_j^\vee(-1)\right)^{a'_j}\
\text{ i.e., }\
r_i *\left(\sum_{j=1}^n a_j\,\alpha_j^\vee\right)=\sum_{j=1}^n a'_j\,\alpha_j^\vee\in X_*/2X_*.
\end{equation*}

\begin{lemma}\label{prop:twisted}
For the $t$-twisted action of $W$ and the move $\T_i$ just defined,
we have, as in Lemma \ref{prop:non-twisted}, $a'_j=a_j$ for $j\neq i$,
while in  formula \eqref{non-twisted-simply-laced}
the term  $t_i\in\Z/2\Z$ defined by $(-1)^{t_i}=\alpha_i(t)$ must be added.
Thus we have
\begin{equation}\label{twisted-simply-laced}
 a'_i = a_i+t_i +  \Ss a_k \ ,
\end{equation}
where the meaning of $\Ss$ is the same as in formula \eqref{non-twisted-simply-laced}.
\end{lemma}

\begin{proof}
By  \eqref{e:gamma}, \eqref{eq:twisting} and Lemma \ref{prop:non-twisted} it suffices to show
that $n_j\ttil n_j^{-1} \ttil^{-1}=\left(\alpha_j^\vee(-1)\right)^{t_j}$.
We are indebted to Dmitry A. Timashev for the idea of the following proof.

Consider the $\C$-torus $T_\C$. As above, we
write $X_*$ for $\X_*(T_\C)=\Hom(\G_{m,\C},T_\C)$.
We have a canonical isomorphism of abelian complex Lie groups
$$
X_*\underset{\Z}{\otimes} \C^\times\isoto T(\C),\quad x\otimes u\mapsto x(u),\quad x\in X_*,\ u\in \C^\times=\G_{m,\C}(\C).
$$
Thus we obtain an isomorphism of abelian complex Lie algebras (vector spaces over $\C$)
$$
X_* \underset{\Z}{\otimes} \C \isoto \Lie T_\C,\quad x\otimes v\mapsto dx(v),\quad x\in X_*,\ v\in\C,\
dx:=d_1 x\,\colon \C=\Lie\G_{m,\C}\to \Lie T_\C\,.
$$
In particular, we obtain a canonical embedding
\begin{equation}\label{eq:embedding-X*}
X_*\into X_* \underset{\Z}{\otimes} \C \isoto \Lie T_\C\qquad x\mapsto x\otimes 1\mapsto dx(1).
\end{equation}
Now it is an easy exercise to deduce from \eqref{eq:Springer-reflection} and \eqref{eq:embedding-X*}
that for $1\le j\le n$ and for any $y\in\Lie T_\C$ we have
\begin{equation}\label{eq:Springer-reflection-Lie}
r_j(y)=y-\langle d \alpha_j,y\rangle d\alpha_j^\vee(1),
\end{equation}
where we write  $d\alpha_j$ for $d_1\alpha_j\, \colon \Lie T_\C\to\Lie\G_{m,\C}= \C$, and
we write $\langle d \alpha_j,y\rangle$ for $d\alpha_j(y)\in\C$.

Let $\omega_k^\vee\in \Lie T_\C$ be the element such that
 $\langle d\alpha_j,\omega_k^\vee\rangle=\delta_{jk}$\,, where  $\delta_{jk}$ is Kronecker's delta symbol.
We set
\[\ttil=\exp\left(\pi \ii\,\sum_k\hh t_k \omega_k^\vee \right) \in T(\C),\quad\text{where }\ii^2=-1.\]
Then
\begin{equation*}
\alpha_j(\ttil)=\exp \left\langle d\alpha_j,\,\pi \ii\,\sum_k t_k \omega_k^\vee\right\rangle=
\exp \left(\pi\ii\,\sum_k t_k\langle d\alpha_j,\omega_k^\vee\rangle\right)=
\exp(\pi\ii\, t_j)=(-1)^{t_j},
\end{equation*}
because the exponential map commutes with homomorphisms of Lie groups; see \cite[Section 1.2.7, p.~29, Problem 26]{OV}.
It follows that the image of $\ttil$ in $T^\ad(\C)$ is indeed $t$.
By \eqref{eq:Springer-reflection-Lie} we have
\begin{align*}\label{eq:Dima}
n_j\ttil n_j^{-1} \ttil^{-1}&=r_j(\ttil) \ttil^{-1}
=\exp\left(\pi \ii\, \sum_k\, t_k\left(r_j(\omega_k^\vee)-\omega_k^\vee\right)\,\right)\\
&=\exp\left(-\pi\ii\sum_k t_k\langle d\alpha_j,\omega_k^\vee\rangle d\alpha_j^\vee(1)\right)
=\exp\left( t_j\, d\alpha_j^\vee(-\pi\ii)\right)
=\left(\alpha_j^\vee(-1)\right)^{t_j}.
\end{align*}
Thus $n_j\ttil n_j^{-1} \ttil^{-1}=\left(\alpha_j^\vee(-1)\right)^{t_j}$, as required.
\end{proof}

According to Lemma \ref{prop:twisted}, the twisted action of $\T_i$ on a labeling $\aa=(a_i)\in (\Z/2\Z)^n$
is given by formula \eqref{twisted-simply-laced}.
This means that  for any vertex $i$ of $D$,
the action of $\T_i$   is given
by formula \eqref{non-twisted-simply-laced} if vertex $i$ is white  (i.e., $t_i=0$),
and by formula
\begin{equation}\label{twisted-action}
 a'_i = a_i + 1+ \Ss a_k \ ,
\end{equation}
if vertex $i$ is black (i.e., $t_i=1$).
In other words, this is exactly the generalized Reader puzzle as described in the Introduction.
We denote by $L(D,\ttt)$ (or $L(\hs_\ttt D)$) the set of labelings $(\Z/2\Z)^n$
with this twisted action of the moves $\T_i$.
By Lemma \ref{prop:twisted} the action of $\T_i$ on $L(D,\ttt)$
is compatible with the $t$-twisted action of the reflection $r_i\in W$
on $T(\R)_2=\hs_t T(\R)_2$ with respect to the canonical bijection
\begin{equation*}
\gamma_\ttt\colon L(D,\ttt)\isoto T(\R)_2\subset Z^1(\R,\hs_t G) \quad \aa\mapsto a=\prod_i \left(\alpha_i^\vee(-1)\right)^{a_i}.
\end{equation*}
By abuse of notation we denote by $\gamma_\ttt$ both the isomorphism $\gamma_\ttt\colon L(\hs_\ttt D)\isoto T(\R)_2$
and the embedding $\gamma_\ttt\colon L(\hs_\ttt D)\isoto T(\R)_2\into Z^1(\R,G)$.
We regard the twisted diagram  $_\ttt D=(D,\ttt)$ as the {\em colored Dynkin diagram of the twisted group $\hs_t G$}
(with respect to $T$ and $\Pi$).
We denote by $\Or(\hs_\ttt D)$ the set of equivalence classes (orbits) in $L(\hs_\ttt D)$
with respect to the equivalence relation given by the  moves of Lemma \ref{prop:twisted}
(in the Introduction we denoted this set of equivalence classes by $\Cl(D,\ttt)$\hs).

The following theorem describes the Galois cohomology of an {\em inner} form $_t G$
of a  compact, simply connected, simple $\R$-group $G$
in terms of labelings of the corresponding colored Dynkin diagram $_\ttt D$.

\begin{theorem}\label{thm:inner}
Let $G$, $T$, $R$, $\Pi$, $D$, and $W$ be as in Section \ref{sec:compact}.
Let $t\in T^\ad(\R)_2$ and let $\ttt\in(\Z/2\Z)^n$ be defined by \eqref{eq:t-bold}.
Let $L(_\ttt D)$ be the set of labelings of the colored Dynkin diagram $_\ttt D$
with the moves given by formula \eqref{twisted-simply-laced}.
Then the canonical map
$$
\gamma_\ttt\colon L(\hs_\ttt D)\isoto T(\R)_2 \into Z^1(\R,\hs_t G)
$$
induces a canonical bijection
$$
\lambda_\ttt\colon \Or(\hs_\ttt D)\isoto H^1(\R,\hs_t G).
$$
\end{theorem}
The theorem follows immediately from Proposition \ref{prop:Bo88} and Lemma  \ref{prop:twisted}.
We specify that $\lambda_\ttt$ takes the orbit (class) of a labeling $\aa=(a_j)$ to the cohomology class of the cocycle
\[\prod_{j=1}^n(\alpha_j^\vee(-1))^{a_j}\in T(\R)_2\subset Z^1(\R,\hs_t G).\]


\section{Weyl action for outer forms}
\label{sec:outer}

In this section again $G$, $T$, $R$, $\Pi$, $B$, $D$, $N$ and $W$ are as in Section \ref{sec:compact},
in particular $G$ is a  simply connected, simple,
{\em compact}  linear algebraic group over $\R$, and $B$ is
 the Borel subgroup of $G_\C$ containing $T_\C$, corresponding to the basis $\Pi$ of $R$.

Let $\rho\in\Gal(\C/\R)$ denote the complex conjugation.
Since $G$ is defined over $\R$, the Galois group $\Gal(\C/\R)=\{1,\rho\}$ acts on $\Aut\,G_\C$.
The {\em group of semi-automorphisms} $\SAut\,G_\C:=(\Aut\,G_\C)\rtimes \Gal(\C/\R)$ acts on $D$,
see  \cite[Proposition 3.1]{BKLR}. We describe this action here.

 We construct a homomorphism
\begin{equation*}
\psi_S\colon \SAut\, G_\C\to\Aut\,D.
\end{equation*}
Let $s\in\SAut\,G_\C$.
We have $T_\C\subset B\subset G_\C$.
Consider the pair $(s(T_\C),s(B))$.
There exists $g\in G(\C)$ such that
\[g\cdot s(T_\C)\cdot g^{-1}=T_\C,\quad g\cdot s(B)\cdot g^{-1}=B,\]
and if $g'\in G(\C)$ is another such element, then $g'=t'g$ for some $t'\in T(\C)$.
We obtain a semi-automorphism
\[\Inn(g)\circ s \in\SAut(G_\C,T_\C,B),\]
which induces a well-defined automorphism
\[\psi_S(s)\in\Aut\, D.\]

The restriction of $\psi_S$ to the subgroup \ $\Gal(\C/\R)\subset \SAut\, G$ \
gives the  {\em $^*$-action} of  $\Gal(\C/\R)$  on $D$, see \cite[Section 2.3]{Tits}.
Let
\begin{equation}\label{e:psi}
\psi\colon \Aut\, G_\C\to\Aut\,D
\end{equation}
denote the restriction of $\psi_S$ to the subgroup $\Aut\, G_\C\subset \SAut\, G_\C$, then
$\psi$  is clearly $\Gal(\C/\R)$-equivariant with respect to the $^*$-action of $\Gal(\C/\R)$ on $D$.
The homomorphism $\psi$ fits into the exact sequence
\begin{equation}\label{split-Springer}
1\to G^\ad(\C)\to \Aut\,G_\C\labelto{\psi} \Aut\,D\to 1
\end{equation}
which admits a {\em splitting,} that is, a homomorphism
$\phi\colon\Aut\,D\to\Aut\,G_\C$ such that $\psi\circ\phi=\id_{\Aut D}$\,,
see \cite[Expos\'e XXIV, Theorem 1.3]{SGA3} or \cite[Corollary 2.14]{SpringerAMS}, or \cite[Proposition 1.5.5]{Conrad-RGS}.
We construct  a splitting of \eqref{split-Springer} of a special kind  in the next lemma.

\begin{lemma}\label{lem:phi}
Let $G$ be as above, in particular compact and simply connected.
Then there exists a  homomorphism
\[\phi\colon\Aut\,D\to\Aut\, G_\C,\quad \theta\mapsto\phi_\theta\]
such that $\psi\circ\phi=\id_{\Aut D}$
and for any $\theta\in\Aut\,D$ the automorphism $\phi_\theta\in\Aut\,G_\C$ is defined over $\R$.
\end{lemma}

\begin{proof}
Consider the complexification $\gg_\C$ of $\gg=\Lie G$ and the root decomposition
\[ \gg_\C=\Lie T_\C \oplus\bigoplus_{\beta\in R}\gg_\beta\,. \]
Consider a ``canonical system of generators''  $h_i\,,e_i\,,f_i\ (i=1,\dots,n)$  of $\gg_\C$
satisfying
\begin{align*}
&[h_i\,,h_j]=0,\quad [e_i\,,f_i]=h_i\,,\quad [e_i\,,f_j]=0\text{ for }i\neq j\\
&[h_i\,,e_j]=a_{ji}e_j\,,\quad [h_i\,, f_j]=-a_{ji} f_j\,,
\end{align*}
see \cite[Section 4.3.2]{OV}.
Here $(a_{ij})$ is the Cartan matrix,
\[e_i\in \gg_{\alpha_i},\quad f_i\in \gg_{-\alpha_i},\quad h_i\in \Lie T_\C,\quad \Pi=\{\alpha_1,\dots,\alpha_n\}.\]
Since $G$ is compact, one can choose the generators $h_i\,,e_i\,,f_i$ such that
\[^\rho h_i=-h_i\,, \quad ^\rho e_i=-f_i\,, \quad ^\rho f_i=-e_i\,, \]
see \cite[Section 5.1.3, Problem 19]{OV}.

Now let $\theta\in \Aut\,D$.
We define an automorphism $\phi_\theta$ of $\gg_\C$ on the generators by
\[\phi_\theta(h_i)=h_{\theta(i)}\,,\quad \phi_\theta(e_i)=e_{\theta(i)}\,,\quad  \phi_\theta(f_i)=f_{\theta(i)}\,.\]
Clearly $\phi_\theta$ commutes with $\rho$, hence the automorphism $\phi_\theta$ of $\gg_\C$ is defined over $\R$.
The $\R$-automorphism $\phi_\theta$ of $\gg$ induces a unique automorphism of the connected simply connected algebraic $\R$-group $G$;
by abuse of notation we denote this automorphism again by $\phi_\theta$.
We have $\phi_\theta(T)=T,\ \phi_\theta(B)=B$, and it is clear from our construction of $\psi$ that $\psi(\phi_\theta)=\theta$, hence
$\psi\circ\phi=\id_{\Aut D}$\,.
Clearly
\[\phi\colon\Aut\,D\to\Aut_\R\, G,\quad \theta\mapsto \phi_\theta\]
is a homomorphism.
\end{proof}

\begin{corollary}
The complex conjugation $\rho$, when acting on $D$ via the $^*$-action, acts on $\Aut\,D$ trivially.
\end{corollary}
\begin{proof}
Indeed, if $\theta\in\Aut\,D$, then
\[^\rho\theta=\hs^\rho(\psi(\phi_\theta))=\psi(\hs^\rho(\phi_\theta))=\psi(\phi_\theta)=\theta,\]
because $\phi_\theta\in\Aut_\R\,G$.
\end{proof}

Let $_z G$ be an outer twisted form (outer twist) of $G$, where $z\in Z^1(\R,\Aut\, G)$ and $z\notin Z^1(\R,\Inn\, G)$.
The homomorphism $\psi$ of \eqref{e:psi} induces a map
\[ Z^1(\R,\Aut\,G)\to Z^1(\R,\Aut\,D)=(\Aut\,D)_2\,.\]
We obtain an element $\tau=\psi(z)\in (\Aut\,D)_2$, then $\tau$ is a nontrivial involutive automorphism of $D$.
It acts on the set of vertices $\Pi$ of $D$; we write $\alpha_j\mapsto \alpha_{\tau(j)}$, $j=1,\dots,n$.
We write $\Pi^\tau$ for the set of fixed points of $\tau$ in $\Pi$, and $D^\tau$ for the corresponding Dynkin subdiagram.
Furthermore, $\tau$ acts on  $\Pi^\vee$ by $\alpha_j^\vee\mapsto\alpha_{\tau(j)}^\vee$
and on $W=\langle r_j\rangle_{j=1,\dots,n}$ by $\tau(r_j)=r_{\tau(j)}$.
We write $W^\tau$ for the algebraic subgroup of fixed points of $\tau$ in $W$.

The homomorphism $\phi$ of  Lemma \ref{lem:phi} gives an involutive automorphism $\phi_\tau$ of $(G,T,B)$.
By abuse of notation, we shall denote this ``diagrammatic'' automorphism $\phi_\tau$ again by $\tau$.
Then $\tau\in\Aut_\R(G,T)_2$ and $\tau$ acts on $W$.
We write $_\tau T$, $_\tau T^\ad$, $_\tau G$, $_\tau G^\ad$, and $_\tau W$ for the corresponding twisted algebraic groups.

We consider the action of $\tau$ on $\Pi$ and on $\Pi^\vee$.
The decomposition
$$
\Pi=\Pi^\tau\cup (\Pi\smallsetminus \Pi^\tau)
$$
of the basis $\Pi$ of the character group $\X^*(T^\ad)$ of the adjoint torus $T^\ad$
induces a $\tau$-invariant decomposition into a direct product
\begin{equation}\label{e:T-product}
T^\ad=T^\ad(D^\tau)\times_\R T^\ad(D\smallsetminus D^\tau)
\end{equation}
with $\X^*(T^\ad(D^\tau))=\langle\Pi^\tau\rangle$ and $\X^*(T^\ad(D\smallsetminus D^\tau))=\langle\,\Pi\smallsetminus \Pi^\tau\rangle$.
Here for  a subset $S\subset \X^*(T^\ad)$, we denote by $\langle S\rangle$ the subgroup generated by $S$.
Concerning the corresponding $\tau$-twisted tori, we see that $_\tau T^\ad(D^\tau)=T^\ad(D^\tau)$ is a compact torus,
while the $\R$-torus $_\tau T^\ad(D\smallsetminus D^\tau)$
is isomorphic to the Weil restriction of scalars $R_{\C/\R} T'$ of some $\C$-torus $T'$.
It follows that
$$
H^1(\R,\hs_\tau T^\ad(D^\tau))=T^\ad(D^\tau)(\R)_2,\quad \text{while}\quad H^1(\R,\hs_\tau T^\ad(D\smallsetminus D^\tau))=1,
$$
and therefore, the embedding ${}_\tau T^\ad(D^\tau)\into\hs_\tau T^\ad$ induces a canonical isomorphism
\begin{equation}\label{eq:isomorphism-ad}
T^\ad(D^\tau)(\R)_2=H^1(\R,\hs_\tau T^\ad(D^\tau))\isoto H^1(\R,\hs_\tau T^\ad).
\end{equation}

Similarly, we have a decomposition
$$
\Pi^\vee=\Pi^{\vee\,\tau}\,\cup\, (\Pi^\vee\smallsetminus \Pi^{\vee\,\tau}),
$$
where we write $\Pi^{\vee\,\tau}$ for $(\Pi^\vee)^\tau$. This decomposition
of the basis $\Pi^\vee$ of the cocharacter group $\X_*(T)$ induces a $\tau$-invariant decomposition into a direct product
$$
T=T(D^\tau)\times T(D\smallsetminus D^\tau)
$$
with $\X_*(T(D^\tau))=\langle\,\Pi^{\vee\,\tau} \rangle$ and
$\X_*(T(D\smallsetminus\, D^\tau))=\langle\,\Pi^\vee\smallsetminus\, \Pi^{\vee\,\tau} \rangle$.
As above, we have $_\tau T(D^\tau)=T(D^\tau)$, hence
$$
H^1(\R,\hs_\tau T(D^\tau))=T(D^\tau)(\R)_2,\quad \text{while}\quad H^1(\R,\hs_\tau T(D\smallsetminus D^\tau))=1.
$$

The involutive automorphism $\tau$ of $D$ acts on the set of labelings $L(D)$, and we denote by $L(D)^\tau$ the subset of invariants.
The homomorphism $\gamma\colon L(D)\to  T(\C)_2$ given by formula \eqref{e:gamma} induces an isomorphism $L(D)^\tau\to\hs_\tau T(\R)_2$
(because the complex conjugation acts on $_\tau T(\C)_2$ as $\tau$).
We obtain a commutative diagram
\begin{equation}\label{eq:isomorphism}
\xymatrix{
L(D^\tau)\ar[r]^-\sim \ar@/_1pc/[d]   &T(D^\tau)(\R)_2\ar[r]^-\sim\ar@/_1pc/[d]    &H^1(\R, T(D^\tau))\ar[d]^\sim  \\
L(D)^\tau\ar[r]^-\sim\ar[u]           &_\tau T(\R)_2\ar[r]\ar[u]                      &H^1(\R,\hs_\tau T)
}
\end{equation}
with obvious maps.

For our $z\in Z^1(\R,\Aut\,G_\C)$ and $\tau=\phi_{\psi(z)}$ we have $\psi(z)=\psi(\tau)$.
It follows from the exact sequence  \eqref{split-Springer} and \cite[I.5.5, Corollary 2 of Proposition 39]{Serre}
that our outer form $_z G$ of $G$ is an {\em inner} twist of $_\tau G$,
i.e. $_z G\simeq\hs_{z'} (\hs_\tau G)$ for some $z'\in Z^1(\R,\hs_\tau G^\ad)$.
By Proposition \ref{prop:Bo88} the cocycle $z'$ is cohomologous to some $t\in Z^1(\R,\hs_\tau T^\ad)\subset Z^1(\R,\hs_\tau G^\ad)$,
and by \eqref{eq:isomorphism-ad} we may assume that
$t\in T^\ad(D^\tau)(\R)_2 \subset \hs_\tau T^\ad(\R)_2$.
We denote by $\Inn(t)$ the corresponding inner automorphism of $_\tau G$ of order dividing 2.
We set $\sigma=\Inn(t)\circ\tau$.
Note that $\Inn(t)$ and $\tau$ commute, hence $\sigma$ is an outer automorphism of order 2 of $G$.
We write $_{\sigma} G=\, _{\Inn(t)}(_\tau G)$ for the corresponding twisted form of $G$, then $\sigma\sim z$ and $_\sigma G\simeq\hs_z G$.
For simplicity we also write $_{\sigma} G=\hs_{t\tau} G$.
We have $_\sigma G(\C)= G(\C)$, but the complex conjugation in $_\sigma G(\C)$ is given by
$$
^*\gbar=\sigma(\gbar)=\Inn(t)(\tau(\gbar)) .
$$
Note that $\Inn(t)$ acts trivially on $_\tau T$, hence also on $_\tau W$,
because $_\tau W\subset\Aut(_\tau T)$.
We see that $_{t\tau} T=\hs_\tau T$ and $_{t\tau} W=\hs_\tau W$.

We consider the group $W_0:=W_0({}_{t\tau} G)=W_0({}_\tau G)$; see Section \ref{sec:1}.
We have $W_0(\C)=W_0(\R)={}_\tau W(\R)$; see \cite[Section 7]{Borovoi-arXiv}.
Clearly ${}_\tau W(\R)=W^\tau(\C)$, hence $W_0(\R)=W^\tau(\C)$.
The group $W_0(\R)$ acts on $H^1(\R,{}_{t\tau} T)=H^1(\R,{}_\tau T)$ as in formula \eqref{eq:Bo-action},
and it acts on the set of labelings $L(D^\tau)$
via  \eqref{eq:isomorphism}.
We wish to describe this action explicitly.

Note that if $D$ is of type $\AA_{2n}$, then $D^\tau=\emptyset$, $T(D^\tau)=1$,
$H^1(\R,{}_\tau T)=1$, $H^1(\R,{}_\tau G)=1$ (in this case $_\tau G\simeq \SL_{2n+1}$).
From now till the end of this section we shall assume that {\em $D$ is not of type $\AA_{2n}$.}
Then from the classification of Dynkin diagrams we know that for any $j\in D\smallsetminus D^\tau$,
the vertices $j$ and $\tau(j)$ are not connected by an edge,
and therefore, the reflections $r_j$ and $r_{\tau(j)}$ commute.

\begin{lemma}\label{lem-pairs}
Assume that  $D$ is not of type $\AA_{2n}$.
Then the group $W_0(\R)$
is generated by the reflections $r_i$ for $i\in D^\tau$
and by the products $r_j\cdot r_{\tau(j)}$ for $j\in D\smallsetminus D^\tau$.
\end{lemma}

\begin{proof}
We have
$$
 W_0(\R)={}_\tau W(\R)=W^\tau(\C).
$$
Now the lemma follows from \cite[Proposition 13.1.2]{Ca}.
\end{proof}

\begin{lemma}\label{lem:prod-pairs}
Assume that  $D$ is not of type $\AA_{2n}$.
Let $j\in D\smallsetminus D^\tau$.
Then the product   $r_j\cdot r_{\tau(j)}$ acts trivially on $H^1(\R,\hs_\sigma T)=H^1(\R,\hs_\tau T)$,
where $\sigma=\Inn(t)\circ\tau$.
\end{lemma}

\begin{proof}
Let $b\in Z^1(\R,{}_\tau T)$.
By diagram \eqref{eq:isomorphism} we may assume that $b\in T(D^\tau)(\R)_2\subset T(\C)_2$.
Let $\bb=(b_k)\in (\Z/2\Z)^D$ be the corresponding labeling of $D$ such that $b=\prod_k (\alpha_k^\vee(-1))^{b_k}$.

We set $w_{j,\tau(j)}=r_j r_{\tau(j)}$.
Let $G_j=G_{\alpha_j}$ denote the simple 3-dimensional subgroup of $G$ corresponding to the simple root $\alpha_j$.
Choose a representative $n_j\in G_j(\R)\cap \mathcal{N}_G(T)(\R)$ of $r_j\in W(\R)$.
Set $n_{\tau(j)}=\tau(n_j)\in G_{\tau(j)}(\R)$, then $\tau(n_{\tau(j)})=n_j$, because $\tau^2=1$.
Since the vertices $j$ and $\tau(j)$ are not connected by an edge,
the subgroups $G_j$ and $G_{\tau(j)}$ of $G$ commute,
hence $n_j$ and $n_{\tau(j)}$ commute.
Set $n_{j,\tau(j)}:=n_j n_{\tau(j)}$, then
\[\tau(n_{j,\tau(j)})=\tau(n_j n_{\tau(j)})=n_{\tau(j)} n_j= n_j n_{\tau(j)}=n_{j,\tau(j)}\, ,\]
hence $n_{j,\tau(j)}\in \mathcal{N}_G(T)(\R)^\tau$
and $n_{j,\tau(j)}$ represents $w_{j,\tau(j)}$.

We consider the action \eqref{eq:Bo-action} of $w_{j,\tau(j)}$ on $H^1(\R,\,_{t\tau}T)$.
We write $w$ for $w_{j,\tau(j)}$ and $n$ for $n_{j,\tau(j)}$.
Recall that $\sigma=\Inn(t)\circ\tau$, where $t\in T^\ad(D^\tau)(\R)_2\subset T^\ad(\C)_2$.
We lift $t$ to some $\ttil\in T(\C)$.
Then we have
\begin{equation*}
w* [b]:=[nbn^{-1}\cdot n\,\Inn(t)(\tau(\nbar)^{-1})]=
[nbn^{-1}\cdot n\ttil\tau(\nbar)^{-1}\ttil^{-1}]=[nbn^{-1}\cdot n\ttil n^{-1}\ttil^{-1}],
\end{equation*}
because $\tau(\nbar)=n$.
Thus the action \eqref{eq:Bo-action} of $w\in W_0(\R)$ on $Z^1(\R,\,_{t\tau}T)$
is compatible with the action   \eqref{eq:ect-0-gen} of $w\in W(\C)$ on $T(\C)_2$.

We consider the $t$-twisted action \eqref{eq:ect-0-gen} of $W(\C)$ on $T(\C)_2$.
Then Lemma \ref{prop:twisted} is applicable, and it
implies that the move $\T_j$ corresponding to the reflection $r_j\in W(\C)$
can change only the $j$-coordinate $b_j$ of $\bb$.
Now consider $w_{j,\tau(j)}=r_{\tau(j)}r_j\in W_0(\R)\subset W(\C)$ for $j\in D\smallsetminus D^\tau$,
then we see that $\T_{\tau(j)}\T_j$ can change only the $j$- and the $\tau(j)$-coordinates of $\bb$.
In particular, if we write $\bb'=  (\T_j \T_{\tau(j)}) \bb$, then $b'_i=b_i$ for any $i\in D^\tau$.

Since $w_{j,\tau(j)}\in W_0(\R)$ and $b\in Z^1(\R,\hs_\sigma T)$,
we see that $b':=w_{j,\tau(j)}(b)$ is contained in $Z^1(\R,\hs_\sigma T)$.
Since $b'_i=b_i$ for any $i\in D^\tau$, by diagram \eqref{eq:isomorphism} $b'\sim b$ in $Z^1(\R,\hs_\sigma T)$.
Thus $w_{j,\tau(j)}=r_{\tau(j)}\,r_j$ acts trivially on $H^1(\R,\hs_\sigma T)$.
\end{proof}

\begin{lemma}\label{lem:r-i-t-tau}
Let $a\in T(D^\tau)(\R)_2\subset Z^1(\R,\hs_\sigma T)$.
Let $i\in D^\tau$, and write $[a']= r_i[a]$, where $a'\in T(D^\tau)(\R)_2\subset Z^1(\R,\hs_\sigma T)$,
and $[a]\mapsto r_i[a]$ refers to the action  \eqref{eq:Bo-action} of $r_i\in W_0(\R)$ on $H^1(\R,\,_\sigma T)$.
Write
\begin{equation*}
a=\prod_{j\in D^\tau} \left(\alpha_j^\vee(-1)\right)^{a_j},\qquad
a'=\prod_{j\in D^\tau} \left(\alpha_j^\vee(-1)\right)^{a'_j}
\end{equation*}
Then $a'_j=a_j$ for $j\neq i$ and
\begin{equation}\label{eq:twisted-outer}
a'_i=a_i+t_i+\Ssd a_k\,,
\end{equation}
where $(-1)^{t_i}=\alpha_i(t)$ and the sum is taken over the neighbors $k$ of $i$ {\em lying in $D^\tau$}.
\end{lemma}

\begin{proof}
Write $a=\prod_{j\in D} \left(\alpha_j^\vee(-1)\right)^{a_j}$, then $a_j=0$ for $j\in {D\smallsetminus D^\tau}$.
Now let $i\in D^\tau$, then arguing as in the proof of Lemma \ref{lem:prod-pairs},
we see that the action \eqref{eq:Bo-action} of $r_i\in W_0(\R)$
is compatible with the action \eqref{eq:ect-0-gen}, where $n=n_i\in G_i(\R)\cap \mathcal{N}_G(T)(\R)$.
By Lemma \ref{prop:twisted} this action is given
by formula \eqref{twisted-simply-laced}, i.e., by formula \eqref{eq:twisted-outer},
where the sum is taken over {\em all}  neighbors $k$ of $i$ in $D$.
However, if $k\in D\smallsetminus D^\tau$, then $a_k=0$.
We see that in formula \eqref{eq:twisted-outer}
we may take the sum only over neighbors $k$ of $i$ contained in $D^\tau$,  as required.
\end{proof}

The following theorem was announced in \cite{Bo}.
It reduces computing the Galois cohomology of an outer form of a simply connected compact group $G$
to computing the Galois cohomology of an inner form of some other
simply connected compact group (of type $\AA_l$ for some $l$).

\begin{theorem}[{\cite[Theorem 3]{Bo}}]
\label{cor:Theorem-3-Bo}
Let $G$ be a simply connected, simple,
compact  linear algebraic group over $\R$.
Let\, $T$, $R$, $\Pi$ and $D$ be as in Section \ref{sec:compact}.
Let $\tau$ be an automorphism of order $2$ of the Dynkin diagram $D$ of $G_\C$
(then $D$ is simply-laced).
Let $l=\# D^\tau$.
Let $G(D^\tau)\subset G$ be the $\R$-subgroup of type $\AA_l$
corresponding to the Dynkin subdiagram $D^\tau$ of $D$.
Let $t\in T^\ad(D^\tau)(\R)_2\subset \hs_\tau T^\ad(\R)_2$.
Then the natural embedding
$\hs_t G(D^\tau)\into {}_{t\,\tau} G,$
obtained by twisting by $t$ from the embedding $G(D^\tau)\into\hs_\tau G$,
induces a bijection
$$
H^1(\R,\hs_t G(D^\tau))\isoto H^1(\R,{}_{t\,\tau} G).
$$
\end{theorem}

\begin{proof}
The embedding $T(D^\tau)\into {}_\tau T$ induces an isomorphism
\begin{equation}\label{eq:isom-2}
H^1(\R,T(D^\tau))\isoto H^1(\R, {}_\tau T).
\end{equation}
The group $W(D^\tau)(\R)$ acts on the left-hand side of \eqref{eq:isom-2};
this group is generated by $r_i$ for $i\in D^\tau$.
The group $W_0(\R)$ acts on the right-hand side; by Lemma \ref{lem-pairs}
this group is generated by  $r_i$ for $i\in D^\tau$ and by $r_{\tau(j)} r_j$ for $j\in {D\smallsetminus D^\tau}$.
By Lemma \ref{lem:prod-pairs} the products $r_{\tau(j)} r_j$ for $j\in {D\smallsetminus D^\tau}$
act trivially on the right-hand side of \eqref{eq:isom-2}.
Comparing formulas \eqref{twisted-simply-laced} and \eqref{eq:twisted-outer},
we see that the actions of a reflection $r_i$ for $i\in D^\tau$
on the left-hand side and the right-hand side of  \eqref{eq:isom-2} are compatible.
Thus we obtain a bijection of the quotients:
$$
H^1(\R,\hs_t G(D^\tau))=W(D^\tau)(\R)\backslash T(D^\tau)(\R)_2 \isoto W_0(\R)\backslash H^1(\R,{}_\tau T)= H^1(\R,{}_{t\,\tau} G),
$$
where the left-hand and right-hand equalities are  bijections of Proposition \ref{prop:Bo88}.
\end{proof}

From diagram \eqref{eq:isomorphism} and Theorem \ref{cor:Theorem-3-Bo} we obtain a commutative diagram
\begin{equation}\label{eq:bijections}
\xymatrix{
L(D^\tau)\ar[r]^-\sim \ar@/_1pc/[d]   &T(D^\tau)(\R)_2\ar[r]^-\sim\ar@/_1pc/[d]    &H^1(\R, T(D^\tau))\ar[r]\ar[d]^\sim   &H^1(\R,\hs_t G(D^\tau))\ar[d]^\sim \\
L(D)^\tau\ar[r]^-\sim\ar[u]           &_\tau T(\R)_2\ar[r]\ar[u]                      &H^1(\R,\hs_\tau T)\ar[r]              &H^1(\R,\hs_{t\hs\tau} G)
}
\end{equation}

Recall that $t\in T^\ad(D^\tau)(\R)_2$.
We define a coloring    $\ttt=(t_j)_{j\in D^\tau}\ (t_j\in \Z/2\Z)$ of $D^\tau$
as in \eqref{eq:t-bold}, i.e.,  by $(-1)^{t_j}=\alpha_j(t)$.

\begin{proposition}\label{c:restriction}
\begin{enumerate}
\item[(i)]
In diagram \eqref{eq:bijections}, the map $L(D)^\tau\to H^1(\R,\hs_{t\hs\tau} G)$
of the bottom row of  the diagram  is surjective.
\item[(ii)]
Two labelings $\aa,\aa'\in L(D)^\tau$ have the same image in $H^1(\R,\hs_{t\hs\tau} G)$
if and only if their images in $L(D^\tau)$
(i.e., their restrictions to $D^\tau$) lie in the same equivalence class in $L(D^\tau,\ttt)$.
In particular, a labeling $\aa\in L(D)^\tau$ maps to $[1]\in H^1(\R,\hs_{t\hs\tau} G)$
if and only if its restriction to $D^\tau$
lies in the equivalence class of 0 in $L(D^\tau,\ttt)$.
\end{enumerate}
\end{proposition}

\begin{proof} (i) This follows from Lemma \ref{lem:Bo88} and Proposition \ref{prop:Bo88}.

(ii) Indeed, by Theorem \ref{thm:inner}
two labelings $\bb,\bb'\in L(D^\tau)$ have the same image in $H^1(\R,\hs_t G(D^\tau))$
if and only if they lie in the same equivalence class in $L(D^\tau,\ttt)$.
\end{proof}

We say that two labelings $\aa,\aa'\in L(D)^\tau$ are {\em $\ttt$-equivalent}
if their restrictions to $D^\tau$ lie in the same equivalence class in $L(D^\tau,\ttt)$.
We denote by $\Cl(D,\tau,\ttt)$ the set of $\ttt$-equivalence classes in $L(D)^\tau$,
then the restriction map $L(D)^\tau\to L(D^\tau)$ induces a bijection
\begin{equation}\label{e:res-bij}
\Cl(D,\tau,\ttt)\isoto \Or(D^\tau,\ttt).
\end{equation}
By Proposition \ref{c:restriction} we have a bijection
\begin{equation}\label{e:general-bijection}
 \Cl(D,\tau,\ttt)\isoto H^1(\R,\hs_{t\hs\tau} G).
\end{equation}
We obtain a bijection $\Or(D^\tau,\ttt)\isoto H^1(\R,\hs_{t\hs\tau} G)$.
We specify that this bijection takes the class of a labeling $\bb=(b_\alpha)_{\alpha\in \Pi^\tau}$
of $D^\tau$ to the cohomology class of the cocycle
\[\prod_{\alpha\in \Pi^\tau}(\alpha^\vee(-1))^{b_{\alpha}}\in T(D^\tau)(\R)_2\subset Z^1(\R, \hs_{t\tau} G).\]

In the case when $_z G$ is an {\em inner} form of the compact group $G$,
we again set $\tau=\psi(z)\in\Aut(D)$, then $\tau=1$,  and we set $\Cl(D,\tau,\ttt)=\Or(D,\ttt)$ in this case.
In particular, when $_z G=G$ is {\em compact},
we have $\tau=1, t=1,\ttt=0$, and we set $\Cl(D,\tau,\ttt)=\Or(D)$ in this case.
Then we have bijection \eqref{e:general-bijection} in all the cases.

\section{Reeder puzzles from Kac diagrams}
\label{sec:Kac}

In this section $G$, $T$, $R$, $\Pi$, $B$, $D$, and $W$  are as in Section \ref{sec:compact},
in particular $G$ is a  simply connected, simple, {\em compact}  linear algebraic group over $\R$.

\subsection{Inner forms}
\label{ss:Kac}

Let $_zG$ be a noncompact inner twisted form (inner twist) of $G$, where $z\in Z^1(\R, G^\ad)$.
By \cite[III.4.5, Example (a)]{Serre}, $z$ is cohomologous to some $t\in T^\ad(\R)_2\subset Z^1(\R, G^\ad)$.
We regard $t\in T^\ad(\R)_2\subset (\Aut\, G)_2$ as an involutive inner automorphism of $G$.
Since $_zG$ is noncompact, we have $t\neq 1$.
By Kac \cite{Kac}, see also Helgason \cite[Ch.~X, \S\,5]{Helgason},
Onishchik and Vinberg \cite[Ch.~4, \S\,4 and Ch.~5, \S\,1]{OV},
and Gorbatsevich, Onishchik, and Vinberg \cite[Section 3.3.7]{OV2},
involutive {\em inner} automorphisms of $G$ can be described
using Kac diagrams of types I and II in \cite[Table 7]{OV}.
(The Kac diagrams of type III in \cite[Table 7]{OV} correspond to involutive {\em outer} automorphisms.)

The relation between Kac diagrams and involutive inner automorphisms
is as follows, see \cite[Problem 5.1.38]{OV}.
Consider the extended Dynkin diagram $\Dtil$ of $G$.
Its vertices correspond to the roots $\alpha_0,\alpha_1,\dots,\alpha_n$,
where $\alpha_1,\dots,\alpha_n$ are the simple roots and $\alpha_0$ is the {\em lowest} root.
There is a unique linear dependence
$$
m_0\alpha_0+m_1\alpha_1+\dots+m_n\alpha_n=0
$$
normalized so that  $m_0=1$, and then $m_j$ are positive integers tabulated in \cite[Table 6]{OV}.
A {\em Kac $2$-marking of} $\Dtil$ is a family of nonnegative integral
numerical marks $\qq=(q_j)_{j=0,1,\dots,n}\in\Z^{n+1}_{\ge 0}$
at the vertices ${j=0,1,\dots,n}$ of $\Dtil$, satisfying
\begin{equation}\label{eq:2markings}
m_0 q_0+m_1 q_1+\dots+m_n q_n=2.
\end{equation}
A Kac 2-marking $\qq$ determines a unique element $t\in T^\ad(\R)_2\subset G^\ad(\R)_2$ such that
\begin{equation}\label{e:qj}
\alpha_j(t)=(-1)^{q_j},\quad j=1,\dots,n.
\end{equation}
Involutive inner automorphisms of $G$ are classified, up to conjugacy, by Kac 2-markings of $\Dtil$.
Two Kac 2-markings $\pp$ and $\qq$ give conjugate inner automorphisms
if and only if $\pp$ can be obtained form $\qq$ by an automorphism of $\Dtil$,
see \cite[3.3.6, Theorem 3.11]{OV2}.

\begin{lemma}\label{l:special-transitive}
The group $\Aut\,\Dtil$ acts transitively on the set  $\{j\in\Dtil\ |\  m_j=1\}$.
\end{lemma}
\begin{proof}
By \cite[Section VI.2.3, Corollary of Proposition 6]{Bourbaki},
the center $Z(G)$ acts simply transitively on this set when acting on $\Dtil$.
\end{proof}

Let $\qq$ be a Kac 2-marking of $\Dtil$.
It follows from \eqref{eq:2markings} that there are three possibilities:
\begin{enumerate}
\item[(0)] $q_i=2$ for some $i\in \Dtil$ with $m_i=1$, and $q_j=0$ for all $j\neq i$.
\item[(1)] (Type I) $q_i=1$ for some $i$ with $m_i=2$, and $q_j=0$ for all $j\neq i$.
\item[(2)] (Type II) $q_{i_1}=1,\ q_{i_2}=1$ for some $i_1\neq i_2$
           with $m_{i_1}=1,\ m_{i_2}=1$, and $q_j=0$ for all $j\neq i_1,i_2$.
\end{enumerate}

In case (0) clearly $t=1$ and hence, $_tG=G$ is compact, so we do not consider this case.

In case (1) we have $i\neq 0$, because $m_i=2$, while $m_0=1$.
We color vertex $i$ of $\Dtil$ in black and leave the other vertices white.
We say that $(\Dtil,\qq)$ is a Kac diagram of type I.

In case (2) by Lemma \ref{l:special-transitive} we may and shall assume that $i_1=0$.
We write $i$ for $i_2$, then $i\neq 0$.
We color vertices $i_1=0$ and $i_2=i$ in black and leave all the other vertices white.
We say that $(\Dtil,\qq)$ is a Kac diagram of type II.

The Kac diagrams of types I and II up to isomorphism are tabulated in \cite[Table 7]{OV}.
From now on, when we consider Kac diagrams of types I and II, we  assume that they are ones from that table.
Then in both cases I and II, $\Dtil$ has exactly one nonzero black vertex $i$.

We compute the coloring $\ttt$ of $D$, induced by $t$, in terms of $\qq$.
Comparing \eqref{e:qj} and \eqref{eq:t-bold}, we obtain that
\[(-1)^{t_j}=(-1)^{q_j}\quad\text{for } j=1,\dots,n. \]
Since $t_j=0,1$ and $q_j=0,1$, we see that $t_j=q_j$ for $j=1,\dots,n$.
Thus the map $\ttt\colon D\to \{0,1\}$ is the restriction of $\qq\colon \Dtil\to\Z_{\ge0}$ to the subset $D$ of $\Dtil$.

We conclude that if a noncompact inner form $_tG$ of $G$
is given by a Kac diagram $(\Dtil,\qq)$ of type I or II from \cite[Table 7]{OV},
then we can obtain a colored Dynkin diagram $(D,\ttt)$ of the twisted group $_t G$
by removing vertex 0 from $(\Dtil,\qq)$.
We say that $(D,\ttt)$ is a {\em twisting diagram} for $_t G$.
It has exactly one black vertex $i$.
We call $i$ the {\em twisting vertex.}
The coloring $\ttt$ of $D$ defines the action of Lemma \ref{prop:twisted} of the Weyl group $W$ on $L(D)$;
we say that this action of $W$ is {\em twisted at} $i$.
The next proposition follows immediately from Lemma \ref{prop:twisted}.

\begin{proposition}\label{prop:twisted-i}
Let $_z G$ be a noncompact inner form of $G$.
Then $_z G\simeq \hs_t G$, where $t\in T^\ad(\R)_2$, $t\neq 1$, and $t$
comes from a Kac diagram $(\Dtil,\qq)$ of type I or II in \cite[Table 7]{OV}.
The corresponding colored Dynkin diagram
$(D,\ttt)$ is obtained from $(\Dtil,\qq)$ by removing vertex 0.
It has a unique black vertex $i$.
For the (twisted at $i$)  action of $W$ on $L(D,\ttt)$, we have the same formula \eqref{non-twisted-simply-laced}
for $\T_j$ for $j\neq i$ as in Lemma \ref{prop:non-twisted}.
For $\T_i$ we have, as  in Lemma \ref{prop:non-twisted}, $a'_j=a_j$ for $j\neq i$,
while in formula \eqref{non-twisted-simply-laced} for $a'_i$ we must add $1$.
Namely, we have
\begin{equation}\label{twisted-simply-laced-new}
 a'_i = a_i+1 + \Ss a_k\, ,
\end{equation}
where the meaning of $\Ss$ is the same as in formula \eqref{non-twisted-simply-laced}.
\end{proposition}

\begin{construction} \label{const:augmented-diag}
Assume we have a twisting diagram with a  black vertex $i$:
\begin{equation*}
\sxymatrix{ \cdots  \rline &\bc{}\rline  & \bcb{i} \rline  &\bc{}\rline & \cdots  }
\end{equation*}
We have formula \eqref{twisted-simply-laced-new}  for $\T_i$.
In order to get formula \eqref{non-twisted-simply-laced}  instead, we uncolor the vertex $i$
and {\em augment} our diagram by formally adding a new vertex which we call the {\em boxed 1},
connected by a simple edge to vertex $i$:
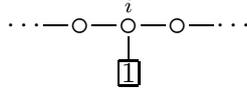
\begin{equation*}
\sxymatrix{  \cdots \rline &\bc{}\rline \rline & \bc{i} \rline \dline &\bc{}\rline & \cdots  \\
& & \boxone & & }
\end{equation*}
Here 1 in the box means that we put 1 as the label at this new  vertex.
Now the formula for $\T_i$ becomes \eqref{non-twisted-simply-laced}
where the boxed 1 is included in the sum.
Thus this boxed 1 accounts for twisting at $i$.
Note that we do not add a move corresponding to the boxed 1,
so the label 1 at the boxed 1 cannot be changed by moves.
We call the obtained diagram the {\em augmented diagram} corresponding to the twisting vertex $i$.
\end{construction}

\begin{example}
These are the Kac diagram, the twisting diagram and the augmented diagram
for the group $EIII$ of type $\EE_6$ (see Subsection \ref{subsec:E6(1)} below):
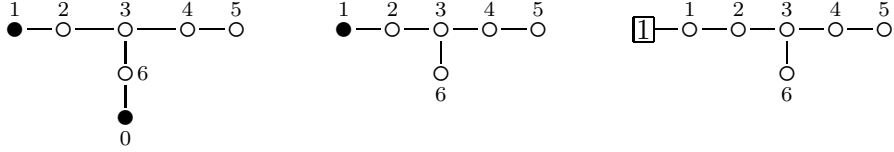
\begin{equation*}
\mxymatrix
{ \bcb{1} \rline & \bc{2} \rline & \bc{3} \rline \dline & \bc{4} \rline &
\bc{5} &&& \bcb{1} \rline & \bc{2} \rline & \bc{3} \rline \dline & \bc{4} \rline & \bc{5} &&&\boxone\rline &\bc{1}
\rline & \bc{2} \rline  & \bc{3} \rline \dline & \bc{4} \rline & \bc{5}
\\ & & {\phantom{\hbox{\SMALL\hs 6}}}\ccc \hbox{\SMALL\hs 6} \dline & &
&&&& &\bcu{6}  & & &&& &  & &\bcu{6} & &
\\ & & \bcbu{0} & & }
\end{equation*}
\end{example}

\subsection{Outer forms}
\label{ss:Kac-outer}

Let $_z G$ be an outer twisted form (outer twist) of $G$, where $z\in Z^1(\R,\Aut\, G)$.
Using the homomorphism $\psi\colon\Aut\,G_\C\to\Aut\,D$ of \eqref{e:psi},
we obtain an element $\tau=\psi(z)\in Z^1(\R,\Aut\,D)=(\Aut\,D)_2$; then $\tau^2=1$.
Since $_z G$ is an {\em outer} form, $\tau\neq 1$.
Thus $\tau$ is a nontrivial involutive automorphism of $D$.
We write $\Pi^\tau$ for the set of fixed points of $\tau$ in $\Pi$,
and $D^\tau$ for the corresponding Dynkin subdiagram.
As in Section \ref{sec:outer}, we  denote the ``diagrammatic''
automorphism $\phi_\tau$ of $(G,T,B)$ again by $\tau$.
There exists $t\in T^\ad(D^\tau)(\R)_2$ such that
$_z G\simeq\hs_{t\tau}G$, see Section \ref{sec:outer}.
(Here $T^\ad= T^\ad(D^\tau)\times_\R T^\ad(D\smallsetminus D^\tau)$, see \eqref{e:T-product}.)

By \cite{Kac}, see also \cite[Ch.~X, \S\,5]{Helgason},
 \cite[Ch.~4, \S\,4 and Ch.~5, \S\,1]{OV}, and \cite[Section 3.3.11]{OV2},
involutive {\em outer} automorphisms $\sigma=t\tau\in (\Aut\, G)_2$
can be described using Kac diagrams of type III in \cite[Table 7]{OV}
as follows.

Set $\gg=\Lie G$.
The diagrammatic automorphism $\tau$ acts on $T^\ad$, and we consider the torus $(T^\ad)^\tau$
of dimension $n'=\#\Pi^\tau+\half\#(\Pi-\Pi^\tau)$.
The nonzero weights $\alphabar$ of $(T^\ad)^\tau_\C$ in $\gg_\C$
are the restrictions of the roots $\alpha$ of $\gg_\C$ with respect to $T^\ad_\C$.
The restricted roots form a (possibly nonreduced) root system, see \cite[Section 3.3.9, Theorem 3.14]{OV2}.
In particular, for any restricted root $\alpha'$ the coroot $(\alpha')^\vee$ is defined.
Write $\gg^\tau$ and $\gg^{-\tau}$
for the $+1$ and $-1$ eigenspaces of $\tau$, respectively.
It is known (see \cite[Section 3.3.9]{OV2}) that $\gg^\tau_\C$ is a simple Lie algebra
and the representation of $G^\tau$ in $\gg^{-\tau}_\C$ is irreducible.

The set  $\Pibar=\{\alphabar_1,\dots,\alphabar_\nprime\}$ of {\em distinct} restrictions
 of simple roots $\{\alpha_1,\dots,\alpha_n\}$ to $(T^\ad)^\tau$
is in a bijection with the set of orbits of $\tau$ in $\Pi=\{\alpha_1,\dots,\alpha_n\}$.
This set $\Pibar$ is  a set of simple roots of $\gg^\tau_\C$
with respect to $(T^\ad)_\C^\tau$ (see \cite[Section 3.3.9]{OV2}).
Let $\alphabar_0$ denote the {\em lowest weight} of $(T^\ad)^\tau_\C$ in $\gg^{-\tau}_\C$.
Then $\{\alphabar_0,\alphabar_1,\dots,\alphabar_\nprime\}$ is an admissible system of roots
in the sense that the Cartan numbers $\langle \alphabar_\ibar,(\alpha_\jbar^\prime)^\vee\rangle$
are non-positive for $\ibar\neq \jbar$ (see \cite[Section 3.3.9]{OV2}).
The Cartan matrix is encoded by a twisted affine Dynkin diagram $\Dtilbar$, see \cite[Section 3.1.7]{OV2}.
There is a unique linear dependence
\[ \mbar_0\alphabar_0+ \mbar_1\alphabar_1+\dots +\mbar_\nprime\alphabar_\nprime=0,\]
normalized so that $\mbar_0=1$, then  $\mbar_\jbar$ are positive integers.

Set
\[\tfrak=\Lie T=\Lie T^\ad,\quad \tfrak_0=\tfrak^\tau=\Lie(T^\ad)^\tau,\quad \tfrak_1=\{x\in\tfrak\ |\ \tau(x)=-x\},\]
Set  $V=\ii\tfrak\subset\tfrak_\C,\quad V_0=\ii\tfrak_0,\ V_1=\ii\tfrak_1,$
where $\ii^2=-1$.
Let $V^*$, $V_0^*$, and $V_1^*$ denote the dual spaces to the vector $\R$-spaces $V$, $V_0$, and $V_1$, resp.
Let $Q=\X^*(T^\ad_\C)$ and  $Q_0=\X^*(\hs (T^\ad)_\C^\tau)$ be the corresponding character groups,
then $Q$ embeds into $V^*$ and $Q_0$ embeds into $V_0^*$.
The Killing form is positive definite on $V\subset \tfrak_\C\subset \Lie G_\C$.
The orthogonal decomposition $V=V_0\oplus V_1$  with respect to the Killing form
induces an identification $V^*=V_0^*\oplus V_1^*$,
and the restriction map $Q\to Q_0$ is compatible with the orthogonal projection
\[ V^*=V_0^*\oplus V_1^*\to V_0^*\]
with respect to this identification.
Note that all the simple roots $\alpha\in\Pi$ have the same length
(because $D$ admits a nontrivial automorphism).
We see that the restrictions (projections onto $V_0^*=(V^*)^\tau\subset V^*$\hs)
of the $\tau$-fixed roots $\alpha\in\Pi^\tau$  are longer
than the restrictions of the nonfixed roots $\alpha\in \Pi\smallsetminus\Pi^\tau$.

The involutive {\em outer} automorphisms $\sigma=t\tau\in (\Aut\, G)_2$
are classified, up to conjugacy, by Kac 1-markings of $\Dtilbar$.
A {\em Kac $1$-marking of}  $\Dtilbar$ is a family of nonnegative integer numerical marks
$\qq=(q_\jbar)_{\jbar=0,1,\dots,\nprime}\in\Z_{\ge 0}^{\nprime+1}$
at the vertices of $\Dtilbar$, satisfying
\begin{equation}\label{e:1-marking}
 \mbar_0 q_0+ \mbar_1 q_1+\dots +\mbar_\nprime q_\nprime=1.
 \end{equation}
A Kac 1-marking of $\Dtilbar$ determines a unique element $t\in(T^\ad)(\R)^\tau_2$ such that
\begin{equation}\label{e:aut-1-marking}
\alphabar_\jbar(t)=(-1)^{\varkappa}\quad \text{with } \varkappa=q_\jbar,\ \jbar=1,\dots,\nprime.
\end{equation}
With $\qq$ one associates the outer involutive automorphism $\sigma=t\tau$ of $G$.
Two Kac 1-markings $\pp$ and $\qq$ of $\Dtilbar$ give conjugate automorphisms of $G$
if and only if $\pp$ can be obtained from $\qq$ by an automorphism of $\Dtilbar$,
see \cite[3.3.10, Theorem 3.16]{OV2}.

For a 1-marking $\qq$ of $\Dtilbar$, it follows from \eqref{e:1-marking}
that there exists a unique vertex $\ibar$ with $\mbar_\ibar =1$ such that $q_\ibar=1$;
for all $\jbar\neq \ibar$ we have $q_\jbar=0$.
We color vertex $\ibar$ of $\Dtilbar$ in black and leave all the other vertices white.
We obtain a Kac diagram $(\Dtilbar,\qq)$ of type III.
The Kac diagrams of type III up to isomorphism are tabulated in \cite[Table 7]{OV}.
From now on, when we consider a Kac diagram of type III, we assume that it is from  \cite[Table 7]{OV}.

Let $\ibar\in \Dtilbar$ be the black vertex. We see from \eqref{e:aut-1-marking} that
\begin{equation}\label{e:Dynkin-twisted}
\alphabar_{\ibar}(t)=-1,\quad \alphabar_{\jbar}(t)=1\text{ for }1\le\jbar\le\nprime,\  \jbar\neq \ibar.
\end{equation}
Clearly, if $\ibar=0$, then $\alphabar_{\jbar}(t)=1$ for all $\jbar\neq 0$, hence $t=1$.
We use the English (not Russian) version of \cite{OV}.
In the English version of \cite[Table 7]{OV} the vertices of Kac diagrams are numbered,
and one can easily see from the table that when $\ibar\neq 0$,
the restricted root $\alphabar_{\ibar}$ is  long, hence
$\alphabar_{\ibar}$ is the restriction to $T^\ad_\C$ of
some {\em $\tau$-fixed} root $\alpha_i\in D^\tau$.
The element $t\in (T^\ad)^\tau(\R)_2$ is determined by \eqref{e:Dynkin-twisted}.
For any $\alpha_j\in \Pi\smallsetminus\Pi^\tau$,
let $\alphabar_{\jbar}$ denote the restriction of $\alpha_j$  to $(T^\ad)^\tau_\C$,
then the restricted root $\alphabar_\jbar$ is  short, hence   $\jbar\neq \ibar$
and therefore, $\alpha_j(t)=\alphabar_{\jbar}(t)=1$.
We see that $t\in T^\ad(D^\tau)(\R)_2$ and
\begin{equation}\label{e:Dynkin-non-twisted}
\alpha_{i}(t)=-1,\quad \alpha_{j}(t)=1\text{ for all }j\in D^\tau,\  j\neq i.
\end{equation}

Thus we can compute the Galois cohomology of $_z G$ as follows.
We take the Kac diagram of $_z G$ from \cite[Table 7]{OV} (the English version).
We remove vertex 0 and all vertices corresponding to the short roots.
What remains is a simply-laced colored Dynkin diagram $(D^\tau,\ttt)$
with one black vertex or without black vertices;
this is a colored Dynkin diagram for $_t G(D^\tau)$, where $_{t\tau} G\simeq\hs_z G$.
The map $\ttt\colon D^\tau\to\{0,1\}$ is the restriction of the map
$\qq\colon \Dtilbar\to\Z_{\ge0}$ to the subset $D^\tau$ of $\Dtilbar$.
If $\ttt=0$ (no black vertices), then the moves $\T_j$ of the Reeder puzzle for $(D^\tau,\ttt)$
are given by formula \eqref{non-twisted-simply-laced} for $D^\tau$.
If $\ttt\neq 0$, i.e, there is one black vertex $i$ in $D^\tau$,
then  the moves $\T_j$ for $j\in D^\tau,\ j\neq i$ are given by formula \eqref{non-twisted-simply-laced},
while the move $\T_i$ is given by formula \eqref{twisted-simply-laced-new}.
In these formulas the sum is taken over the neighbors $k$  {\em lying in $D^\tau$}.
By solving the Reeder puzzle for $(D^\tau,\ttt)$, we compute
\[\Or(D^\tau,\ttt)\cong H^1(\R,\hs_t G(D^\tau))\cong H^1(\R,\hs_{t\tau} G)\simeq H^1(\R,\hs_z G).\]

\begin{example}
These are the Kac diagram for the group $_{t\tau} G=EIV$
of type  $\EE_6$ (see Subsection \ref{subsec:EIV} below)
and the twisting diagram  for $_t G(D^\tau)=G(D^\tau)$ with trivial twisting,
i.e., the uncolored Dynkin diagram $\AA_2$:
\begin{equation*}
\mxymatrix
{ \bcb{0} \rline & \bc{1} \rline & \bc{2}  &  \ar@{=>}[l] \bc{3} \rline &\bc{4} &&&
 \bc{3} \rline & \bc{4} &&&
 {\phantom{\hbox{AAAAA}}}
}
\end{equation*}
\end{example}

\begin{example}
These are the Kac diagram for the group $_{t\tau} G=EI$
of type  $\EE_6$ (see Subsection \ref{subsec:EI} below),
 the twisting diagram for $_t G(D^\tau)$, and the augmented diagram for $_t G(D^\tau)$:
\begin{equation*}
\mxymatrix
{ \bc{0} \rline & \bc{1} \rline & \bc{2} & \ar@{=>}[l] \bc{3} \rline &\bcb{4} &&&
 \bc{3} \rline & \bcb{4} &&&
\bc{3} \rline  & \bc{4} \rline & \boxone
}
\end{equation*}
\end{example}


\section{Orbits: definitions and terminology}
\label{sec:term}

Starting with the next section, we solve the Reeder puzzles case by case,
i.e., describe  the sets of equivalence classes $\Cl(D,\tau,\ttt)$.
Proposition \ref{c:restriction} reduces the case of an outer form of a compact group
to the case of an inner form of another compact group.
In the case of an inner form we determine the set $\Or(D,\ttt)$ of
the orbits of the group $W$ generated by the moves $\T_i$
(i.e., reflections $r_{\alpha_i}$) acting on the set  $L(\hs_\ttt D)$.
Here $L(\hs_\ttt D)$ is the set of  labelings $\aa=(a_1,\dots,a_n)$
corresponding to a twisting diagram  $\hs_\ttt D$ with vertices $i=1,\dots n$,
where $a_i\in \Z/2\Z$ and each $i$ corresponds to the simple root $\alpha_i$.
We number the vertices of $D$ as in Onishchik and Vinberg \cite[Table 1]{OV}.

By {\em (connected) components} of a labeling of a Dynkin diagram we mean the connected
components of the graph obtained by removing the vertices with zeros and the corresponding edges.
For example, the following labeling of $\AA_9$ has $3$ connected components:
\[
\sxymatrix{ 1 \rline & 1  \rline & 0  \rline & 1  \rline & 0  \rline & 0  \rline & 1  \rline & 1 \rline & 1}
\qquad\longmapsto\qquad
\sxymatrix{ 1 \rline & 1         &           & 1         &                    & 1  \rline & 1 \rline & 1.}
\]
For some diagrams $D$ the number of components of a labeling is an invariant of the action of $W$.
For some others, the parity of the number of components is an invariant.

By a {\em fixed  labeling} we mean a fixed point of the action of $W$, that is, a labeling
 which is fixed under all moves $\T_i$\,.
For example, for the action of Lemma \ref{prop:non-twisted} on $\AA_5$, the labelings
\[
\sxymatrix{ 0 \rline & 0  \rline & 0  \rline & 0  \rline & 0} \quad \text{and}\quad
\sxymatrix{ 1 \rline & 0  \rline & 1  \rline & 0  \rline & 1}
\]
are fixed.

We say that two vertices $i,j$ of a Dynkin diagram $D$  are {\em neighbors} if they are connected by an
edge (single or multiple).
We say that $i$ is a vertex {\em of degree $d$} if it has exactly $d$ neighbors. We are especially
interested in vertices of degree 3. The Dynkin diagrams $\DD_n$ $(n\ge4)$, $\EE_6$, $\EE_7$ and $\EE_8$
have vertices of degree 3. Now let $D$ be a Dynkin diagram with a vertex $i$ of degree 3, and let
$\aa$ be a labeling of $D$ that looks near $i$ like
\begin{eqnarray}
\sxymatrix{\dots \rline  &1 \rline & 1 \rline \dline & 1\rline &\dots
                               \\& & 1 & }
\end{eqnarray}
The move $\T_i$ of Lemma  \ref{prop:non-twisted} splits the  component of $i$
to three  components (because $D$ has no cycles):
\begin{eqnarray}
\sxymatrix{\dots \rline  &1 \rline & 0 \rline \dline & 1\rline &\dots
                               \\& & 1 & }
\end{eqnarray}
and therefore increases the number of components by 2. We call this process {\em splitting at $i$}.
The reverse process is called {\em unsplitting}.

Let $G$ be a group with twisting  diagram $_\ttt D$ and the set of labelings $L(\hs_\ttt D)$.
We denote by $\Or(\hs_\ttt D)$ the set of orbits in $L(\hs_\ttt D)$ under the action of the Weyl group,
i.e., the set of equivalence classes of labelings with respect to the moves.
We denote the number of orbits by $\Orbs{\hs_\ttt D}$.

In Sections \ref{sec:An} -- \ref{sec:G2} below we describe the pointed sets
$\Cl(D,\tau,\ttt)$ for simply connected groups of types $\AA_n$ -- $\GG_2$.
Since these sections may be regarded as parts of a table, most proofs are omitted.
In Section \ref{sec:An}  we introduce notation which will be used in subsequent sections.

\section{Groups of type $\AA_n$}
\label{sec:An}

\subsection{The compact group $\SU(n+1)$  of type  $\AA_n^{(0)}$}
\label{sect:An}\label{subsec:An}

Here $n\ge 1$.
The Dynkin diagram  is
\begin{equation*}
\sxymatrix{ \bc{1} \rline & \cdots \rline & \bc{n} }
\end{equation*}
The Weyl group acts by the moves that are described in Lemma \ref{prop:non-twisted}.
We denote the compact form of the complex group of type $\AA_n$ by $\AA_n^{(0)}$.
The superscript 0 shows that the group is compact and the diagram is uncolored.

\begin{lemma} \label{lem:basic-An}
For $\AA_n^{(0)}$:
\begin{enumerate}
\item[(a)] The moves do not change the number of components.
\item[(b)] Every component can be reduced to length 1, e.g.
\[ 0 \ll 1 \ll 1 \ll 1 \ll 0\quad \mapsto\quad 0 \ll 0 \ll 1 \ll 0 \ll 0 \ . \]
\item[(c)] Components may be pushed so that the space between components is of length 1, e.g.
\[ 1 \ll 0 \ll 0 \ll 1 \ll 0\quad \mapsto\quad  1 \ll 0 \ll 1 \ll 0 \ll 0 \ .\]
\end{enumerate}
\end{lemma}

\begin{notation} \label{def:xi-form}
By $\xi_r^n$ (or just $\xi_r$) we mean the
labeling of $\AA_n^{(0)}$ of the form
\begin{equation*}
\xi_r\quad = \quad \sxymatrix{ 1 \rline & 0 \rline & \overset{1-0}{\cdots\cdots} & \lline
\underset{r}{1} \rline & 0 \rline & \cdots }
\end{equation*}
which has $r$ components  packed maximally to the left, namely,
\[ (\xi^n_r)_i =
\begin{cases}
1 & \mbox{ if } i=1,3,\dots,2r-1 \ , \\
0 &  \mbox{ otherwise}\ .
\end{cases}
\]
By $\eta_r^n$ (or just $\eta_r$) we mean the labeling of $\AA_n^{(0)}$
which has $r$ components packed maximally to the right, namely
\[ (\eta^n_r)_i = \begin{cases} 1 & \mbox{ if } i=n,n-2,\dots,n-2(r-1) \ , \\ 0 &  \mbox{ otherwise}\ .
\end{cases}
\]
\end{notation}

\begin{example}\
 $\xi_3^7 = \  1 \ll 0 \ll 1 \ll 0 \ll 1 \ll 0 \ll 0 $,\quad\  $\eta_2^7= \ 0 \ll 0 \ll 0 \ll 0 \ll 1 \ll 0 \ll 1$.
\end{example}

\begin{lemma} \label{cor:basic-An}
Two labelings of $\AA_n^{(0)}$ are equivalent if and only if they have the same number of components.
In particular, any labeling of $\AA_n^{(0)}$ with $r$ components is equivalent to $\xi_r^n$ and to $\eta_r^n$.
\end{lemma}

Thus, in $\AA_n^{(0)}$ the number of components is an invariant which fully characterizes orbits.

\begin{corollary}\label{cor:An-reps}
{\emm The orbit of zero} in $L(\AA_n^{(0)})$ consists of one labeling  $\xi_0$.
As {\emm representatives of orbits}  we can take $\xi_0,\xi_1,\dots,\xi_r$, where $r=\lceil
n/2\rceil$. We have
\begin{equation} \label{eq:An-num-of-orbits}
\Orbs{\AA_n^{(0)}}=r+1=\lceil n/2\rceil+1 = \begin{cases} k+1 & \mbox{ if }n=2k \ , \\ k+2 &
\mbox{ if }n=2k+1 \ . \end{cases}
\end{equation}
\end{corollary}

\subsection{The  group $\SU(m,\,n+1-m)$ \ $(\,1 \le m \le\lceil n/2\rceil\,)$ with twisting diagram  $\AA_n^{(m)}$}
\label{sect:AnTwistm}\label{subsec:An^m}

The group $G$ is the special unitary group $\SU(m,n+1-m)$ of the diagonal Hermitian form
with $m$ times $-1$ and $n+1-m$ times $+1$ on the diagonal.
Our results are valid for all $1\le m\le n$, though $\SU(m,n+1-m)\simeq \SU(n+1-m,m)$
and therefore, it suffices to consider only the case  $1 \le m \le\lceil n/2\rceil$.
 The Kac diagram   of $G$ is
\begin{equation*}
\sxymatrix{
& & \ar@{-}[lld] \bcb{0} \ar@{-}[rrd] & &  &
\\
\bc{1}  \rline & \cdots  \rline & \bcb{m} \rline  & \cdots  \rline & \bc{n}
}
\end{equation*}
 see \cite[Table 7]{OV}.
We obtain the {\em twisting diagram} $\AA_n^{(m)}$ by removing vertex 0 from the Kac diagram:
\begin{equation*}
\sxymatrix{
\bc{1}  \rline & \cdots  \rline & \bcb{m} \rline  & \cdots  \rline & \bc{n}
}
\end{equation*}
The superscript in $\AA_n^{(m)}$ refers to  the twisting (black) vertex $m$ in the twisting diagram,
with respect to the numbering of Onishchik and Vinberg \cite{OV}.
We construct the augmented diagram
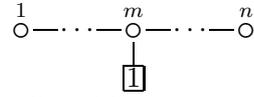
\begin{equation*}
\sxymatrix{
\bc{1}  \rline & \cdots  \rline & \bc{m} \rline \dline  & \cdots  \rline & \bc{n} \\
 & &\boxone & &
}
\end{equation*}
as in Construction \ref{const:augmented-diag},
i.e. by formally adding a new vertex with constant label 1, called boxed 1,
connected to the twisting vertex by a simple edge.
This retains the set of labelings and the set of orbits.

\begin{notation}\label{n:l,r}
Let $\aa=(a_i)\in L(\AA_n^{(m)})$.
We  have a schematic diagram:
\begin{equation} \label{schematic}
\sxymatrix{ \mathrm{LHS} \rline &a_m  \rline \dline & \mathrm{RHS} \\ & \boxone & }
\end{equation}
where LHS denotes the left-hand side and RHS denotes the right-hand side.
We denote by $l(\aa)$ the number of components of $\aa$ in LHS (to the left of the twisting vertex $m$),
and by $r(\aa)$ the number of components of $\aa$ in RHS (to the right of the twisting vertex $m$),
in both cases not taking into account the component of the boxed 1.
\end{notation}

\begin{remark}
For $\AA_n^{(m)}$:
\begin{enumerate}
\item[(i)] Any labeling $\aa$ is equivalent to a labeling $\aa'$ with $a'_m=0$.

\item[(ii)] For the schematic diagram \eqref{schematic}, if $l(\aa)\ge 1$ and $r(\aa)\ge 1$,
then the rightmost component in LHS and the leftmost component in RHS
 can be made to cancel each other out by unsplitting at vertex $m$
(which is a vertex of degree 3 for $1 < m < n$).
In other words, if $l(\aa),r(\aa) \ge 1$,  then  $\aa$ is equivalent to some labeling $\aa'$ with
$l(\aa')=l(\aa)-1$ and $r(\aa')=r(\aa)-1$.

\item[(iii)] A component cannot pass  from one hand side to the opposite hand side.
 In other words, if $\aa\sim\aa'$, then $r(\aa)-l(\aa)=r(\aa')-l(\aa')$.
\end{enumerate}
\end{remark}

\begin{proposition}\label{prop:An^(m)-invariant}
Two labelings $\aa,\aa'\in L(\AA_n^{(m)})$ are equivalent if and only if
 is $r(\aa)-l(\aa)=r(\aa')-l(\aa')$, and this invariant $r(\aa)-l(\aa)$ can take values
between $-\lceil ({m-1})/{2}\rceil$ and $\lceil ({n-m})/{2}\rceil$.
\end{proposition}

\begin{notation}\label{not:pq}
Let $p,q$ be integers, $0 \le p \le  \lceil ({m-1})/{2}\rceil$ and $0 \le q \le \lceil ({n-m})/{2}\rceil$.
 We write
\begin{equation}\label{eq:(p|q)}
(p|q) \ := \quad \sxymatrix{ \eta_p^{m-1} \rline & 0 \rline \dline & \xi_q^{n-m} \\ & \boxone & }
\end{equation}
We have $l(p|q)=p$ and $r(p|q)=q$.
\end{notation}

\begin{corollary}\label{cor:An^(m)-reps}
For $\AA_n^{(m)}$:
\begin{enumerate}
\item[(i)]
{\emm The orbit of zero} is the set of the labelings $\aa$ such that $l(\aa)=r(\aa)$.

\item[(ii)]
As {\emm representatives of orbits} we can take
\[
(0|0)\ = \quad \sxymatrix{ \eta_0 \rline & 0 \rline \dline & \xi_0 \\ & \boxone & } ,
\qquad (p|0)\ = \quad \sxymatrix{ {\eta_p} \rline & 0 \rline \dline & {\xi_0} \\ & \boxone & },
 \qquad (0|q)\  = \quad \sxymatrix{ {\eta_0} \rline & 0 \rline \dline & {\xi_q} \\ &
\boxone & }
\]
with $1 \le p \le  \lceil ({m-1})/{2}\rceil$ and $1 \le q \le \lceil ({n-m})/{2}\rceil$.

\item[(iii)]
The number of orbits is
\begin{eqnarray} \label{eq:An^m-num-of-orbits}
 \Orbs{\AA_n^{(m)}}  & = &  \left\lceil ({m-1})/{2}\right\rceil + 1 + \left\lceil
({n-m})/{2}\right\rceil \\ & = & \begin{cases}
k+1 & \mbox{ if }n=2k \ , \\
k+1 & \mbox{ if }n=2k+1 \mbox{ and } m \mbox{ is odd}, \\
k+2 & \mbox{ if }n=2k+1 \mbox{ and } m \mbox{ is even}.
\end{cases} \nonumber
\end{eqnarray}
\end{enumerate}
\end{corollary}

\subsection{Outer forms of $\SU(n+1)$} \label{ssec:An-outer}
Here $\tau$ is the nontrivial involutive automorphism of the Dynkin diagram $D=\AA_n$, where $n\ge 2$.

Case $G=\SL(n+1)$, $n=2k$. Then $D^\tau=\emptyset$, and it is well known that $H^1(\R,\SL(n+1))=1$
(this follows, for example, from \cite[III.1.1, Proposition 1]{Serre}).

Case $G=\SL(n+1)$, $n=2k+1$. Then $\#D^\tau=1$, $D^\tau=\bcb{}$\,, and again it is well known that
$H^1(\R ,\SL(n+1))=1$.

Case $G=\SL(k+1,\H)$, where $\H$ denotes the Hamilton quaternions.
Then $n=2k+1$, $\#D^\tau=1$, $D^\tau=\bc{}$, $\Orbs{\AA_1}=2$.
The orbit of zero in $L(\AA_1)$ consists of 0.
{\emm The class of zero} in $L(D)^\tau$ consists of the labelings
whose restriction to $D^\tau$ is 0, namely with $a_{k+1}=0$.
The other class consists of the labelings with $a_{k+1}=1$.

\section{Groups of type $\BB_n$}
\label{sec:Bn}

\subsection{The compact group $\Spin(2n+1)$ of type $\BB_n^{(0)}$} \label{subsec:Bn}

The Dynkin diagram is
\begin{equation*}
\sxymatrix{ \bc{1} \rline & \cdots & \lline \bc{n-1} \ar@{=>}[r] & \bc{n} }\, ,
\end{equation*}
where $n\ge 2$.

We write a labeling $\bb\in L(\BB_n^{(0)})$ as $\bb = (\aa \RRR \vk)$, where $\aa\in L(\AA_{n-1}^{(0)})$ and
$\vk \in \{0,1\}$.

Note  that the labeling
\begin{equation*}
\ell_1^{(0)} = \  (\xi_0 \RRR 1) \  = \quad (0 \ll ... \ll 0 \RRR 1)
\end{equation*}
is a fixed labeling.
We denote by $[\ell_1^{(0)}] \in \Or(\BB_n^{(0)})$ the orbit  of $\ell_1^{(0)}$ (consisting of one labeling),
and also, by slight abuse of notation, the subset
$\{\,[\ell_1^{(0)}]\,\}\subset \Or(\BB_n^{(0)})$ consisting of this orbit.

We also note that if $\aa\in L(\AA_{n-1}^{(0)})$, $\aa\neq 0$,
then $(\aa\RRR 1)$ is equivalent to $(\aa\RRR 0)$ in $L(\BB_n^{(0)})$.

\begin{proposition}
The map $\varphi \colon L(\AA_{n-1}^{(0)}) \to L(\BB_n^{(0)}) $ defined by
\[
 \aa \,{\longmapsto}\, (\aa \RRR 0) \]
 induces a bijection  $\varphi_* \colon \Or(\AA_{n-1}^{(0)}) \isoto
\Or(\BB_n^{(0)}) \smallsetminus [\ell_1^{(0)}]$ on the orbits.
\end{proposition}

\begin{corollary} \label{cor:Bn-reps} For $\BB_n^{(0)}$:
\begin{enumerate}
\item[(i)]
{\emm The orbit of zero} consists of the fixed labeling  $\xi_0\RRR 0$.
\item[(ii)]
As {\emm representatives of  orbits} we can take
$$\xi_0\RRR 1, \quad \xi_0 \RRR 0, \quad
\xi_1 \RRR 0, \quad \xi_2 \RRR 0 \ , \quad ... \ ,\quad  \xi_r \RRR 0
$$
with $r = \lceil ({n-1})/{2} \rceil$.
\item[(iii)]
\begin{equation*}
\Orbs{\BB_n^{(0)}} = \Orbs{ \AA_{n-1}^{(0)} } + 1 = \begin{cases} k+2 & \mbox{ if } n=2k \ , \\ k+2 & \mbox{ if }
n=2k+1 \ . \end{cases}
\end{equation*}
\end{enumerate}
\end{corollary}

\subsection{The group $\Spin(2m,\,2n+1-2m)$ \  $(\,1 \le m < n\,)$  with twisting diagram   $\BB_n^{(m)}$}
\label{subsec:Bn^m}
The group $G$ is the universal covering $\Spin(2m,\,2n+1-2m)$
of the special orthogonal group $\SO(2m,\,2n+1-2m)$
of the diagonal quadratic form
with $2m$ times $-1$ and $2n+1-2m$ times $+1$ on the diagonal.
The  twisting diagram  and the augmented diagram are:
\begin{equation*}
\sxymatrix{
 \bc{1} \rline & \cdots & \lline \bcb{m} \rline & \cdots & \lline \bc{n-1} \ar@{=>}[r] & \bc{n}
& \qquad & \bc{1}  \rline & \bc{2} \rline & \cdots & \lline \bc{m} \dline \rline & \cdots &
\lline \bc{n-1} \ar@{=>}[r] & \bc{n} \\
  &  & &  & \qquad & & & &  & & \boxone  &  & & }
\end{equation*}
(see \cite[Table 7]{OV} and Construction \ref{const:augmented-diag}).

Note that if $m$ is even, the labeling
\[ \ell_1^{(m)} \ = \quad
\sxymatrix{ 1 \rline & 0 \rline & \cdots & \lline 1 \rline & 0 \dline \rline & \cdots & \lline  0
\ar@{=>}[r] & 1 \\ & & & & \boxone & & }
\]
is a fixed labeling.
Note also that if $\bb=(\aa\RRR 1)\in L(\BB_n^{(m)})$ and $\bb\neq\ell_1^{(m)}$, then $\bb\sim (\aa\RRR 0)$.

\begin{proposition} \label{lem:Bn^m-bijection}
The map $\varphi\colon L(\AA_{n-1}^{(m)})\to  L(\BB_n^{(m)})$ defined by
$$
\aa\longmapsto (\sxymatrix{\aa \arr &0})
$$
induces an injection
$$
\Or(\AA_{n-1}^{(m)})\to  \Or(\BB_n^{(m)})
$$
which is bijective when $m$ is odd, and whose image is
$\Or(\BB_n^{(m)})\smallsetminus[\ell_1^{(m)}]$
when $m$ is even.
\end{proposition}

We write
\begin{equation}\label{eq:(>)}
(p|q\!\gr\!\vk):=\,\aa\RRR\vk,\quad\text{where}\quad\aa=(p|q)\in L(\AA_{n-1}^{(m)}),\ \vk \in \{0,1\}.
\end{equation}
It follows from  Proposition \ref{lem:Bn^m-bijection} that as a set of representatives for the
orbits in $L(\BB_n^{(m)})$  we can take the labelings of the form $\sxymatrix{\aa \arr &0}$, where $\aa$ runs
over the set of representatives of orbits in $L(\AA_{n-1}^{(m)})$ from Corollary \ref{cor:An^(m)-reps},
and when $m$ is even we should add the fixed labeling $\ell_1^{(m)}\in L(\BB_n^{(m)})$. Explicitly, we obtain:

\begin{corollary} \label{prop:Bn^(m)-reps}
{\emm  The orbit of zero} in $L(\BB_n^{(m)})$ is the set of the labelings
$\bb=(\sxymatrix{\aa \arr &\vk})$ with  $r(\aa)=l(\aa)$.
As {\emm representatives of orbits} in $L(\BB_n^{(m)})$ we can take
$(p|0\gr 0)$ for $0\le p\le \lceil({m-1})/{2}\rceil$,\ \ $(0|q\gr 0)$ for $0<q\le \lceil({n-1-m})/{2}\rceil$,
and $\ell_1^{(m)}$ when $m$ is even.
\end{corollary}

\begin{corollary} \label{cor:Bn^(m)-number}
We have:
\[
\Orbs{\BB_n^{(m)}} = \begin{cases} \Orbs{\AA_{n-1}^{(m)}} & \mbox{ if }m\mbox{ is odd,} \\
\Orbs{\AA_{n-1}^{(m)}} + 1 & \mbox{ if }m\mbox{ is even.} \end{cases}
\]
Using Corollary \ref{cor:An^(m)-reps}(iii), we obtain
\begin{equation*} \label{eq:Bn^m-number-of-orbits}
\Orbs{ \BB_n^{(m)} } = \begin{cases} k & \mbox{ if } n=2k \mbox{ and } m \mbox{ is odd,} \\
 k+2 & \mbox{ if } n=2k \mbox{ and } m \mbox{ is even,} \\
 k+1 & \mbox{ if } n=2k+1 \mbox{ and } m \mbox{ is odd,} \\
  k+2 & \mbox{ if } n=2k+1 \mbox{ and } m \mbox{ is even.}
  \end{cases}
\end{equation*}
\end{corollary}

\subsection{The group $\Spin(2n,1)$   with twisting diagram   $\BB_n^{(n)}$} \label{subsec:Bn^n}
The  twisting diagram  and the augmented diagram are:
\begin{equation*}
\sxymatrix{ \bc{1} \rline & \cdots & \lline \bc{n-1} \ar@{=>}[r] & \bcb{n} }
\qquad\qquad
\sxymatrix{ \bc{1} \rline & \cdots & \lline \bc{n-1} \ar@{=>}[r] & \bc{n} \rline &
\boxone }
\end{equation*}
(see \cite[Table 7]{OV} and Construction \ref{const:augmented-diag}).

If $n=2k$, we have a fixed labeling in $L(\BB_n^{(n)})$
\begin{equation*}
\ell_1^{(n)} \ =\quad  \sxymatrix{ 1 \rline & 0 \rline & 1 \rline & \overset{0-1}{\cdots} & \lline 0
\rline & 1 \arr & 1 \rline & \boxone } \quad  = \quad \xi_k \RRR 1\ll\boe\quad  \in L(\BB_n^{(n)}) \ .
\end{equation*}

\begin{proposition}
Define a map $\varphi\colon L(\AA_{n-1}^{(0)}) \to L(\BB_n^{(n)})$  by
\[ \varphi(\aa) = (\, \aa \RRR 0 \ll \boe \, ) \, , \]
then the induced map
$$
\varphi_*\colon \Or(\AA_{n-1}^{(0)})\to \Or(\BB_n^{(n)})
$$
is injective.
If $n$ is odd, then $\varphi_*$ is bijective;
 if $n$ is even, the image of $\varphi_*$ is
$\Or(\BB_n^{(n)})\smallsetminus[\ell_1^{(n)}]$.
\end{proposition}

\begin{proposition}
{\emm The orbit of zero} in $L(\BB_n^{(n)})$ consists of two labelings:
\[
\sxymatrix{ 0 \rline & \cdots & \lline 0 \ar@{=>}[r] & 0 \rline & \boxone }
\text{ \quad and \quad }
\sxymatrix{ 0 \rline & \cdots & \lline 0 \ar@{=>}[r] & 1 \rline & \boxone } \ .
\]
\end{proposition}

\begin{corollary}\label{cor:Bnn-reps}
As {\emm representatives of orbits} in $ L(\BB_n^{(n)})$ we can take
\[\sxymatrix{ \xi_0 \arr & 0  \rline & \boxone } \ ,\qquad
\sxymatrix{ \xi_1 \arr & 0  \rline & \boxone } \ ,\qquad \dots\ ,\qquad
\sxymatrix{ \xi_r \arr & 0  \rline & \boxone }
\]
where $r = \lceil ({n-1})/{2} \rceil$,  together with  $\ell_1^{(n)}$ when $n$ is even.
\end{corollary}

\begin{corollary}\label{Bnn-orbits}
\begin{equation*}
\Orbs{ \BB_n^{(n)} } = \begin{cases} \Orbs{ \AA_{n-1}^{(0)} } +1=k+2 &
\mbox{ if }\  n=2k \ , \\ \Orbs{ \AA_{n-1}^{(0)} }=k+1 & \mbox{ if }\
n=2k+1 \ .
\end{cases}
\end{equation*}
\end{corollary}

\section{Groups of type $\CC_n$}
\label{sec:Cn}

\subsection{The compact group $\Sp(n)$ with  diagram $\CC_n^{(0)}$} \label{subsec:Cn}

The group $G$ is  the compact ``quaternionic'' group $\Sp(n)$
of type $\CC_n$ ($n\ge 3$) with Dynkin diagram
\begin{equation*}
\sxymatrix{ \bc{1} \rline & \cdots & \lline \bc{n-1} & \ar@{=>}[l] \bc{n} }.
\end{equation*}

\begin{construction}\label{con:Cn}
Let $L(\AA_{n-1}^{(0)}) \sqcup L(\AA_{n-1}^{(n-1)})$ denote the disjoint union of the sets of labelings
$L(\AA_{n-1}^{(0)})$ and $L(\AA_{n-1}^{(n-1)})$.
We define a map
$$
\varphi\colon L(\AA_{n-1}^{(0)}) \sqcup L(\AA_{n-1}^{(n-1)})\to L(\CC_n^{(0)})
$$
sending $\aa\in  L(\AA_{n-1}^{(0)})$ to $\aa\LLL 0$ and sending
$\aa' \! \ll \boe \ \in L(\AA_{n-1}^{(n-1)})$ to $\aa'\LLL 1$.
Clearly $\varphi$ is a bijection.
\end{construction}

Note that
for any  $\aa\in  L(\AA_{n-1}^{(0)})$ and $\aa'\! \ll \boe\ \in L(\AA_{n-1}^{(n-1)})$,
the labelings $\aa\LLL 0$ and $\aa'\LLL 1$
are not equivalent in $L(\CC_n^{(0)})$.

\begin{proposition}
The bijection $\varphi$ of Construction \ref{con:Cn} induces a bijection on orbits
$$
\varphi_*\colon \Or(\AA_{n-1}^{(0)}) \sqcup \Or(\AA_{n-1}^{(n-1)})\isoto \Or(\CC_n^{(0)}).
$$
\end{proposition}

\begin{corollary} \label{cor:Cn}
For $\CC_n^{(0)}$:
\begin{itemize}
   \item[(i)]
{\emm The orbit of zero} is just 0.
   \item[(ii)]
As {\emm  representatives for orbits} we can take
\[ \xi_0 \LLL 0  , \quad \xi_1 \LLL 0  , \quad  \cdots  , \quad \xi_r \LLL 0, \]
where $r = \lceil ({n-1})/{2} \rceil$,
and
\[ \xi_0 \LLL 1  , \quad \xi_1 \LLL 1  , \quad  \cdots  , \quad \xi_s \LLL 1, \]
where $s = \lceil {n}/{2} \rceil - 1$.
  \item[(iii)] $\Orbs{ \CC_n^{(0)} } = \Orbs{ \AA_{n-1}^{(0)} } + \Orbs{  \AA_{n-1}^{(n-1)} }= n+1.$
\end{itemize}
\end{corollary}

(Of course, it is well known that $\# H^1(\R,G)=n+1$ in this case,
this follows, for example, from \cite[III.1.1, Proposition 1]{Serre}.)

\subsection{The diagram $\AA_n^{(m,n)}$}

(We shall need this diagram in Subsection \ref{subsec:Cm,n-m}.)
Denote by $\AA_n^{(m,n)}$ the Dynkin diagram $\AA_n$ with {\em two} black vertices $m$ and $n$, where $1\le m<n$:
$$
\sxymatrix{
\bc{1}  \rline & \cdots \rline & \bc{m-1} \rline & \bcb{m}  \rline & \bc{m+1} \rline  & \cdots  \rline & \bc{n-1} \rline & \bcb{n} \,.
}
$$
We denote by $L(\AA_n^{(m,n)})$ the set of labelings $\aa=(a_i)$ of $\AA_n^{(m,n)}$.
We consider the moves $\T_i$ given by formula \eqref{non-twisted-simply-laced}
for white vertices and by formula \eqref{twisted-simply-laced-new}     for black vertices.
We construct the augmented diagram
\begin{equation*}
\sxymatrix{
\bc{1}  \rline & \cdots  \rline & \bc{m} \rline \dline  & \cdots  \rline & \bc{n} \rline & \boxone \\
 & &\boxone & &
}\ ,
\end{equation*}
by adding $\sxymatrix{\boxone}$ two times, and now the moves $\T_i$
are given by formula \eqref{non-twisted-simply-laced} for all $i=1,\dots,n$.

We consider the orbits (equivalence classes) in $L(\AA_n^{(m,n)})$. Note that when $m$ is odd,
the labeling
\begin{equation*}
\ell_1^{(m,n)}\ =\quad  \sxymatrix{ 1 \rline & 0 \rline & \stackrel{1-0}{\cdots} & \lline 1 \rline \dline &
1 \rline & \stackrel{1}{\cdots} & \lline 1\rline & \boxone \\ & & & \boxone & & & & }
\end{equation*}
is a fixed labeling.

\begin{lemma}
For $\AA_n^{(m,n)}$,
we can take the following labelings as representatives of orbits:
$(0|0)$, $(p|0)$ for $p=1,...,\lceil (m-1)/2 \rceil$ and $(0|q)$ for $q=1,...,\lceil
(n-1-m)/2 \rceil$, and when $m$ is odd, also the fixed labeling $\ell_1$.
\end{lemma}

\subsection{The group $\Sp(m,\,n-m)$ \ $(\, 1\le m\le \lfloor n/2\rfloor\,)$
with twisting diagram   $\CC_n^{(m)}$}\label{subsec:Cm,n-m}
The group $G$ is the ``quaternionic'' group $\Sp(m,\,n-m)$,
the unitary group of the diagonal quaternionic  Hermitian form with $m$ times $-1$ and $n-m$ times $+1$ on the diagonal.
 The  twisting diagram  and the augmented diagram are:
\begin{equation*}
\sxymatrix{ \bc{1} \rline & \cdots & \lline \bcb{m} \rline & \cdots & \lline
\bc{n-1} & \ar@{=>}[l] \bc{n} }
\qquad\qquad
\sxymatrix{ \bc{1} \rline & \cdots & \lline \bc{m} \dline \rline & \cdots & \lline \bc{n-1} & \ar@{=>}[l] \bc{n} \\
                          &        & \boxone                     &        &                 & }
\end{equation*}
(see \cite[Table 7]{OV}  and Construction \ref{const:augmented-diag}).

\begin{proposition}\label{prop0:Cn^(m)}
The bijection
$$
\varphi\colon L(\AA_{n-1}^{(m)}) \sqcup L(\AA_{n-1}^{(m,n-1)})\to L(\CC_n^{(m)})
$$
sending $\aa\in  L(\AA_{n-1}^{(m)})$ to $\aa\LLL 0$ and sending $\aa'\! \ll\boe \in L(\AA_{n-1}^{(m,n-1)})$ to $\aa'\LLL 1$,
induces  a bijection
$$
\varphi_*\colon \Or(\AA_{n-1}^{(m)}) \sqcup \Or(\AA_{n-1}^{(m,n-1)})\isoto \Or(\CC_n^{(m)}).
$$
\end{proposition}

Denote $(p|q\less\vk)\,=\,\aa\!\LLL\! \vk$, where $\aa=(p|q)\in  L(\AA_{n-1}^{(m)})$ and  $\vk \in \{0,1\}$.
For example, for $\CC_5^{(3)}$ we have
\[
 (1|0\less 1)\ =\quad  \sxymatrix{ 0 \rline & 1 \rline & 0 \dline \rline & 0 & \ar@{=>}[l] 1 \\ & & \boxone & & }  \ .
\]

\begin{corollary} \label{prop:Cn^(m)}
For $\CC_n^{(m)}$:
\begin{enumerate}
\item[(i)]   {\emm The orbit of zero} is
$$
\{\,(\aa\LLL 0)\ |\ \aa\in L(\AA_{n-1}^{(m)}),\, l(\aa)=r(\aa)\,\}.
$$
\item[(ii)]  As {\emm representatives of orbits} we can take $(p|0\less 0)$ with $p=0,...,\lceil ({m-1})/{2}
\rceil$, $(0|q\less 0)$ with $q=1,...,\lceil ({n-1-m})/{2} \rceil$,
$(p|0\less 1)$ with $p = 0,...,\lceil ({m-1})/{2} \rceil$,
$(0|q\less 1)$ with $q = 1,..., \lfloor ({n-1-m})/{2} \rfloor = \lceil ({n-2-m})/{2} \rceil$,
and when $m$ is odd, the fixed labeling
\[
 \ell_1^{(m,n)}\ =\quad \sxymatrix{ 1 \rline & 0  \rline & \overset{1-0}{\cdots} & \lline 1 \dline \rline &
\overset{1}{\cdots} & \lline 1 & \ar@{=>}[l] 1 \\ & & & \boxone & & }\ .
\]
\item[(iii)] $\Orbs{ \CC_n^{(m)} } =\Orbs{\AA_{n-1}^{(m)}} + \Orbs{\AA_{n-1}^{(m,n-1)}}= n+1$.
\end{enumerate}
\end{corollary}

(Of course, it is well known that $\# H^1(\R,G) = n+1$  in this case,
this follows, for example, from \cite[III.1.1, Proposition 1]{Serre}.)

\subsection{The split group $\Sp(2n,\R)$   with twisting diagram   $\CC_n^{(n)}$}

 The   twisting diagram  and the augmented diagram are
\begin{equation*}
\sxymatrix{
 \bc{1}  \rline & \cdots & \lline \bc{n-1} & \ar@{=>}[l] \bcb{n} & \qquad\qquad & \bc{1} \rline & \cdots & \lline \bc{n-1}
& \ar@{=>}[l] \bc{n} \rline & \boxone
 }
\end{equation*}
In this case there is only one orbit, $\Orbs{ \CC_n^{(n)} } = 1$
(it is well known that $H^1(\R,G)=1$ in this case,
see for example \cite[III.1.2, Proposition 3]{Serre}).

\section{Groups of type $\DD_n$}
\label{sec:Dn}

\subsection{The compact group  $\Spin(2n)$  of type  $\DD_n^{(0)}$} \label{subsec:Dn}

The group $G$ is the spin group  $\Spin(2n)$, the universal covering of the special orthogonal group $\SO(2n)$, where $n\ge 4$.
The Dynkin diagram of $G$ is
\begin{equation*}
\mxymatrix{ \bc{1} \rline & \cdots \rline & \bc{n-3} \rline & \bc{n-2} \dline \rline & \bc{n-1} \\
& & & \bcu{n} & }
\end{equation*}
This diagram has a vertex of degree 3, the vertex $n-2$.
For brevity we introduce the following notation:
if $\aa\in L(\AA_{n-2}^{(0)})$, $\aa=(a_i)_{i=1}^{n-2}$, $\vk,\lambda\in\{0,1\}$,  we write
\begin{equation}\label{eq:frac-Dn}
\aa \frac{\vk}{\lambda} \quad:= \quad
\xymatrix@1@R=15pt@C=9pt
{a_1\rline &\cdots &\ a_{n-2}\lline \dline \rline &\ \vk\ \\ & &\overset{\ }{\lambda}  }  \ .
\end{equation}

Note that for $\DD_n^{(0)}$ the labelings
$$\ell_2^{(0)}=0=\xi_0^{n-2} \dfrac{0}{0}=0...0\frac{0}{0} \quad \text{and}\quad \ell_4^{(0)}= \xi_0^{n-2} \dfrac{1}{1}= 0...0\frac{1}{1}$$
 are fixed labelings. If $n$ is even, $n=2k$, then the labelings
$$
\ell_1^{(0)}=\xi_{k-1}^{n-2}\dfrac{1}{0}= 10..10\frac{1}{0}\quad \text{and}\quad \ell_3^{(0)}=\xi_{k-1}^{n-2}\dfrac{0}{1}= 10..10\frac{0}{1}
$$
are fixed labelings.

\begin{proposition}
Define a map
\[
\varphi\colon L(\AA_{n-2}^{(0)})  \longrightarrow  L(\DD_n^{(0)}), \quad \aa  \longmapsto  \aa \frac{0}{0}\, .
\]
Then the induced map $\varphi_*\colon \Orbset(\AA_{n-2}^{(0)})  \to  \Orbset(\DD_n^{(0)})$ is injective.
If $n$ is even, $n=2k$, then the image of $\varphi_*$ is
$$\Orbset(\DD_n^{(0)}) \smallsetminus \left\{ [\ell_4^{(0)}] , \ [\ell_1^{(0)}], \ [\ell_3^{(0)}] \right\}.$$
  If $n$ is odd, $n=2k+1$, then the image of $\varphi_*$ is
$$\Orbset(\DD_n^{(0)}) \smallsetminus  [\ell_4^{(0)}].$$
\end{proposition}

\begin{corollary} \label{cor:Dn-reps}
For $\DD_n^{(0)}$:
\begin{itemize}
   \item[(i)]
{\emm The orbit of zero} is just the labeling $\ell_2^{(0)}=0$.
\item[(ii)]
{\emm Representatives of orbits} are:
\begin{itemize}
\item For $n=2k+1$  we can take the following representatives coming from $L(\AA_{n-2}^{(0)})$:
\[
\xi_0^{n-2} \frac{0}{0} = 0...0\frac{0}{0} \ , \quad \xi_1^{n-2} \frac{0}{0} = 10...0\frac{0}{0} \
, \quad ... ,
\quad \xi_{k}^{n-2}\frac{0}{0} = 101..01\frac{0}{0}
\]
and the fixed labeling $\ell_4^{(0)}$.
\item For $n=2k$ we can take the following representatives coming from $L(\AA_{n-2}^{(0)})$:
\[
\xi_0^{n-2} \frac{0}{0} = 0...0\frac{0}{0} \ , \quad \xi_1^{n-2} \frac{0}{0} = 10..0\frac{0}{0} \ ,
\quad ... , \quad \xi_{k-1}^{n-2} \frac{0}{0} = 10..10\frac{0}{0} \ ,
\]
and the fixed labelings $\ell_4^{(0)}$, $\ell_1^{(0)}$, and $\ell_3^{(0)}$.
\end{itemize}
\item[(iii)] We have
\begin{equation*}
\Orbs{\DD_n^{(0)}} = \begin{cases} \Orbs{\AA_{n-2}^{(0)}} + 3=k+3 & \mbox{ if } \ n = 2k \ , \\
\Orbs{\AA_{n-2}^{(0)}} + 1=k+2 & \mbox{ if } \ n = 2k+1 \ . \end{cases}
\end{equation*}
\end{itemize}
\end{corollary}

\begin{example}
For $\DD_5^{(0)}$ we have representatives of orbits
\[ 000\frac{0}{0} \ , \quad 100\frac{0}{0} \ , \quad 101\frac{0}{0} , \quad 000\frac{1}{1} \ . \]
For $\DD_6^{(0)}$ we have representatives of orbits
\[
0000\frac{0}{0} \ , \quad 1000\frac{0}{0} \ , \quad 1010\frac{0}{0} \ , \quad 1010\frac{1}{0} \ ,
\quad 1010\frac{0}{1} \ , \quad 0000\frac{1}{1} \ .
\]
\end{example}

\subsection{The group $\Spin(2m,\,2n-2m)$ \ $(\,1 \le m \le \lfloor n/2 \rfloor\,)$  with twisting diagram   $\DD_n^{(m)}$}
\label{subsec:Dn^(m)}
The group $G$ is $\Spin(2m,\,2n-2m)$, the universal covering of the special orthogonal group $\SO(2n,\,2n-2m)$
of the diagonal quadratic form with $2m$ times $-1$ and $2n-2m$ times $+1$ on the diagonal.
The  twisting diagram  and the augmented diagram are:
\begin{equation*}
\mxymatrix{ \bc{1} \rline  & \cdots & \lline \bcb{m} \rline & \cdots & \lline
\bc{n-2} \dline \rline & \bc{n-1} &\qquad\qquad& \bc{1} \rline  & \cdots & \lline
\bc{m} \dline \rline & \cdots & \lline
\bc{n-2} \dline \rline & \bc{n-1} \\
   &  &   &  & \bcu{n} & &\qquad\qquad&  &    & \boxone  &  & \bcu{n} &  }
\end{equation*}
(see \cite[Table 7]{OV} and Construction \ref{const:augmented-diag}).

\begin{remark}\label{rem:when-they-occur}
For $\DD_n^{(m)}$:
\begin{enumerate}
\item[(a)]  When $m$ is even, we have  fixed labelings
\[ \
\ell_2^{(m)}\  =\quad  \sxymatrix{ 1 \rline & 0 \rline & \overset{1-0}{\cdots} & \lline 1 \rline & 0
\dline \rline & 0 \rline & \cdots & \lline 0 \dline \rline & 0
\\ & & & & \boxone & & & 0 & }
\]
and
\[ \
\ell_4^{(m)}\  =\quad \sxymatrix{ 1 \rline & 0 \rline & \overset{1-0}{\cdots} & \lline 1 \rline & 0
\dline \rline & 0 \rline & \cdots & \lline 0 \dline \rline & 1
\\ & & & & \boxone & & & 1 & }  \ .
\]

\item[(b)] When $n-m$ is even, we have  fixed labelings
\[  \ell_1^{(m)}\ = \quad
\sxymatrix{ \xi_0 \rline & 0 \dline \rline & 1 \rline & 0 \rline & \overset{1-0}{\cdots} & \lline 1
\rline & 0 \dline \rline & 1 \\ & \boxone & & & & & 0 & }
\]
and
\[ \ell_3^{(m)}\  =\quad
\sxymatrix{ \xi_0 \rline & 0 \dline \rline & 1 \rline & 0 \rline & \overset{1-0}{\cdots} & \lline 1
\rline & 0 \dline \rline & 0 \\ & \boxone & & & & & 1 & }  \ .
\]
\end{enumerate}
(Cases (a) and (b) can occur together.)
\end{remark}

Note that $[\ell_2^{(m)}]$ is in the image of the map $\varphi_*$ of Theorem \ref{lem:Dn(m)} below,
while  $[\ell_1^{(m)}], [\ell_3^{(m)}]$ and $[\ell_4^{(m)}]$ are not.

\begin{theorem}\label{lem:Dn(m)}
Consider the  map $\varphi \colon L(\AA_{n-2}^{(m)}) \to L(\DD_n^{(m)})$ defined by
$\aa \mapsto \aa \frac{0}{0}$\,.
Then the induced map on orbits $\varphi_* \colon \Or(\AA_{n-2}^{(m)}) \to \Or(\DD_n^{(m)})$ is injective,
and its image is the whole set $\Or(\DD_n^{(m)})$
except for the fixed labelings $\ell_1^{(m)}$, $\ell_3^{(m)}$, and $\ell_4^{(m)}$ when they occur;
see Remark \ref{rem:when-they-occur}.
\end{theorem}

\begin{proof}
We prove the injectivity.
Let $\dd=\aa\frac{\vk}{\lambda}\in L(\DD_n^{(m)})$, where $\aa\in L(\AA_{n-2}^{(m)})$\,.
Set
$$
\delta(\dd)=(\vk+\lambda\ {\rm mod}\ 2)(1-d_{n-2}) +r(\aa)-l(\aa),
$$
where $\vk+\lambda\ {\rm mod}\ 2\in \{0,1\}\subset \Z$, $d_{n-2}=a_{n-2} \in \{0,1\}\subset \Z$.
It is easy to check that $\delta(\dd)$ does not change under the moves in $L(\DD_n^{(m)})$.
Clearly we have  $\delta(\aa\frac{0}{0})=r(\aa)-l(\aa)$.
Now if $\aa,\aa'\in L(\AA_{n-2}^{(m)})$ and $\aa\not\sim \aa'$ in $L(\AA_n^{(m)})$,
then by Proposition \ref{prop:An^(m)-invariant} $r(\aa)-l(\aa)\neq r(\aa')-l(\aa')$,
hence $\delta(\aa\frac{0}{0})\neq \delta(\aa'\frac{0}{0})$,
and therefore, $(\aa\frac{0}{0})\not\sim (\aa'\frac{0}{0})$ in $L(\DD_n^{(m)})$.

We prove the assertion about the image. There are two cases: (1) $n-m$ is odd, and (2) $n-m$ is even.

Case (1): $n-m$ is odd. Let $\dd\in L(\DD_n^{(m)})$.
We prove that either $\dd\sim(\dots\frac{0}{0})$ or $\dd=\ell_4^{(m)}$.
Up to equivalence, we may assume that
\begin{equation}\label{eq:schematic-Dn(m)}
\dd=\aa\frac{\vk}{\lambda}=\ \sxymatrix{ \aa^l \rline &0  \rline \dline & \aa^r\dfrac{\vk}{\lambda} \\ & \boxone & }\ ,
\end{equation}
where $\aa^l\in L(\AA_{m-1}^{(0)})$ is the left-hand side of $\aa$,
$\aa^r\in L(\AA_{n-2-m}^{(0)})$ is the right-hand side of $\aa$,
and $\vk,\lambda\in\{0,1\}.$
If $\vk=\lambda=0$, then $\dd=\aa\frac{0}{0}$, as required.
If $\vk=1$, $\lambda=0$, then $\aa^r\vk=\aa^r 1\sim(\dots 0)$ in $L(\AA_{n-m-1}^{(0)})$, because $n-m-1$ is even.
Thus $\dd\sim(\dots\frac{0}{0})$, as required.
The case $\vk=0$, $\lambda=1$ is similar to the case $\vk=1$, $\lambda=0$.

Now assume that  $\vk=\lambda=1$. If $\aa^r\neq 0$, then $\aa^r\sim(\dots 1)$.
Thus $\dd\sim(\dots 1\frac{1}{1})\sim(\dots 1\frac{0}{0})$, as required.
If $\aa^r=0$ and either $m$ is odd or $m$ is even and $\aa^l\neq\xi_{m/2}$,
then we may assume that  $d_{m-1}=(\aa^l)_{m-1}=0$. Then, applying moves, we can change $d_m$ to 1,
then change $d_{m+1}$ to 1, \dots then change $d_{n-2}$ to 1, and finally we obtain
that $\dd\sim(\dots 1\frac{1}{1})\sim (\dots 1\frac{0}{0})$, as required.
If $\aa^r=0$, $m$ is even and $\aa^l=\xi_{m/2}$, then $\dd=\ell_4^{(m)}$, which completes the proof in Case (1).

Case (2): $n-m$ is even. Let $\dd\in L(\DD_n^{(m)})$.
Up to equivalence, we may assume that $\dd$ is as in \eqref{eq:schematic-Dn(m)}.
If $\vk=\lambda=0$, we have nothing to prove.
If $\vk=\lambda=1$ and $\dd\neq\ell_4^{(m)}$, then the argument in Case (1)
shows that $\dd\sim(\dots\frac{0}{0})$, as required.
Two cases remain: $\vk=1$, $\lambda=0$, and $\vk=0$, $\lambda=1$.
They are similar; we treat only the case $\vk=1$, $\lambda=0$.

Consider $\aa\vk=\aa 1\in L(\AA_{n-1}^{(m)})$.
Using moves in $L(\AA_{n-1}^{(m)})$, we can reduce $\aa1$ to a labeling which has
either 0 components right to the vertex $m$, or 0 components left to $m$.
In the former case $\dd\sim(\dots\frac{0}{0})$, as required.
In the latter case, if $\dd$ is as in \eqref{eq:schematic-Dn(m)} and $\aa^r 1$
has less than $k:=(n-m)/2$ components, then $\aa^r 1\sim(\dots 0)$ and $\dd\sim(\dots \frac{0}{0})$, as required.
If $\aa^r 1$ has $k$ components, then $\aa^r1=\xi_k$. Since  $\aa^l=0$, we see that  $\dd=\ell_1^{(m)}$.
This completes the proof in Case (2).
\end{proof}

\begin{corollary} \label{cor:Dn^(m)-zero}
Set $A_0=\{\aa \in L(\AA_{n-2}^{(m)})\ |\ l(\aa)=r(\aa)\}$ (this is the orbit of zero in $L(\AA_{n-2}^{(m)})$).
We write $\aa=(a_i)$.
Then {\emm the orbit of zero} in  $L(\DD_n^{(m)})$ is
$$
\left\{ \aa\frac{0}{0},\, \aa\frac{1}{1}\ |\ \aa\in A_0\right\}\,\cup\,
\left\{ \aa\frac{1}{0},\, \aa\frac{0}{1}\ |\ \aa\in A_0,\, a_{n-2}=1\right\}.
$$
\end{corollary}

Set
\begin{equation}\label{eq:(p|q)frac}
(p|q)\frac{\vk}{\lambda}:=\aa\frac{\vk}{\lambda}\in L(\DD_n^{(m)}),\quad
\text{where}\quad \aa=(p|q)\in L(\AA_{n-2}^{(m)}),\ \ \vk,\lambda \in \{0,1\},
\end{equation}
see formulas \eqref{eq:(p|q)} and \eqref{eq:frac-Dn}.

\begin{corollary} \label{cor:Dn^(m)-reps}
For $L(\DD_n^{(m)})$,
as {\emm representatives of  orbits} we can take
the labelings $(p|0)\frac{0}{0}$ \ for \ $0 \le p \le \lfloor m/2 \rfloor = \lceil (m-1)/2 \rceil$,
the labelings  $(0|q)\frac{0}{0}$ \ for \ $1 \le q \le \lceil ((n-2)-m)/2 \rceil$,
and the fixed labelings $\ell_4^{(m)}$ and $\ell_1^{(m)}$, $\ell_3^{(m)}$ when they occur;
see Remark \ref{rem:when-they-occur}.
\end{corollary}

\begin{corollary} \label{prop:Dn^(m)-number}
\begin{equation*}
\Orbs{ \DD_n^{(m)} } = \begin{cases} k+2 & \mbox{ if } n=2k+1 \ , \\
 k+3 & \mbox{ if } n=2k \mbox{ and } m \mbox{ is even} , \\
 k & \mbox{ if } n=2k \mbox{ and } m \mbox{ is odd}.
\end{cases}
\end{equation*}
\end{corollary}

\subsection{The group $\Spin^*(2n)$   with twisting diagram   $\DD_n^{(n)}$}
\label{sec:Dn(n)} \label{subsec:Dn^(n)}

The group $G$ is the ``quaternionic'' spin group $\Spin^*(2n)$,
the universal covering of $\SO^*(2n)$, the special unitary group
of the diagonal quaternionic skew-Hermitian form in $n$ variables
$$
\ii x_1 \xbar _1+\dots+\ii x_n \xbar_n.
$$
The  twisting diagram  and augmented diagram are:
\begin{equation*}
\xymatrix
@1@R=1pt@C=9pt
{ \bc{1} \rline   & \cdots \rline & \bc{n-3} \rline & \bc{n-2}
\dline \rline & \bc{n-1}&& && \bc{1} \rline & \cdots \rline & \bc{n-3} \rline & \bc{n-2} \dline
\rline & \bc{n-1}
\\
  & & & \bcbu{n} &
&&&& & & & \ccc \dline &
\\
&&&&&&& &&& & \boxone &
}
\end{equation*}
(see \cite[Table 7]{OV} and Construction \ref{const:augmented-diag}).

We consider the following labelings of $\DD_n^{(n)}$:
\begin{equation*} \label{diag:Dn(n)-odd-rep}
 m_1\ =\quad \sxymatrix{ 0 \rline & \cdots & \lline 0 \rline & 0 \dline \rline & 0 \\ & & & 0 \dline  & \\ & & & \boxone & }
\qquad\text{and}\qquad
m_2\ =\quad\sxymatrix{ 1 \rline & \cdots & \lline 0 \rline & 0 \dline \rline & 0 \\ & & & 0 \dline & \\ & & & \boxone & }
\end{equation*}

\begin{proposition}
For $\DD_n^{(n)}$ there are exactly two orbits:
\begin{enumerate}
\item[1.] {\emm The orbit of zero} which consists of the labelings with odd number of components (including the boxed 1)
and we can take $m_1$ as a  representative.
\item[2.] The other orbit that consists of the labelings with even number of components (including the boxed 1)
and we can take $m_2$ as a representative.
\end{enumerate}
\end{proposition}

\subsection{The group $\Spin(2m+1,\, 2(n-m)-1)$}
\label{ssec:Dn-outer}
Here $0\le m \le \lfloor (n-1)/2\rfloor$.
The group $G$ is an outer form of the compact group $\Spin(2n)$ of type $\DD_{n}$.
Here for $n>4$ we consider  the nontrivial involutive automorphism  $\tau$ of the Dynkin diagram $\DD_n$,
while for $n=4$ $\tau$ is {\em a} nontrivial involutive automorphism of $\DD_4$.
The Kac diagram is:
\[
\sxymatrix{ \bc{0} \ar@{<=}[r] & \bc{1} \rline & \cdots & \lline \bcb{m} \rline & \cdots & \lline
\bc{n-2} \ar@{=>}[r] & \bc{n-1} }\,
\]
see \cite[Table 7]{OV}.
We erase vertex 0  and also the ``short'' vertex $n-1$ (which comes from $D\smallsetminus D^\tau$).

If $m=0$, we obtain $D^\tau=\AA_{n-2}^{(0)}$ (non-twisted). By  formula \eqref{eq:An-num-of-orbits}
$$
\Orbs{D^\tau}=\lceil (n-2)/2\rceil +1.
$$
The orbit of zero in $L(D^\tau)$ is 0.
{\emm The class of 0} in $L(D)^\tau$ consists of the labelings with zero restriction to $D^\tau$.
As {\emm representatives of equivalence classes} we can take $\xi_0$,
$\xi_1$, \dots, $\xi_r$, where $r=\lceil (n-2)/2\rceil$.
These representatives lie in $L(D^\tau)$ and hence in $L(D)^\tau$.

If $m \neq 0$, then after erasing the vertices $0$ and $n-1$ of the Kac diagram we obtain the twisted diagram
$\AA_{n-2}^{(m)}$. We add boxed 1 as a neighbor to vertex $m$ and obtain the augmented diagram
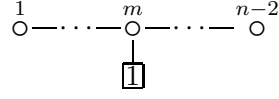
\begin{equation*}
\sxymatrix{ \bc{1}  \rline & \cdots  \rline & \bc{m} \rline \dline  & \cdots  \rline & \bc{n-2} \\
& & \boxone & & }
\end{equation*}
By  formula \eqref{eq:An^m-num-of-orbits}
\begin{eqnarray*} \label{eq:outer-num-of-orbits}
 \Orbs{D^\tau}  & = &  \left\lceil ({m-1})/{2}\right\rceil + 1 + \left\lceil
({n-2-m})/{2}\right\rceil \\ & = & \begin{cases}
k & \mbox{ if }n=2k \ , \\
k & \mbox{ if }n=2k+1 \mbox{ and } m \mbox{ is odd}, \\
k+1 & \mbox{ if }n=2k+1 \mbox{ and } m \mbox{ is even}.
\end{cases}
\end{eqnarray*}
{ The orbit of zero in $L(D,\ttt)$ consists of the labelings $\aa\in L(\AA_{n-2}^{(m)})$ such that $l(\aa)=r(\aa)$,
see Notation \ref{n:l,r}.
{\emm The class of zero in $L(D)^\tau$ } consists of the labelings $\dd$
whose restriction $\aa=\res_{D^\tau}(\dd)$ to $D^\tau$ satisfy $l(\aa)=r(\aa)$.
 As {\emm representatives of  equivalence classes} we can take
$(p|0)$ where $0\le p\le\lceil (m-1)/2\rceil$, and $(0|q)$ where $1\le q\le\lceil (n-2-m)/2\rceil $.
Again, these representatives lie in $L(D^\tau)$ and hence in $L(D)^\tau$.

\section{Groups of type $\EE_6$}
\label{sec:E6}

\subsection{The compact group of type $\EE_6^{(0)}$}

The Dynkin diagram  of $G$ is
\begin{equation*}
\mxymatrix{ \bc{1} \rline & \bc{2} \rline & \bc{3} \rline \dline & \bc{4} \rline & \bc{5}  \\
& & \bcu{6} & & }
\end{equation*}

\begin{proposition} [Reeder \cite{Reeder}]\label{prop:E6}
The diagram $\EE_6^{(0)}$ has $3$ orbits. The orbits are:
\begin{enumerate}
\item[1.] {\emm The orbit of zero} consisting of $0$, which is a fixed labeling.
\item[2.] The orbit consisting of all the labelings with $1$ or $3$  components  with representative
\[ \ell_1\ = \quad
\sxymatrix{
 1 \rline & 0 \rline & 0 \rline \dline & 0 \rline & 0  \\ &  & 0 & & }
 \]
\item[3.] The orbit consisting of all the labelings with $2$ components with representative
\[ \ell_2\ = \quad
\sxymatrix{  0 \rline & 1 \rline & 0 \rline \dline & 0 \rline & 0  \\  & & 1 & & }
\]
\end{enumerate}
\end{proposition}

\begin{remark}
The moves in $L(\EE_n)$ for  $n=6,7,8$ preserve the parity of the number of components.
\end{remark}

\begin{remark}
By \cite[Example 4.4]{Reeder} each of the graphs $D=\EE_6$ and $D=\EE_8$ is nonsingular
(namely, a certain quadratic form introduced by Reeder is nonsingular).
By \cite[Theorem 7.3 and Lemma 2.2(2)]{Reeder} in in both cases we have exactly $3$ orbits in $L(D)$: $\{0\}$,
the orbit consisting of all nonzero labelings with even number of components,
and the orbit consisting of all labelings with odd number of components.
\end{remark}

\subsection{The group $EII$    with twisting diagram   $\EE_6^{(2)}$}\label{subsec:E6(2)}

A maximal compact subgroup is of type $\AA_1  \AA_5$.
The  twisting diagram  and the augmented diagram are:
\begin{equation*}
\mxymatrix{ \bc{1} \rline & \bcb{2} \rline & \bc{3} \rline \dline & \bc{4} \rline & \bc{5}  \\ & &
\bcu{6} & &
 }
\qquad\qquad
\mxymatrix{
 \bc{1} \rline & \bc{2} \rline \dline & \bc{3} \rline \dline & \bc{4} \rline & \bc{5}  \\ & \boxone & \bcu{6} & & }
\end{equation*}
(see \cite[Table 7]{OV} and Construction \ref{const:augmented-diag}).

\begin{proposition} \label{prop:E6^(2)}
The diagram $\EE_6^{(2)}$ has $3$ orbits. The orbits are:
\begin{enumerate}
\item[1.] {\emm The orbit of zero} consisting of all the labelings with  $1$ or $3$  components (including the boxed 1).
\item[2.] The orbit consisting of the labelings with $2$ components excluding the fixed labeling $\ell'_1$,   with representative
\[ \ell_3\ =\quad \sxymatrix{
 1 \rline & 1 \rline \dline & 0 \rline \dline & 0 \rline & 0  \\ & \boxone & 1 & & } \]
\item[3.] The fixed labeling
\[ \ell'_1\ = \quad
\sxymatrix{
 1 \rline & 0 \rline \dline & 0 \rline \dline & 0 \rline & 0  \\ & \boxone & 0 & & }
 \]
\end{enumerate}
\end{proposition}

\subsection{The group $EIII$ of Hermitian type   with twisting diagram   $\EE_6^{(1)}$}\label{subsec:E6(1)}

A maximal compact subgroup of $G$ is of type $\DD_5  T^1$.
The  twisting diagram  and the augmented diagram are:
\begin{equation*}
\mxymatrix{ \bcb{1} \rline & \bc{2} \rline & \bc{3} \rline \dline & \bc{4} \rline & \bc{5}  \\ & &
\bcu{6}  & &
}
\qquad\qquad
\mxymatrix{ \boxone \rline & \bc{1} \rline & \bc{2} \rline & \bc{3} \dline \rline  & \bc{4} \rline &
\bc{5}
\\ & & & \bcu{6} & &
}
\end{equation*}
(see \cite[Table 7]{OV} and Construction \ref{const:augmented-diag}).

\begin{proposition} \label{prop:E6^(1)}
The diagram $\EE_6^{(1)}$ has $3$ orbits. The orbits are:
\begin{enumerate}
\item[1.] {\emm The orbit of zero} consisting of the labelings with  $1$ or $3$ components excluding the fixed labeling $\ell'_2$.
\item[2.] The orbit consisting of all the labelings with $2$ components with representative
\[ \ell'_3\ = \quad
\sxymatrix{ \boxone \rline & 1 \rline & 1 \rline & 0 \rline \dline & 0 \rline & 0  \\ & & & 1 & & }
\]
\item[3.] The fixed labeling
\[ \ell'_2\ = \quad
\sxymatrix{ \boxone \rline & 0 \rline & 1 \rline & 0 \rline \dline & 0 \rline & 0  \\ & & & 1 & & }
\]
\end{enumerate}
\end{proposition}

\subsection{The group $EIV$  of type  $\EE_6$}\label{subsec:EIV}

This is an outer form of the compact group of type $\EE_6$ with
maximal compact subgroup of type $\FF_4$.
The Kac diagram is
$$
\sxymatrix{ \bcb{0} \rline & \bc{}  & \bc{} \ar@{=>}[l] \rline  & \bc{}  }
$$
We denote by $\tau$ the nontrivial automorphism of the Dynkin diagram $D=\EE_6$.
We erase vertex 0 and the other ``short'' vertex of the Kac diagram.
We obtain $D^\tau=\sxymatrix{ \bc{3} \rline & \bc{6}}$ and
$\Orbs{D^\tau}=2$.
The orbit of zero in $L(D^\tau)$ consists of one labeling 0 of $D^\tau$.
{\emm The equivalence class of zero} in $L(D)^\tau$ consists of the labelings whose restriction to $D^\tau$ is 0.

\subsection{The split group $EI$ of type $\EE_6$}\label{subsec:EI}

This is an outer form of the compact group of type $\EE_6$ with maximal compact subgroup of type $\CC_4$.
The Kac diagram is
$$\sxymatrix{ \bc{0} \rline & \bc{} & \bc{} \ar@{=>}[l] \rline  & \bcb{}  }$$
We erase vertex 0 and the other ``short'' vertex of the Kac diagram.
We obtain $(D^\tau,\ttt)=\sxymatrix{ \bc{3} \rline & \bcb{6}}$.
The augmented diagram is
\[ \sxymatrix{ \bc{3} \rline & \bc{6}\rline &\boxone } \]
We have
$\Orbs{D^\tau,\ttt}=2$.
 The orbit of zero in $L(D^\tau,\ttt)$ consists of the labelings with one component (including the boxed 1).
{\emm The equivalence class of zero} in $L(D)^\tau$ consists of the labelings $\aa$ such that either $a_3=0$ or $a_6=1$.

Note that $H^1(\R,EI)$ was earlier computed by B.~Conrad \cite[Proof of Lemma 4.9]{Conrad}.

\section{Groups of type $\EE_7$}
\label{sec:E7}

\subsection{The  the compact group of type $\EE_7^{(0)}$}

The Dynkin diagram  is
\begin{equation*}
\mxymatrix{  \bc{1} \rline &  \bc{2} \rline &  \bc{3} \rline &  \bc{4}
\rline \dline &  \bc{5} \rline &  \bc{6} \\
& & & \bcu{7} & & }
\end{equation*}

\begin{proposition}[Weng \cite{Weng}] \label{prop:E7}
The diagram $\EE_7^{(0)}$ has $4$ orbits. The orbits are:
\begin{enumerate}
\item[1.] {\emm The orbit of zero}  consisting of the fixed labeling 0.
\item[2.] The fixed labeling
\[ \ell_3\ =\quad
\sxymatrix{  1 \rline & 0 \rline &  1 \rline & 0 \rline \dline &  0 \rline &  0 \\
& & & 1 & & }
\]
\item[3.] The orbit consisting of the labelings with $1$ or $3$ components
excluding the fixed labeling $\ell_3$, with representative
\[ \ell_1\ =\quad
\sxymatrix{  1 \rline & 0 \rline &  0 \rline & 0 \rline \dline &  0 \rline &  0 \\
& & & 0 & & }
\]
\item[4.] The orbit consisting of all the labelings with $2$ or $4$ components, with representative
\[ \ell_2\ =\quad
\sxymatrix{  0 \rline & 0 \rline &  1 \rline & 0 \rline \dline &  0 \rline &  0 \\
& & & 1 & & }
\]
\end{enumerate}
\end{proposition}

\begin{remark}
By a lemma of Chih-wen Weng \cite{Weng}, for any (uncolored) simply-laced tree
(not necessarily a Dynkin diagram) containing $\EE_6$ as a subgraph,
any {\em movable} (non-fixed) labeling is equivalent either to a labeling
with one component or to a labeling with two components.
\end{remark}

\subsection{The split group $EV$   with twisting diagram   $\EE_7^{(7)}$}
\label{sec:E7(7)}

A maximal compact subgroup is of type $\AA_7$.
The  twisting diagram  and the augmented diagram are:
\begin{equation*}
\xymatrix@1@R=0pt@C=9pt
{  \bc{1} \rline &  \bc{2} \rline &  \bc{3} \rline &  \bc{4}
\rline \dline &  \bc{5} \rline & \bc{6}  \\
& & & \bcbu{7} & & &}
\qquad\qquad
\sxymatrix{  \bc{1} \rline &  \bc{2} \rline &  \bc{3} \rline &  \bc{4}
\rline \dline &  \bc{5} \rline &  \bc{6} \\
& & & \ccc \dline & & \\
& & & \boxone & & }
\end{equation*}
(see \cite[Table 7]{OV} and Construction \ref{const:augmented-diag}).

\begin{proposition} \label{prop:E7^(7)}
The diagram $\EE_7^{(7)}$ has $2$ orbits. The orbits are:
\begin{enumerate}
\item[1.] {\emm The orbit of zero}
is the orbit consisting of all the labelings with $1$ or $3$ components (including the boxed 1).

\item[2.] The orbit consisting of all the labelings with $2$ or $4$ components (including the boxed 1), with representative
\[ m_3\ =\quad
\sxymatrix{  0 \rline & 1 \rline & 0 \rline & 1 \rline \dline &  0 \rline & 1 \\
& & & 0 \dline & & \\
& & & \boxone & & } \ .
\]
\end{enumerate}
\end{proposition}

Note that $H^1(\R,EV)$ was earlier computed by B.~Conrad \cite[Proof of Lemma 4.9]{Conrad}.

\subsection{The group $EVI$   with twisting diagram   $\EE_7^{(2)}$}

A maximal compact subgroup is of type  $\AA_1  \DD_6$.
The  twisting diagram  and augmented diagram are:
\begin{equation*}
\mxymatrix{  \bc{1} \rline &  \bcb{2} \rline &  \bc{3} \rline &  \bc{4}
\rline \dline &  \bc{5} \rline &  \bc{6}  \\
& & & \bcu{7} & & & }
\qquad\qquad
\mxymatrix{  \bc{1} \rline &  \bc{2} \rline \dline &  \bc{3} \rline & \bc{4}
\rline \dline &  \bc{5} \rline &  \bc{6} \\
& \boxone & & \bcu{7} & & }
\end{equation*}
(see \cite[Table 7]{OV} and Construction \ref{const:augmented-diag}).

\begin{proposition} \label{prop:E7^(2)}
The diagram $\EE_7^{(2)}$ has $4$ orbits. The orbits are:
\begin{enumerate}
\item[1.] {\emm The orbit of zero}  consisting of the labelings with $1$ or $3$ or $5$ components
(including the boxed 1), excluding the fixed labeling
$\ell'_2$ (see below).
\item[2.] The fixed labeling
\[\ell'_1\ =\quad
\sxymatrix{  1 \rline &  0 \rline \dline &  0 \rline & 0 \rline \dline &  0 \rline &  0 \\
& \boxone  & & 0  & & }
\]
\item[3.] The fixed labeling
\[ \ell'_2\ =\quad
\sxymatrix{  0 \rline &  0 \rline \dline &  1 \rline & 0 \rline \dline &  0 \rline &  0 \\
&  \boxone  & & 1  & & }
\]
\item[4.] The orbit consisting of the labelings with $2$ or $4$ components (including the boxed 1)
excluding the fixed labeling $\ell'_1$, with representative
\[ \ell'_3\ =\quad
\sxymatrix{  1 \rline &  0 \rline \dline &  1 \rline & 0 \rline \dline &  0 \rline &  0 \\
&  \boxone  & & 1  & & }
\]
\end{enumerate}
\end{proposition}

Note that $H^1(\R,EVI)$ was earlier computed by Garibaldi and Semenov \cite[Example 5.1]{GS} by a different method.

\subsection{The group $EVII$  of Hermitian type   with twisting diagram   $\EE_7^{(1)}$}

A maximal compact subgroup is of type $\EE_6  T^1$.
The  twisting diagram  and augmented diagram are:
\begin{equation*}
\mxymatrix{  \bcb{1} \rline &  \bc{2} \rline &  \bc{3} \rline &  \bc{4}
\rline \dline &  \bc{5} \rline &  \bc{6} \\
& & & \bcu{7} & & & }
\qquad\qquad
\mxymatrix{ \boxone \rline & \bc{1} \rline &  \bc{2} \rline &  \bc{3} \rline & \bc{4}
\rline \dline &  \bc{5} \rline &  \bc{6} \\
& & & & \bcu{7} & & }
\end{equation*}
(see \cite[Table 7]{OV} and Construction \ref{const:augmented-diag}).

\begin{proposition}\label{prop:E7^(1)}
The diagram $\EE_7^{(1)}$ has $2$ orbits. The orbits are:
\begin{enumerate}
\item[1.] {\emm The orbit of zero}  consisting  of all the labelings with $1$ or $3$ components (including the boxed 1).
\item[2.] The orbit consisting of all the labelings with $2$ or $4$ components (including the boxed 1), with representative
\[ m'_3\ =\quad
\sxymatrix{ \boxone \rline & 0 \rline &  1 \rline &  0 \rline & 1
\rline \dline &  0 \rline &  1 \\
& & & & 0 & & }
\]
\end{enumerate}
\end{proposition}


\section{Groups of type $\EE_8$}
\label{sec:E8}

\subsection{The compact group  of type  $\EE_8^{(0)}$}

The Dynkin diagram is
\begin{equation*}
\mxymatrix{ \bc{1} \rline & \bc{2} \rline & \bc{3} \rline & \bc{4} \rline &
\bc{5} \rline \dline & \bc{6} \rline & \bc{7}  \\ & & & & \bcu{8} & & }
\end{equation*}

\begin{proposition} [Reeder \cite{Reeder}] \label{prop:E8}
The diagram $\EE_8^{(0)}$ has $3$ orbits. The orbits are:
\begin{enumerate}
\item[1.] {\emm The orbit of zero} which contains only $0$.
\item[2.] The orbit consisting of all the labelings with odd number of components, with representative
\[\ell_3\ =\quad
\sxymatrix{ 0 \rline & 1 \rline & 0 \rline & 1 \rline & 0 \dline \rline & 0 \rline & 0 \\
& & & & 1 & & } \ .
\]
\item[3.] The orbit consisting of all the labelings with nonzero even number of components, with representative
\[\ell_2\ =\quad
\sxymatrix{ 0 \rline & 0 \rline & 0 \rline & 0 \rline & 0 \dline \rline & 1 \rline & 0 \\
& & & & 1 & & } \ .
\]
\end{enumerate}
\end{proposition}

\subsection{The split group $EVIII$   with twisting diagram   $\EE_8^{(7)}$}
\label{subsec:EE8(7)}

A maximal compact subgroup is of type $\DD_8$.
The  twisting diagram   and the augmented diagram are:
\begin{equation*}
\mxymatrix{  \bc{1} \rline & \bc{2} \rline & \bc{3} \rline & \bc{4} \rline & \bc{5}
\rline \dline & \bc{6} \rline & \bcb{7}  \\ & & & & \bcu{8} & & }
\qquad\qquad
\mxymatrix{ \bc{1} \rline &\bc{2} \rline & \bc{3} \rline & \bc{4} \rline &
\bc{5} \rline \dline & \bc{6} \rline & \bc{7} \rline & \boxone  \\  & & & &
\bcu{8} & & & }
\end{equation*}
(see \cite[Table 7]{OV} and Construction \ref{const:augmented-diag}).

\begin{proposition} \label{prop:E8^(7)}
The diagram $\EE_8^{(7)}$ has $3$ orbits. The orbits are:
\begin{enumerate}
\item[1.] {\emm The orbit of zero} consisting of the labelings with odd number of components (including the boxed 1),
excluding the fixed labeling $\ell'_2$.
\item[2.] The fixed labeling
\[ \ell'_2\ = \quad
\sxymatrix{ 0 \rline & 0 \rline & 0 \rline & 0 \rline & 0 \rline \dline & 1 \rline & 0 \rline & \boxone  \\
 & & & & 1 & & &}
\]
\item[3.] The orbit consisting of all the labelings with even number of components, with representative
\[ m_3\ = \quad
\sxymatrix{ 0 \rline & 1 \rline & 0 \rline & 1 \rline & 0 \rline \dline & 1 \rline & 0 \rline & \boxone  \\
 & & & & 0 & & &}
\]
\end{enumerate}
\end{proposition}

\subsection{The group $EIX$   with twisting diagram   $\EE_8^{(1)}$}

A maximal compact subgroup is of type $\AA_1  \EE_7$.
The  twisting diagram    and the augmented diagram are:
\begin{equation*}
\mxymatrix{   \bcb{1} \rline & \bc{2} \rline & \bc{3} \rline & \bc{4} \rline & \bc{5}
\rline \dline & \bc{6} \rline & \bc{7}  \\  & & & & \bcu{8} & & }
\qquad\qquad
\mxymatrix{ \boxone \rline & \bc{1} \rline & \bc{2} \rline & \bc{3} \rline &
\bc{4} \rline & \bc{5} \rline \dline & \bc{6} \rline & \bc{7}  \\ & & & & &
\bcu{8} & & }
\end{equation*}
(see \cite[Table 7]{OV} and Construction \ref{const:augmented-diag}).

\begin{proposition} \label{prop:E8^(1)}
The diagram $\EE_8^{(1)}$ has $3$ orbits. The orbits are:
\begin{enumerate}
\item[1.] {\emm The orbit of zero} consisting of all the labelings
with odd number of components (including the boxed 1).
\item[2.] The fixed labeling
\[ \ell'_3\ =\quad  \sxymatrix{ \boxone \rline & 0 \rline & 1 \rline & 0 \rline &
1 \rline & 0 \rline \dline & 0 \rline & 0  \\ & & & & & 1 & & }  \]
\item[3.] The orbit consisting of the labelings with even number of components,
excluding the fixed labeling $\ell'_3$, with representative
\[ m'_3 \ =\quad
\sxymatrix{ \boxone \rline & 0 \rline & 1 \rline & 0 \rline & 1 \rline & 0 \rline \dline & 1 \rline & 0  \\
 & & & & & 0 & & }  \]
\end{enumerate}
\end{proposition}

\section{Groups of type $\FF_4$}
\label{sec:F4}

\subsection{The compact group of type $\FF_4^{(0)}$}

The Dynkin diagram is
\begin{equation*}
\sxymatrix{ \bc{1} \rline & \bc{2} & \ar@{=>}[l] \bc{3} \rline & \bc{4} }\,.
\end{equation*}

\begin{proposition} \label{prop:F4}
The diagram $\FF_4^{(0)}$ has $3$ orbits. The orbits are:
\begin{enumerate}
\item[1.] {\emm The orbit of zero} which contains only \ \  $0 \ll 0 \LLL 0 \ll 0$\,.
\item[2.] The orbit
\[ \left\{\   1\ll 0 \LLL 0\ll 0\, , \quad   1 \ll 1 \LLL 0\ll 0\, , \quad   0 \ll 1 \LLL 0 \ll 0 \ \right\} \,. \]
\item[3.] The orbit that contains the rest, with representative  \ $\ell_2=\  1 \ll 0 \LLL 1 \ll 0 $\,.
\end{enumerate}
\end{proposition}

\subsection{The split group $FI$   with twisting diagram   $\FF_4^{(4)}$}
A maximal compact subgroup is of type $\CC_3 \AA_1$.
The  twisting diagram  and the augmented diagram are:
\begin{equation*}
\sxymatrix{  \bc{1} \rline & \bc{2} & \ar@{=>}[l] \bc{3} \rline & \bcb{4}  }
\qquad\qquad
\sxymatrix{  \bc{1} \rline & \bc{2} & \ar@{=>}[l] \bc{3} \rline & \bc{4} \rline& \boxone }
\end{equation*}
(see \cite[Table 7]{OV} and Construction \ref{const:augmented-diag}).

\begin{proposition} \label{prop:F4^(4)}
The diagram $\FF_4^{(4)}$ has $3$ orbits. The orbits are:
\begin{enumerate}
\item[1.] {\emm The orbit of zero} which
consists of the labelings of the form $\aa\Leftarrow \aa'$, where
$$
\aa\in L(\sxymatrix{  \bc{1} \rline & \bc{2} }), \quad \aa'\in L( \sxymatrix{ \bc{3} \rline & \bc{4} \rline &  \boxone }\  )\, ,
$$
and $\aa'$ has only one component.
\item[2.] The fixed labeling $\ell'_2\ =\quad  1 \ll 0 \LLL 1 \ll 0 \ll \boe $\ .
\item[3.] The orbit
\[
\left\{\  0 \ll 0 \LLL 1 \ll 0 \ll \boe\ , \quad  0 \ll 1 \LLL 1 \ll 0 \ll \boe\  ,
\quad  1 \ll 1 \LLL 1 \ll 0 \ll \boe \ \right\} \, .
\]
\end{enumerate}
\end{proposition}

\subsection{The group  $FII$   with twisting diagram   $\FF_4^{(1)}$}
A maximal compact subgroup is of type $\BB_4$.
The  twisting diagram   and the augmented diagram are:
\begin{equation*}
\sxymatrix{  \bcb{1} \rline & \bc{2} & \ar@{=>}[l] \bc{3} \rline & \bc{4}  }
\qquad\qquad
\sxymatrix{ \boxone \rline & \bc{1} \rline & \bc{2} & \ar@{=>}[l] \bc{3} \rline & \bc{4}  }
\end{equation*}
(see \cite[Table 7]{OV} and Construction \ref{const:augmented-diag}).

\begin{proposition} \label{prop:F4^(1)}
The diagram $\FF_4^{(1)}$ has $3$ orbits. The orbits are:
\begin{enumerate}
\item[1.] {\emm The orbit of zero} consisting of
\[ \left\{\  \boe\ll 0 \ll 0 \LLL 0 \ll0 \, , \quad  \boe\ll 1 \ll 0 \LLL 0 \ll 0 \, , \quad
\boe\ll 1 \ll1 \LLL0 \ll 0 \ \right\} \, . \]
\item[2.] The fixed labeling \ $\ell'_1\ =\quad   \boe\ll 0 \ll 1 \LLL 0 \ll 0 $\,.
\item[3.] The orbit that contains the rest, with representative \  $\ell'_3=\ \ \boe\ll 1 \ll 1 \LLL 1 \ll 0 $\,.
\end{enumerate}
\end{proposition}

\section{Groups of type $\GG_2$}
\label{sec:G2}

\subsection{The compact group of type $\GG_2^{(0)}$}

The Dynkin diagram is
\begin{equation*}
\sxymatrix{ \bc{1} & \ar@3{->}[l] \bc{2} }\, .
\end{equation*}
The description of orbits is  similar to the case  $\AA_2^{(0)}$, because $3 \equiv 1 \pmod{2}$.
We have $\Orbs{\GG_2^{(0)}}=2$.
 The two orbits are
\[ \{\, 0 \ll 0\, \} \quad \mbox{ and } \quad \{\, 1 \ll 0\,,\quad 1 \ll 1\, , \quad 0 \ll 1\, \} \, . \]

\subsection{The split group   with twisting diagram   $\GG_2^{(2)}$}
A maximal compact subgroup is of type $\AA_1  \AA_1$.
The  twisting diagram  and the augmented diagram are:
\begin{equation*}
\sxymatrix{ \bc{1} & \ar@3{->}[l] \bcb{2}  }
\qquad\qquad
\sxymatrix{ \bc{1} & \ar@3{->}[l] \bc{2} \rline & \boxone }
\end{equation*}
(see \cite[Table 7]{OV} and Construction \ref{const:augmented-diag}).
The description of orbits is  similar to the case  $\AA_2^{(2)}$.
We have $\Orbs{\GG_2^{(2)}}=2$.
The two orbits are
\[ \{\, 0 \ll 0 \ll \boe\ , \quad 0 \ll 1 \ll \boe\ , \quad 1 \ll 1 \ll \boe\   \} \quad
\mbox{ and } \quad \{\, 1 \ll 0 \ll \boe\ \} \,. \]


\section{Connected components in real homogeneous spaces}
\label{sec:examples}

Let $G$ be a  simply connected absolutely simple algebraic group over $\R$.
Let $H\subset G$ be a  simply connected semisimple $\R$-subgroup.
Set $X=G/H$.
In this section we describe our method of calculation
of the number of connected components $\#\pi_0(X(\R))$,
and give examples.

\subsection{Triple $(D,\tau,\ttt)$}
\label{ss:triple}
Let $G$ be a simply connected absolutely simple $\R$-group.

If $G$ is an {\em outer} form of a compact group $G_0$,
we can write $G=\hs_{t\hs \tau}G_0$ as in Section \ref{sec:outer},
where $\tau$ is an automorphism of order 2 of the Dynkin diagram $D$ of $G_\C$.
The element $t\in T^\ad(\R)_2$ defines a coloring $\ttt$ of $D^\tau$,
and we may assume that the coloring comes from a Kac diagram.
We obtain a triple $(D,\tau,\ttt)$.

If $G$ is an {\em inner} form of a compact group $G_0$,
we can write $G=\hs_{t}G_0$ as in Section \ref{sec:inner},
and the element $t\in T^\ad(\R)_2$ defines a coloring $\ttt$ of $D$.
In this case we set $\tau=1$, then again $\ttt$ is a coloring of $D^\tau$,
and again we may assume that the coloring comes from a Kac diagram.
Again we obtain a triple $(D,\tau,\ttt)$.

In both cases  we have the bijection \eqref{e:general-bijection}
$\Cl(D,\tau,\ttt)\isoto H^1(\R,G)$.

\subsection{Describing the connected components}

Let $H$ be  a   simply connected semisimple $\R$-subgroup
of a simply connected absolutely simple $\R$-group $G$.
We do not assume that $H$ is simple and that $H$ and $G$ are inner forms of compact groups.

Let $H=H_1\times\dots\times H_r$ be the decomposition of $H$ into the product of simple $\R$-groups.
We may and shall assume that each $H_i$ is absolutely simple.
Let $T_H$ be a  fundamental  torus of $H$, i.e., a maximal torus containing a maximal compact torus.
Then $T_H=\prod_i T_{i}$ where each $T_{i}\subset H_i$ is a fundamental  torus of $H_i$.
We present $H_i$ as a twisted form of a compact group as in Subsection \ref{ss:triple}
and obtain a triple $(D_i,\tau_i,\ttt_i)$,
where $D_i$ is the Dynkin diagram of $H_i$, $\tau_i$ is an automorphism of $D_i$ with $\tau_i^2=1$,
and $\ttt_i$ is a coloring of $D_i^{\tau_i}$.
Then we have an isomorphism $L(D_i)^{\tau_i}\isoto T_i(\R)_2$.
We set $D_H=\sqcup_i D_i$ (disjoint union), $\tau_H=\prod_i \tau_i\in\Aut(D)$ (direct product of automorphisms),
$L(D_H)=\bigoplus_i L(D_i)$, then $L(D_H)^{\tau_H}=\bigoplus_i L(D_i)^{\tau_i}$,
and we have an isomorphism $L(D_H)^{\tau_H}\isoto T_H(\R)_2$.
We have a coloring $\ttt_H$ of $D_H^{\tau_H}$: a vertex $v\in D_i\subset D_H$
is black in $D_H$ if and only if it is black in $D_i$.
We write also $L(D_H,\tau_H,\ttt_H)$ for $L(D_H)$.
We define $\Cl(D_H,\tau_H,\ttt_H)$ to be $\prod_i\Cl(D_i,\tau_i,\ttt_i)$,
then we have  a bijection $\Cl(D_H,\tau_H,\ttt_H)\isoto H^1(\R,H)$.
Using results of Sections \ref{sec:An}--\ref{sec:G2}, for each $i$
we find a set of representatives $\Xi_i\subset L(D_i,\tau_i,\ttt_i)^{\tau_i}$
of all equivalence classes in $\Cl(D_i,\tau_i,\ttt_i)$.
We set $\Xi=\prod_i\Xi_i\subset L(D_H,\tau_H,\ttt_H)^{\tau_H}$, then $\Xi$
is a set of representatives of all equivalence classes in $\Cl(D_H,\tau_H,\ttt_H)$,
i.e., the composite map $\Xi\into L(D_H,\tau_H,\ttt_H)^{\tau_H}\to \Cl(D_H,\tau_H,\ttt_H)$ is bijective.

Let $T_G$ be a fundamental torus of $G$.
We may and shall assume that $T_H\subset T_G$.
We present $G$ as a twisted form of a compact $\R$-group,
then we have a triple $(D_G,\tau_G,\ttt_G)$.
Using results of Sections  \ref{sec:An}--\ref{sec:G2},
we compute {\em the class of zero} $[\nl]_G\subset L(D_G,\tau_G,\ttt_G)^{\tau_G}$.

The embedding $T_H(\R)_2\into T_G(\R)_2$ induces an injective homomorphism
\[\iota\colon L(D_H)^{\tau_H}\to L(D_G)^{\tau_G},\]
which can be computed explicitly.
Let $\Xi_0$ denote the preimage in $\Xi$ of $[\nl]_G\subset L(D_G,\tau_G,\ttt_G)^{\tau_G}$
under the map $\Xi\into L(D_H,\tau_H,\ttt_H)^{\tau_H}\to L(D_G,\tau_G,\ttt_G)^{\tau_G}$, see the commutative diagram:
\begin{equation*}
\xymatrix{
\Xi\ar[r]   &L(D_H,\tau_H,\ttt_H)^{\tau_H}\ar[r] \ar[d]^\iota   &\Cl(D_H,\tau_H,\ttt_H)\ar[r]^-\sim\ar[d]    &H^1(\R,H) \ar[d]  \\
            &L(D_G, \tau_G,\ttt_G)^{\tau_G}\ar[r]                &\Cl(D_G,\tau_G,\ttt_G)\ar[r]^-\sim\         &H^1(\R,G)
}
\end{equation*}
We see that $\Xi_0$ is in a bijection with $\ker\left[ H^1(\R,H)\to H^1(\R, G)\right]$, and therefore,
the cardinality  of $\Xi_0$ answers Questions \ref{q:2} and  \ref{q:1}.

\subsection{Generalities on reductive groups and Galois cohomology}

Let $G$ be a simply connected semisimple algebraic  $\R$-group, $H\subset G$ be an $\R$-subgroup.
The group $G(\R)$ of $\R$-points acts on the left on $(G/H)(\R)$.

\begin{lemma}\label{lem:orbits-components}
Any orbit of $G(\R)$ in $(G/H)(\R)$
is a connected component of $(G/H)(\R)$.
\end{lemma}

\begin{proof}
Write $X=G/H$.
Let $x\in X(\R)$, then we have a map
$$
\phi_x\colon G(\R)\to X(\R),\quad g\mapsto g\cdot x.
$$
The differential of $\phi_x$ at any point $g\in G(\R)$ is surjective,
hence by the implicit function theorem the map $\phi_x$ is open, hence
the orbits of $G(\R)$ in $X(\R)$ are open, hence they are open and closed.
Since $G$ is semisimple and simply connected,
by \cite[Corollary 4.7]{Borel-Tits}
or \cite[Proposition 7.6]{PR}
the group $G(\R)$ is connected,
hence  the orbits of $G(\R)$ in $X(\R)$ are connected,
hence they are the connected components of $X(\R)$.
\end{proof}

\begin{lemma}\label{lem:pi1}
Let $\varphi\colon S\to T$ be a homomorphism of $k$-tori
over an algebraically closed field $k$ (of arbitrary characteristic),
and let $\varphi_*\colon \X_*(S)\to \X_*(T)$ denote the induced homomorphism of the cocharacter groups.
Then
\begin{enumerate}
\item[(i)] There is a canonical isomorphism
$\Hom(\,(\X^*(\ker\varphi))_\tors,\,\Q/\Z)\isoto(\coker\varphi_*)_\tors$\,,
where by $A_\tors$ we denote the torsion subgroup of an abelian group  $A$.
\item[(ii)]$\#\,(\coker\varphi_*)_\tors=\#\, (\X^*(\ker\varphi))_\tors$.
\end{enumerate}
\end{lemma}

\begin{proof}
Set $T_1=\im\varphi$, then $T_1$ is a subtorus of $T$, and there exists a subtorus $T_2\subset T$
such that $T=T_1\times_k T_2$.
Let $\varphi_1\colon S\to T_1$ be the canonical surjective  homomorphism,
then $\coker\varphi_*=\coker\varphi_{1,*}\oplus\X_*(T_2)$, whence
\[(\coker\varphi_*)_\tors\cong(\coker\varphi_{1,*})_\tors\,.\]
Therefore, we may assume that $\varphi$ is surjective. Write $K=\ker\varphi$.
From the short exact sequence
$$
1\to K\to S\labelto{\varphi} T\to 1
$$
we obtain a short exact sequence
$$
 0\to\X^*(T)\labelto{\varphi^*}\X^*(S)\to\X^*(K)\to 0,
$$
whence, by taking $\Hom(\cdot , \Z)$, we obtain an exact sequence for the functor $\Ext_\Z$
(see e.g.  \cite[Theorem III.3.2]{ML})
$$
\Hom(\X^*(S),\Z)\labelto{\varphi_*} \Hom(\X^*(T),\Z)\to\Ext^1_\Z(\X^*(K),\Z)\to\Ext^1_\Z(\X^*(S),\Z)=0,
$$
where the last equality follows from the fact that $X^*(S)$ is a free abelian group.
We have $\Hom(\X^*(S),\Z)=X_*(S)$ and $\Hom(\X^*(T),\Z)=X_*(T)$.
For a finitely generated abelian group $A$ we have $\Ext^1_\Z(A,\Z)=\Hom(A_\tors,\Q/\Z)$
\cite[Exercise 4 in Section III.9]{ML},
whence
$$
(\coker\varphi_*)_\tors=\coker\varphi_*=\Ext^1_\Z(X^*(K),\Z)=\Hom(\X^*(K)_\tors,\Q/\Z),
$$
which proves (i), and (ii) follows immediately.
\end{proof}

\begin{corollary}\label{cor:pi1}
Let $\varphi\colon H\to G$
be a homomorphism  of  reductive $\C$-groups with finite kernel.
Let $T_H\subset H$ and $T\subset G$ be maximal tori such that $\varphi(T_H)\subset T$.
Let $\varphi_*\colon \X_*(T_H)\to \X_*(T)$ denote the induced homomorphism of the cocharacter groups.
Then $\#\ker[\varphi\colon H\to G]=\#\,(\coker\varphi_*)_\tors$\,.
\end{corollary}

\begin{proof}
Write $K=\ker\varphi=\ker[T_H\to T]$, then
 $$
\#\ker[\varphi\colon H\to G]=\#K=\#\X^*(K)=\#\,(\coker\varphi_*)_\tors\,,
$$
where the last equality follows from  Lemma \ref{lem:pi1}(ii).
\end{proof}

\begin{construction}\label{lem:sc-subgroup}
Let $G$ be a simply connected absolutely simple $\R$-group.
Let $T\subset G$ be a fundamental torus, $R=R(G_\C,T_\C)$ be the root system,
$\Pi\subset R$ be a basis, $D=D(R,\Pi)$ be the Dynkin diagram.
We may and shall assume that $G=\hs_{t\tau}G_0$, where $G_0$ is a compact group,
$\tau\in(\Aut\,D)_2$, and $t\in T^\ad(\R)_2$, see Section \ref{sec:outer}.
Let $\Pi_H\subset \Pi$ be a $\tau$-invariant subset, and let $R_H\subset R$ denote the subset
consisting of integer linear combinations of simple roots $\alpha\in\Pi_H$
(then $R_H$ is a root system with basis $\Pi_H$).
Let $H_1$ denote the algebraic subgroup of $G_\C$
generated by $T_\C$ and the unipotent ``root'' subgroups $U_\beta$ for all roots $\beta\in R_H$.
Let $H$ denote the derived subgroup of $H_1$.
Then by \cite[Proposition 12.6]{MT} $H$ is a semisimple group with root system $R_H$.
Since $G$ is simply connected, by  \cite[Proposition 12.14]{MT}  $H$ is simply connected as well.
Since the complex conjugation $\rho$ acts on $R$ by $\rho(\beta)=-\tau(\beta)$ for $\beta\in R$,
the subset $R_H$ of $R$ is $\rho$-invariant,
and hence, the subgroups $H_1$ and $H$ are defined over $\R$.
\end{construction}

\begin{lemma}\label{lem:odd}
Let
$$
1\to A\to B\labelto{\psi} C\to 1
$$
be a short exact sequence of algebraic $\R$-groups,
where $A$ is finite and central in $B$.
If the order $\# A(\C)$ of $A(\C)$ is odd, then
the induced map
$$
\psi_*\colon H^1(\R,B)\to H^1(\R,C)
$$
is bijective.
\end{lemma}

\begin{proof}
Since $A$ is central, we have a cohomology exact sequence
$$
C(\R)\to H^1(\R,A)\to H^1(\R,B)\labelto{\psi_*} H^1(\R,C)\to H^2(\R,A);
$$
see \cite[I.5.7, Proposition 43]{Serre}.
Since $\#\Gal(\C/\R)=2$ and $\# A(\C)$ is odd, by \cite[Section 6, Corollary 1 of Proposition 8]{AW}
we have $H^1(\R,A)=1$ and $H^2(\R,A)=1$.
It follows that the map $\psi_*$ is surjective and that $\ker\psi_*=1$.
We show that any fiber of $\psi_*$ contains only one element.
Indeed, let $\beta\in H^1(\R,B)$ and let $b\in Z^1(\R,B)$ be a cocycle representing $\beta$.
By \cite[I.5.5, Corollary 2 of Proposition 39]{Serre}, the fiber $\psi_*^{-1}(\psi_*(\beta))$
is in a bijection with the quotient of $H^1(\R, A)$ by an action of the group $_b C(\R)$.
Since $H^1(\R,A)=1$,  our fiber $\psi_*^{-1}(\psi_*(\beta))$ indeed contains only one element.
Thus $\psi_*$ is bijective.
\end{proof}

In Subsections \ref{ex:E7}\,--\,\ref{ex-E8}
we give  examples of calculations of $\#\pi_0(\, (G/H)(\R)\,)$ using results of Sections \ref{sec:An}--\ref{sec:G2}.

\subsection{Example with $\EE_7$}
\label{ex:E7}
Let $G=EV$, the split simply connected simple $\R$-group with compact maximal torus $T$,
of type $\EE_7^{(7)}$ with twisting diagram and augmented diagram
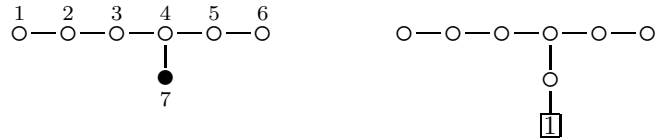
\begin{equation*}
\mxymatrix
{  \bc{1} \rline &  \bc{2} \rline &  \bc{3} \rline &  \bc{4}
\rline \dline &  \bc{5} \rline &  \bc{6} \\
& & & \bcbu{7} & & }
\qquad\qquad
\xymatrix
@1@R=9pt@C=9pt
{  \ccc \rline &  \ccc \rline &  \ccc \rline &  \ccc
\rline \dline &  \ccc \rline &  \ccc \\
& & & \ccc \dline & & \\
& & & *+[F]{\text{\small1}} & & }
\end{equation*}
see Subsection \ref{sec:E7(7)}.
Let $\Pi_G=\{ \alpha_1, \dots, \alpha_7\}$ be the simple roots (numbered as on the twisting diagram above).
We remove vertex $3$.
Set $\Pi_H=\Pi_G\smallsetminus \{\alpha_3\}$, and let $H$ be the corresponding semisimple $\R$-subgroup,
see Construction \ref{lem:sc-subgroup},
with maximal torus $T_H$ (contained in $T$)
and with twisting diagram of type $\AA_2^{(0)}\sqcup \AA_4^{(1)}$
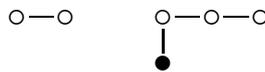
\begin{equation*}
\xymatrix@1@R=9pt@C=9pt
{  \ccc \rline &  \ccc
&\ \
 &  \ccc
\rline \dline &  \ccc \rline &  \ccc \\
& & &\Lbul & & }
\end{equation*}
and  augmented diagram:
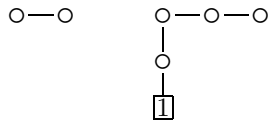
\begin{equation*}
\xymatrix@1@R=9pt@C=9pt
{  \ccc \rline &  \ccc
&\ \
&  \ccc
\rline \dline &  \ccc \rline &  \ccc \\
& & & \ccc \dline & & \\
& & & *+[F]{\text{\small1}} & & }
\end{equation*}

Then the semisimple group $H$ is simply connected, see Construction \ref{lem:sc-subgroup}.
 We have $H=H_1\times H_2$, where $H_1$ is a compact groups of type
$\AA_2^{(0)}$ and $H_2$ is a twisted (noncompact) group of type $\AA_4^{(1)}$. By Subsection \ref{sect:An},
for $H_1$ we have $\# \Or(\AA_2^{(0)})=2$ with a set of representatives
$$
\Xi_1=\ \{\ 0 \ll 0,\quad 1 \ll 0\ \}.
$$
By Subsection \ref{subsec:An^m}, for $H_2$  we have $\# \Or(\AA_4^{(1)})=3$ with a set of representatives
$$
\Xi_2=\ \{\ \boe \ll 0 \ll 0 \ll 0 \ll 0,\quad \boe\ll 0 \ll 1 \ll 0 \ll 0, \quad \boe\ll 0 \ll 1 \ll 0 \ll 1\ \}.
$$
We set $\Xi=\Xi_1\times\Xi_2\subset L(\AA_2^{(0)}\sqcup\AA_4^{(1)})=L(\AA_2^{(0)})\times L(\AA_4^{(1)})$,
hence $\#\Xi=2\cdot 3=6$.
We write down $\Xi$:
\begin{align*}
& 0 \ll 0\qquad \boe\ll 0 \ll 0 \ll 0 \ll 0\\
& 0 \ll 0\qquad \boe\ll 0 \ll 1 \ll 0 \ll 0\\
& 0 \ll 0\qquad \boe\ll 0 \ll 1 \ll 0 \ll 1\\
& 1 \ll 0\qquad \boe\ll 0 \ll 0 \ll 0 \ll 0\\
& 1 \ll 0\qquad \boe\ll 0 \ll 1 \ll 0 \ll 0\\
& 1 \ll 0\qquad \boe\ll 0 \ll 1 \ll 0 \ll 1\\
\end{align*}
We must compute the subset $\Xi_0$ of $\Xi$ consisting  of the labelings
whose images in $L(\EE_7^{(7)})$ are contained in the orbit of zero $[0]$.
The homomorphism $L(\AA_2^{(0)}\sqcup\AA_4^{(1)})\to L(\EE_7^{(7)})$ is induced
by the embedding $\AA_2^{(0)}\sqcup\AA_4^{(1)}\into \EE_7^{(7)}$.
By Subsection \ref{sec:E7(7)} the labelings of $\EE_7^{(7)}$ in the orbit of zero are those
with 1 or 3 components (including the boxed 1).
Thus  $\Xi_0$ consists of the following labelings of $\AA_2^{(0)}\times \AA_4^{(1)}$:
\begin{align*}
& 0 \ll 0\qquad \boe\ll 0 \ll 0 \ll 0 \ll 0\\
& 0 \ll 0\qquad \boe\ll 0 \ll 1 \ll 0 \ll 1\\
& 1 \ll 0\qquad \boe\ll 0 \ll 1 \ll 0 \ll 0
\end{align*}
We conclude that
$$\# \pi_0(\,(G/H)(\R)\,)\ =\ \#\ker \left[ H^1(\R,H)\to H^1(\R,G) \right] =\#\Xi_0=3.$$

Similar calculations show that if we remove vertex 2 instead of vertex 3, then $\# \pi_0(\,(G/H)(\R)\,)=2$,
and if we remove vertex 1 instead of vertex 3, then $\# \pi_0(\,(G/H)(\R)\,)=1$,
i.e. $(G/H)(\R)$ will be connected.

\subsection{Examples with $\Spin^*(2n)$}

Let $G=\Spin^*(2n)\ (n\ge 4)$, the simply connected ``quaternionic'' $\R$-group of type $\DD_n^{(n)}$
with twisting diagram and augmented diagram
\begin{equation*}
\xymatrix@1@R=0pt@C=9pt
{ \bc{1} \rline & \cdots \rline & \bc{n-3} \rline & \bc{n-2} \dline
\rline & \bc{n-1} \\
& & & \bcbu{n} & }
\qquad\qquad
\sxymatrix{ \ccc \rline & \cdots \rline & \ccc \rline & \ccc \dline
\rline & \ccc \\
& & & \ccc \dline & \\
& & & \boxone & }
\end{equation*}
see Subsection \ref{subsec:Dn^(n)}.
Let $\Pi_G=\{ \alpha_1, \dots, \alpha_n\}$ be the  simple roots (numbered as on the twisting diagram above).
We remove  vertex $n-1$.
Set $\Pi_H=\Pi_G\smallsetminus \{\alpha_{n-1}\}$,
and let $H$ be the corresponding semisimple $\R$-subgroup, see Construction \ref{lem:sc-subgroup},
with twisting diagram of type $\AA_{n-1}^{(1)}$
\begin{equation*}
\sxymatrix{ \bc{1}  \rline & \cdots  \rline & \bc{n-2} \rline &\bcb{n}}
\end{equation*}
and augmented diagram
\begin{equation*}
\sxymatrix{ \bc{1}  \rline & \cdots  \rline & \bc{n-2} \rline & \bc{n} \rline &\boxone}\ .
\end{equation*}
Then the semisimple $\R$-subgroup $H$ is simply connected, see Construction \ref{lem:sc-subgroup}.
By Subsection \ref{subsec:An^m}
we can take for representatives of orbits in $L(\AA_{n-1}^{(1)})$ the set
\[ \Xi=\left\{\eta_i\ |\ 1\le i\le\left\lceil n/2\right\rceil\right\},\]
 where $\eta_i$ denotes the labeling with  $i$ components (including the boxed 1) maximally packed to the right.
By  Subsection \ref{sec:Dn(n)},
the orbit of zero in $L(\DD_n^{(n)})$ is the set of labelings
with {\em odd} number of components (including the boxed 1).
Thus
\[ \Xi_0=\left\{\eta_i\ |\ 1\le i\le\left\lceil n/2\right\rceil,\ i \text{ is odd}\right\}.\]
We see that $\#\Xi_0$ is the number of odd numbers $i$ between 1 and $\lceil n/2\rceil$, i.e., $\#\Xi_0=\lceil n/4\rceil$.
  We conclude that
$$\# \pi_0(\,(G/H)(\R)\,)=\#\ker[H^1(\R,H)\to H^1(\R,G)]=\#\Xi_0= \left\lceil n/4\right\rceil.$$

Now, instead of removing vertex $n-1$, let us remove vertex $m$ with $1\le m\le n-2$:
\begin{equation*}
\xymatrix@1@R=0pt@C=9pt
{ \bc{1} \rline & \cdots\rline &\bc{m-1} & &\bc{m+1} \rline &\cdots\rline & \bc{n-3} \rline & \bc{n-2} \dline
\rline & \bc{n-1} \\
& & &&&&& \bcbu{n} &  &}
\end{equation*}
We obtain a subgroup $H=H_1\times H_2$, where $H_1$ is of type $\AA_{m-1}^{(0)}$ (where $m-1=0$ is possible)
and $H_2$ is of type $\DD_{n-m}^{(n-m)}$ (where $n-m=2$ is possible).
On the left of the removed vertex we can take
\[\Xi_1=\{\xi_k\ |\ 0\le k\le \lceil (m-1)/2\rceil\} \]
for representatives of orbits in $L(\AA_{m-1}^{(0)})$, see Subsection \ref{subsec:An}.
On the right of the removed vertex we can take
\[\Xi_2=\{\ell_1,\ell_2\} \]
for representatives of orbits in $L(\DD_{n-m}^{(n-m)})$
(where the labeling $\ell_1$ has one component and $\ell_2$ has two components, including the boxed 1),
see Subsection \ref{subsec:Dn^(n)}.
By Subsection \ref{subsec:Dn^(n)} applied to $G$, the orbit of zero in $L(\DD_n^{(n)})$ is the set of labelings
with {\em odd} number of components (including the boxed 1).
Now with any $\xi_k\in\Xi_1$ we associate the pair $(\xi_k,\ell)\in \Xi_1\times\Xi_2$,
where $\ell$ is either $\ell_1$ or $\ell_2$
such that the total number of components in $\xi_k$ and $\ell$ is odd.
We obtain a bijection $\Xi_1\isoto\Xi_0$.
Thus in this case
\begin{align*}
\# \pi_0(\,(G/H)(\R)\,)=\ &\#\ker[H^1(\R,H)\to H^1(\R,G)]\\
=\ &\#\Xi_0=\, \#\Xi_1=\lceil (m-1)/2\rceil+1=\lceil (m+1)/2\rceil.
\end{align*}
In particular, if $m=n-2$, we obtain $\# \pi_0(\,(G/H)(\R)\,)=\lceil (n-1)/2\rceil$.

\subsection{Example with $\Spin(2m+1,2n+1)$}

Let $G=\Spin(2m+1,2n+1)\ (m\ge 2,\ n\ge 3)$, which is an outer form of a compact group.
The Kac diagram of $G$ is
\[
\sxymatrix{\bc{0} &\bc{1}\ar@{=>}[l] \rline & \cdots &\lline\bcb{m}\rline&\cdots & \lline \bc{\ell-1} \ar@{=>}[r] & \bc{\ell} }\, ,
\]
see Subsection \ref{ssec:Dn-outer}
where $\ell=m+n$, see \cite[Table 7]{OV}.
We write $G=\hs_{t\tau}G_0$ as in Construction \ref{lem:sc-subgroup}.
We remove the $\tau$-stable vertex $m+n-k$ $(2\le k<n)$ of the Dynkin diagram
and denote the obtained semisimple $\C$-subgroup by $H$, then
by Construction \ref{lem:sc-subgroup} the subgroup $H$ is simply connected and defined over $\R$,
and we have $H=\SU(m,n-k)\times\Spin(2k+2)$.
We are interested in $\pi_0(\, (G/H)(\R)\,)$.
By Theorem \ref{cor:Theorem-3-Bo} applied to $G$ and $H$ we have a bijection
$\pi_0(\, (G'/H')(\R)\,)\isoto\pi_0(\, (G/H)(\R)\,)$,
where $G'=\SU(m,n)$ of type $\AA_{m+n-1}^{(m)}$ and $H'=H_1\times H_2$
 with $H_1=\SU(m,n-k)$ of type $\AA_{m+n-k-1}^{(m)}$ and  $H_2=\SU(k)$ of type $\AA_{k-1}^{(0)}$.
Although probably one can compute $\#\pi_0(\, (G'/H')(\R)\,)$ using real algebraic geometry,
we compute this number using Galois cohomology.
Namely, for $H_1$ of type $\AA_{m+n-k-1}^{(m)}$ we can take
\[ \Xi_1=\{(p|0)\ |\ 0\le p\le \lceil (m-1)/2\rceil\}\ \cup\
         \{(0|q)\ |\  1\le q\le \lceil (n-k-1)/2\rceil\} \]
for representatives of orbits in  $L(\AA_{m+n-k-1}^{(m)})$, see Subsection \ref{subsec:An^m}.
For $H_2$ of type $\AA_{k-1}^{(0)}$ we can take
\[\Xi_2=\{\xi_i\ |\ 0\le i\le \lceil (k-1)/2\rceil\}\]
for representatives of orbits in $L(\AA_{k-1}^{(0)})$, see Subsection \ref{subsec:An}.
For $G'$, the orbit of zero in $L(\AA_{m+n-1}^{(m)})$
is the set of labelings with the same number of components on the left and on the right of $m$,
see Subsection \ref{subsec:An^m}.
Thus
\[ \Xi_0=\{(\hs(p|0),\xi_p\hs)\in\Xi_1\times\Xi_2\}.\]
Here $0\le p\le \lceil (m-1)/2\rceil$, $0\le p\le \lceil (k-1)/2\rceil$, hence
\[\#\Xi_0=1+\min(\lceil (m-1)/2\rceil,\lceil (k-1)/2\rceil)=\min(\lceil (m+1)/2\rceil,\lceil (k+1)/2\rceil).\]
We conclude that
\[\#\pi_0(\, (G/H)(\R)\,)=\#\pi_0(\, (G'/H')(\R)\,)=\#\Xi_0= \min(\lceil (m+1)/2\rceil,\lceil (k+1)/2\rceil).\]

\subsection{Example with $\EE_8$}
\label{ex-E8}
Let $G=EVIII$,  the split form $\EE_8^{(7)}$ of $\EE_8$ with compact maximal torus $T$,
with  Kac diagram  and augmented diagram
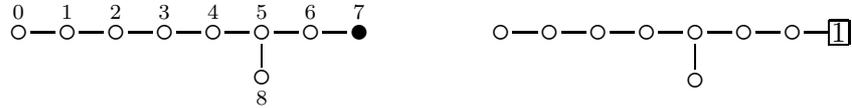
\begin{equation*}
\xymatrix@1@R=0pt@C=9pt
{\bc{0} \rline &\bc{1} \rline & \bc{2} \rline & \bc{3} \rline & \bc{4} \rline &
\bc{5} \rline \dline & \bc{6} \rline & \bcb{7}  \\ & & & & & \bcu{8} & & }
\qquad\qquad
\xymatrix@1@R=9pt@C=9pt
{ \ccc \rline & \ccc \rline & \ccc \rline & \ccc \rline &
\ccc \rline \dline & \ccc \rline & \ccc \rline & \boxone  \\  & & & &
\ccc & & & }
\end{equation*}
see Subsection \ref{subsec:EE8(7)}.
In this example, in contrast to the two previous examples, we construct an $\R$-subgroup $H$ of $G$
of the same rank, and not of smaller rank.
We remove  vertex 4 from the Kac diagram (the extended Dynkin diagram), and we do not erase vertex 0.
This means that we consider the semisimple $\C$-subgroup $H$ of $G$,
generated by $T_\C$ and the unipotent ``root" subgroups $U_\beta$ with $\beta\in R_H$,
where $R_H$ is the set of $\beta\in R$ that are integer linear combinations of the roots
$\alpha_i$, $0\le i\le 8,\ i\neq 4$, where $\alpha_1,\dots,\alpha_8$
are the simple roots and $\alpha_0$ is the lowest root.
Since $-R_H=R_H$, the $\C$-subgroup $H$ is defined over $\R$.
We obtain a maximal connected algebraic subgroup $H$ of $G$ \cite[Table 5]{OV2}
with twisting diagram
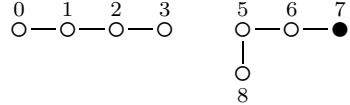
\begin{equation*}
\xymatrix@1@R=0pt@C=9pt
{\bc{0}\rline &\bc{1} \rline & \bc{2} \rline & \bc{3}  &   &
\bc{5} \rline \dline & \bc{6} \rline & \bcb{7}  \\ & & & & & \bcu{8} & & }
\end{equation*}
and augmented diagram
\begin{equation} \label{eq:Htil}
\xymatrix@1@R=9pt@C=9pt
{ \ccc \rline & \ccc \rline & \ccc \rline & \ccc  &   &
\ccc \rline \dline & \ccc \rline & \ccc \rline & \boxone
\\ & & & & & \ccc & & & }
\end{equation}

We compute the fundamental group $\pi_1(H_\C)$ of the semisimple group $H$.
Let $\Htil$ denote the universal covering of $H$.
Consider the composite morphism
$$
\varphi\colon \Htil\to H\to G,
$$
and let $\Ttil_H$ denote the maximal torus of $\Htil$ such that $\varphi(\Ttil_H)=T$.
We denote by $\varphi_*\colon\X_*(\Ttil_H)\to\X_*(T)$ the induced homomorphism of the cocharacter groups.
The cocharacter group $\X_*(\Ttil_H)$ has a basis
\begin{equation}\label{eq:basis}
\alpha_0^\vee,\ \alpha_1^\vee,\ \alpha_2^\vee,\ \alpha_3^\vee,\ \widehat{\alpha_4^\vee},
\ \alpha_5^\vee,\ \alpha_6^\vee,\ \alpha_7^\vee,\ \alpha_8^\vee,
\end{equation}
where $\widehat{\alpha_4^\vee}$ means that $\alpha_4^\vee$ is removed from the list.
The cocharacter group $\X_*(T)$ has a basis $\alpha_1^\vee, \dots, \alpha_8^\vee$,
while the subgroup $\im\varphi_*\subset\X_*(T)$
is generated by the cocharacters \eqref{eq:basis}.
There is a linear relation between $\alpha^\vee_0,\ \alpha^\vee_1,\dots,\alpha^\vee_8$,
in which the removed simple coroot $\alpha^\vee_4$ appears with coefficient 5,
while $\alpha^\vee_0$ appears with coefficient 1;
see \cite[Table 6]{OV} or \cite[Planche VII, (IV)]{Bourbaki}.
We see that $\im\varphi_*\subset \X_*(T)$ contains $\alpha_i^\vee$ for $i\neq 4$,
and it contains $5\alpha_4^\vee$, but not $\alpha_4^\vee$.
Thus $\im\varphi_*$ is a subgroup of index 5 in $\X_*(T)$.
By Corollary \ref{cor:pi1} the kernel of the canonical epimorphism $\Htil\to H$ is of order 5,
hence $\pi_1(H_\C)$ is of order 5.

Since the order 5 of $\ker\varphi$ is odd, by Lemma \ref{lem:odd}
the induced map $H^1(\R,\Htil)\to H^1(\R,H)$ is bijective,
whence
$$
\#\ker[H^1(\R,H)\to H^1(\R,G)]\ =\ \#\ker[H^1(\R,\Htil)\to H^1(\R,G)].
$$
We compute $\#\ker[H^1(\R,\Htil)\to H^1(\R,G)]$.
We have $\Htil=\Htil_1\times \Htil_2$, where $\Htil_1$ is compact of type $\AA_4^{(0)}$
and $\Htil_2$ is of type $\AA_4^{(4)}$.
By Subsection \ref{sect:An} we can take
$$
\Xi_1=\ \{\ 0 \ll 0 \ll 0 \ll 0,  \quad 1 \ll 0 \ll 0 \ll 0, \quad 1 \ll 0 \ll 1 \ll 0\ \}
$$
as a set of representatives of orbits  in $L(\AA_4^{(0)})$.
By Subsection \ref{subsec:An^m} we can take
$$
\Xi_2=\ \{\  0 \ll 0 \ll 0 \ll 0 \ll \boe\,, \quad 0 \ll 0 \ll 1 \ll 0 \ll \boe\,, \quad 1 \ll 0 \ll 1 \ll 0 \ll \boe\ \}
$$
as a set of representatives of orbits in $L(\AA_4^{(4)})$.
Set $\Xi=\Xi_1\times\Xi_2$.
We denote $\Xi_0$ the preimage in $\Xi$ of the orbit of zero in $L(\EE_8^{(7)})$.
 By Subsection \ref{subsec:EE8(7)} the orbit of zero  $[0]\subset L(\EE_8^{(7)})$
consists of the labelings with odd number of components (including the boxed 1),
excluding the fixed labeling $\ell'_2$.
The subset of $\Xi$ consisting of labelings with odd number of components has the following 5 labelings:
\begin{align*}
& 0 \ll 0 \ll 0 \ll 0  \qquad  0 \ll 0 \ll 0 \ll 0 \ll \boe\\
& 0 \ll 0 \ll 0 \ll 0  \qquad  1 \ll 0 \ll 1 \ll 0 \ll \boe\\
& 0 \ll 0 \ll 0 \ll 1  \qquad  0 \ll 0 \ll 1 \ll 0 \ll \boe\\
& 0 \ll 1 \ll 0 \ll 1  \qquad  0 \ll 0 \ll 0 \ll 0 \ll \boe\\
& 0 \ll 1 \ll 0 \ll 1  \qquad  1 \ll 0 \ll 1 \ll 0 \ll \boe
\end{align*}
and one of them  ( $0 \ll 0 \ll 0 \ll 0  \quad  1 \ll 0 \ll 1 \ll 0 \ll \boe$ ) is the preimage of $\ell'_2$.
We see that
$\Xi_0 =\ 4$.
We conclude that
\begin{align*}
\# \pi_0(\,(G/H)(\R)\,)\ =\ &\#\ker[H^1(\R,H)\to H^1(\R,G)]\\
=\ &\#\ker[H^1(\R,\Htil)\to H^1(\R,G)]=\#\Xi_0=4.
\end{align*}

\bigskip


\noindent{\sc Acknowledgements.}
The authors are very grateful to Dmitry A.~Timashev for his help in  proving Lemma \ref{prop:twisted}
and also for reading Subsection \ref{ss:Kac-outer} and correcting an inaccuracy.
We thank the anonymous referee for careful reading the paper and for his/her comments, which helped to improve the exposition.
We note that Erwann Rozier  computed in 2009 the cardinalities $\#\Or(\hs_\ttt D)$ for some colored graphs $\hs_\ttt D$
(in particular, for all Dynkin diagrams and twisting diagrams) under the guidance of the first-named author.


\end{document}